\documentclass[preprint]{elsarticle}

\pdfoutput=1

\usepackage[bookmarks=true]{hyperref}  

   \usepackage{graphicx}

\usepackage{color}
\usepackage{amssymb,amsmath}
\usepackage{amsthm}

\newtheorem{theorem}{Theorem}
\newtheorem{algorithm}{Algorithm}

\usepackage{calc}

\newcommand{\dx}{\Delta x}

\newcommand{\dt}{\Delta t}

\newcommand{\uL}{{v_L}}
\newcommand{\uR}{{v_R}}
\newcommand{\sigmaL}{{\sigma_L}}
\newcommand{\sigmaR}{{\sigma_R}}
\newcommand{\vRB}{v_b}
\newcommand{\vRBtilde}{\tilde{v}_b}

\newcommand{\cL}{c_L}
\newcommand{\rhoL}{{\rho_L}}
\newcommand{\kappaL}{{\rhoL\cL^2}}
\newcommand{\cR}{c_R}
\newcommand{\rhoR}{{\rho_R}}
\newcommand{\kappaR}{{\rhoR\cR^2}}

\newcommand{\UL}{{U_L}}
\newcommand{\UR}{{U_R}}
\newcommand{\VL}{{\partial_t{U}_L}}
\newcommand{\VR}{{\partial_t{U}_R}}

\newcommand{\dxL}{{\Delta x_L}}
\newcommand{\dxR}{{\Delta x_R}}

\newcommand{\Dp}{D_{+}}
\newcommand{\Dm}{D_{-}}

\newcommand{\erf}{\operatorname{erf}}

\newcommand{\uLE}[1]{\begin{bmatrix} v_{#1}\\ \sigma_{#1} \end{bmatrix}}
    
\newcommand{\bogus}[1]{{}}

\newcommand{\av}{\mathbf{ a}}
\newcommand{\bv}{\mathbf{ b}}

\newcommand{\ev}{\mathbf{ e}}
\newcommand{\fv}{\mathbf{ f}}
\newcommand{\gv}{\mathbf{ g}}
\newcommand{\hv}{\mathbf{ h}}
\newcommand{\iv}{\mathbf{ i}}

\newcommand{\nv}{\mathbf{ n}}

\newcommand{\rv}{\mathbf{ r}}

\newcommand{\vv}{\mathbf{ v}}
\newcommand{\wv}{\mathbf{ w}}
\newcommand{\xv}{\mathbf{ x}}
\newcommand{\yv}{\mathbf{ y}}

\newcommand{\Ev}{\mathbf{ E}}

\newcommand{\Gv}{\mathbf{ G}}

\newcommand{\Iv}{\mathbf{ I}}

\newcommand{\half}{{1\over2}}

\newcommand{\Real}{{\mathbb R}}

\newcommand{\zerov}{\mathbf{0}}

\newcommand{\tableFont}{\scriptsize}

\newcommand{\Ic}{{\mathcal I}}
\newcommand{\Rc}{{\mathcal R}}
\newcommand{\Gc}{{\mathcal G}}

\newcommand{\omegav}{\boldsymbol{\omega}}

\newcommand{\mrb}{m_{b}}
\newcommand{\xcm}{x_b}
\newcommand{\vcm}{v_b}
\newcommand{\vb}{v_b}
\newcommand{\xvcm}{\xv_b}
\newcommand{\vvcm}{\vv_b}
\newcommand{\dotxcm}{\dot{x}_b}
\newcommand{\dotrb}{\dot{r}_b}
\newcommand{\dotvcm}{\dot{v}_b}
\newcommand{\dotvb}{\dot{v}_b}
\newcommand{\dotxvcm}{\dot\xv_b}
\newcommand{\dotvvcm}{\dot\vv_b}

\newcommand{\grad}{\nabla}

\newcommand{\tableFontSize}{\small}
\newcommand{\num}[2]{#1e#2} 
\newcommand{\eem}{e^{(j)}}
\newcommand{\rateLabel}{rate}

\clearpage

\newcommand{\OmegaMatrix}{W}
\newcommand{\Force}{\mathcal{F}}
\newcommand{\Forcev}{\boldsymbol{\Force}}
\newcommand{\Forcevtilde}{\widetilde{\Forcev}}
\newcommand{\Fvtilde}{\Forcevtilde}
\newcommand{\Torque}{\mathcal{T}}
\newcommand{\Torquev}{\boldsymbol{\Torque}}
\newcommand{\Torquevtilde}{\widetilde{\Torquev}}
\newcommand{\Gvtilde}{\Torquevtilde}
\newcommand{\Ma}{A_m}

\newcommand{\mass}{{m_b}}
\newcommand{\force}{\Force}

\newcommand{\Gp}{\Gc_p}

\newcommand{\rhos}{\bar{\rho}}
\newcommand{\zs}{\bar{z}}
\newcommand{\vs}{\bar{v}}
\newcommand{\sigmas}{\bar{\sigma}}
\newcommand{\zf}{z}

\newcommand{\rb}{r_b} 

\newcommand{\amp}{{\mathcal{A}}}

\newcommand{\Energy}{\mathcal{E}}
\newcommand{\Area}{\mathcal{A}_b}
\newcommand{\width}{w_b}

\newcommand{\rhoLin}{\hat{\rho}}
\newcommand{\pLin}{\hat{p}}
\newcommand{\vLin}{\hat{v}}

\newcommand{\cLin}{\hat{c}}

\newcommand{\charSpeed}{s}
\newcommand{\charVar}{\chi}

\newcommand{\Avv}{A^{vv}}
\newcommand{\Avw}{A^{v\omega}}
\newcommand{\Awv}{A^{\omega v}}
\newcommand{\Aww}{A^{\omega\omega}}
\newcommand{\avv}{a^{vv}}
\newcommand{\avw}{a^{v\omega}}

\newcommand{\aww}{a^{\omega\omega}}

\usepackage{tikz}

\newlength{\tfwidth}
\newlength{\tfheight}
\newlength{\tfxa}
\newlength{\tfxb}
\newlength{\tfya}
\newlength{\tfyb}
\newcommand{\trimPlotWithBox}[6]{%
\setlength\fboxsep{0pt}%
\setlength\fboxrule{1.0pt}
\fbox{\includegraphics[width=#2, clip, trim=#3 #4 #5 #6]{#1}}%
}
\newcommand{\trimPlotNoBox}[6]{%
\setlength\fboxsep{1pt}
\setlength\fboxrule{0.0pt}
\fbox{\includegraphics[width=#2, clip, trim=#3 #4 #5 #6]{#1}}%
}
\newcommand{\trimPlotb}[6]{%
\setlength{\tfwidth}{19.05cm}
\setlength{\tfxa}{\tfwidth*\real{#3}}%
\setlength{\tfxb}{\tfwidth*\real{#4}}%
\setlength{\tfya}{\tfwidth*\real{#5}}%
\setlength{\tfyb}{\tfwidth*\real{#6}}%
\trimPlotWithBox{#1}{#2}{\tfxa}{\tfya}{\tfxb}{\tfyb}%
}
\newcommand{\trimPlot}[6]{%
\setlength{\tfwidth}{19.05cm}
\setlength{\tfxa}{\tfwidth*\real{#3}}%
\setlength{\tfxb}{\tfwidth*\real{#4}}%
\setlength{\tfya}{\tfwidth*\real{#5}}%
\setlength{\tfyb}{\tfwidth*\real{#6}}%
\trimPlotNoBox{#1}{#2}{\tfxa}{\tfya}{\tfxb}{\tfyb}%
}
\newcommand{\trimFigNoBox}[6]{%
\setlength\fboxsep{1pt}
\setlength\fboxrule{0.0pt}
\fbox{\includegraphics[width=#2, clip, trim=#3 #4 #5 #6]{#1}}%
}

\newcommand{\trimFig}[6]{%
\setlength{\tfwidth}{(#2+#2*\real{#3})+#2*\real{#4}}
\setlength{\tfheight}{(#2+#2*\real{#5})+#2*\real{#6}}%
\setlength{\tfxa}{\tfwidth*\real{#3}}%
\setlength{\tfxb}{\tfwidth*\real{#4}}%
\setlength{\tfya}{\tfheight*\real{#5}}%
\setlength{\tfyb}{\tfheight*\real{#6}}%
\trimFigNoBox{#1}{#2}{\tfxa}{\tfya}{\tfxb}{\tfyb}%
}
\begin{document}

\small

\begin{frontmatter}
\title{A stable FSI algorithm for light rigid bodies in compressible flow}

\author[llnl]{J. W. Banks\corref{cor1}\fnref{llnlThanks}}
\ead{banks20@llnl.gov}

\author[llnl]{W. D. Henshaw\fnref{llnlThanks}}
\ead{henshaw1@llnl.gov}

\author[llnl]{B. Sj{\"o}green\fnref{llnlThanks}}
\ead{sjogreen2@llnl.gov}

\address[llnl]{Center for Applied Scientific Computing, Lawrence Livermore National Laboratory, Livermore, CA 94551, USA}

\cortext[cor1]{Corresponding author. Mailing address: Center for Applied Scientific Computing, L-422, Lawrence Livermore National Laboratory, Livermore, CA 94551, USA. Phone: 925-423-2697. Fax: 925-424-2477. }

\fntext[llnlThanks]{This work was performed under the auspices of the U.S. Department of Energy (DOE) by
  Lawrence Livermore National Laboratory under Contract DE-AC52-07NA27344 and by 
  DOE contracts from the ASCR Applied Math Program.}

\begin{abstract}
In this article we describe a stable partitioned algorithm that overcomes the {\em added mass} instability
arising in fluid-structure interactions of light rigid bodies and inviscid compressible flow.
The new algorithm is stable even for bodies with zero mass and zero moments of inertia. 
The approach is based on a local characteristic projection of the force on the rigid body 
and is a natural extension of the recently developed algorithm for coupling compressible
flow and deformable bodies~\cite{banks11a,banks12a,sjogreen12}.
The new algorithm advances the solution in the fluid domain with a standard
upwind scheme and explicit time-stepping. 
The Newton-Euler system of ordinary differential equations governing the motion of
the rigid body is augmented by added mass correction terms. This system, which is very stiff for light bodies,
is solved with an A-stable diagonally implicit Runge-Kutta scheme. The implicit system (there is one independent system for each body)
consists of only $3d+d^2$ scalar unknowns in $d=2$ or $d=3$ space dimensions and is fast to solve.
The overall cost of the scheme is thus dominated by the
cost of the explicit fluid solver.
Normal mode analysis is used to prove the stability of the approximation for a one-dimensional
model problem and numerical computations confirm these results. 
In multiple space dimensions the approach naturally reveals the form of the added mass tensors 
in the equations governing the motion of the rigid body.
These tensors, which depend on certain
surface integrals of the fluid impedance, couple the translational and angular velocities of the body.
Numerical results in two space dimensions, based on the use of moving overlapping grids and adaptive mesh refinement,
demonstrate the behavior and efficacy of the new scheme.
These results include the simulation of the difficult problems of shock impingement on an ellipse 
and a more complex body with appendages, both with zero mass.
\end{abstract}

\begin{keyword}
fluid-structure interaction, added mass instability, moving overlapping grids, compressible fluid flow, rigid bodies
\end{keyword}

\end{frontmatter}

\section{Introduction}
\label{sec:intro}

An important class of fluid-structure interaction (FSI) problems are those
that involve the interaction of moving bodies with high-speed compressible fluids.
For example, understanding the impact of shock or detonation waves
on rigid structures and embedded rigid bodies is of great interest.
The numerical simulation of such
problems can be difficult, and many techniques have been developed to address
various facets of the problem. For a review of FSI see~\cite{bungartz06} for
example.  One particularly challenging aspect has been the presence of numerical
instabilities that can arise when simulating problems with light bodies. This
so-called {\em added-mass} instability is associated with the fact that the
reaction of a body to an applied force depends not only on the mass of the body
but also on the fluid displaced by the body through its motion. Traditional
partitioned FSI schemes do not take into account the strong coupling between the
fluid and solid and thus can exhibit an instability whereby the over-reaction of
a light solid to an applied force from the fluid leads in turn to an even larger
reaction from the fluid and so on.  Fully coupled monolithic approaches to FSI
can overcome the unstable behavior but are generally more expensive, can be
difficult to implement, and may require advanced solvers or preconditioners.
For compressible fluids the instability in partitioned algorithms can often be
suppressed by choosing a smaller time-step (as the analysis in this article
demonstrates). However, the stable time-step goes to zero as the mass of the
body goes to zero and thus alternative approaches to removing the instability
are desirable.

In a recent series of articles, we have developed a set of stable interface
approximations for partitioned solutions procedures that
couple compressible fluids and {\em deformable} bodies~\cite{banks11a,banks12a,sjogreen12}.
In~\cite{banks11a,banks12a} the interface approximation is based on a local characteristic analysis 
that results in an impedance weighted projection of the velocity and forces on the interface.
These methods ensure the stability
of the partitioned FSI scheme even for {\em light} solids.
In this article we extend these ideas to
the coupling of compressible fluids and {\em rigid bodies}.
The key idea presented in this article can be introduced by considering
the equations of motion for a rigid body (the full set of equations are presented in
detail in Section~\ref{sec:NewtonEuler})
\begin{align}
   \mrb \dotvvcm &= \Forcev , \label{eq:centerOfMassVelocityEquationIntro} \\
   A \dot \omegav &= -\omegav\times(A \omegav) + \Torquev , \label{eq:angularVelocityEquationIntro} 
\end{align}
where $\mrb$ is the mass of the body, $\vvcm(t)$ is the velocity of the center
of mass, $\omegav(t)$ the angular velocity and $A$ the moment of inertia tensor.
$\Forcev$ and $\Torquev$ are, respectively, the force and torque on the body arising from the fluid
forces on the surface of the body.  From
Equations~\eqref{eq:centerOfMassVelocityEquationIntro}-\eqref{eq:angularVelocityEquationIntro}
it would at first seem impossible to solve for $\vvcm$ and/or $\omegav$ when
$\mrb=0$ and/or $A=0$, as the equations apparently become singular.  However, from a
local characteristic analysis of the appropriate fluid-structure Riemann problem, we can
determine how $\Forcev$ and $\Torquev$ implicitly depend on the motion
of the body,
\begin{align}
    \Forcev & = -\Avv \vvcm - \Avw \omegav + \Forcevtilde , \qquad
    \Torquev  =  -\Awv \vvcm - \Aww \omegav + \Torquevtilde . \label{eq:genFormIntro}
\end{align}
The matrices $A^{ij}$ are the {\em added-mass} tensors; these are defined in
terms of certain integrals of the fluid impedance over the boundary of the rigid
body (see Section~\ref{sec:multiDimensionalInterface}). 
It is worth pointing out that the concept of added-mass has a long history in describing
the motion of embedded bodies in both compressible and incompressible flows. For the 
compressible regime the recent article~\cite{parmar08} nicely discusses the history
as well as modern developments.

Using the form of Equation (\ref{eq:genFormIntro}) as 
a starting point, we define a partitioned FSI scheme that remains stable with a large
time-step (i.e. the usual time-step restriction associated with
the fluid domain in isolation) even as $\mrb$ or $A$ go to zero, provided the added-mass tensors
satisfy certain properties. This approach relies on the use of an implicit time
stepping method for the evolution of the rigid body, 
but uses standard upwind schemes and explicit time-stepping for the fluid. 
The number of
equations in the rigid body implicit system is small ($3d+d^2$ scalar unknowns in $d=2$ or $d=3$ space dimensions)
and thus does not have any appreciable impact on the cost of the overall algorithm.
The new added-mass scheme is analyzed in detail for a one-dimensional model problem consisting
of a rigid body embedded in a fluid governed by the linearized Euler equations.
Both a first-order accurate upwind scheme and the second-order accurate Lax-Wendroff
scheme are analyzed using normal mode stability theory~\cite{gustafsson72}. When the rigid body is
integrated with an A-stable time-stepping method, the resulting partitioned FSI scheme is shown to be
stable with a large time step  even when $\mrb=0$.

The added-mass scheme is implemented in two space dimensions
using the moving overlapping grid technique described in~\cite{henshaw06}.  In
this approach, local boundary fitted curvilinear grids are used to represent the
bodies and these move through static background grids that are often chosen to
be Cartesian grids for efficiency.  Adaptive mesh refinement (AMR) is used on both curvilinear and Cartesian grids
to dynamically increase resolution locally in space and time.  We solve the compressible Euler equations with explicit time-stepping,
on possibly moving grids, in the fluid domain using a high-order extension of
Goudnov's method. The Newton-Euler equations (with added-mass corrections) are solved for the motion of the
rigid-body using an implicit Runge Kutta scheme (in contrast to the explicit time-stepping method used 
previously in~\cite{henshaw06}).

In general, the added-mass scheme proposed here could be used in conjunction with
any number of FSI approaches.  The treatment of moving geometry is a major
component for coupling fluid flow to the motion of rigid bodies and many
techniques have been considered. One class of methods relies on a fixed
underlying grid for the fluid domain and includes, embedded boundaries~\cite{cirak07}, immersed
boundaries~\cite{borazjani08,gretarsson11}, level sets~\cite{arienti03,barton11}, and fictitious domain methods~\cite{vanloon07}. 
A second class of methods uses body conforming meshes
and allows the mesh to deform in response to the motion of the body. Popular in
this class of methods are ALE~\cite{donea82,lohner99,kuhl03,ahn06}, multiblock~\cite{schafer01}, and general moving unstructured grids~\cite{tezduyar06}.

The remainder of this article is structured as follows. 
In Section~\ref{sec:multiD}, the governing equations of inviscid compressible flow for the fluid, and the 
Newton-Euler equations for rigid body motion are presented. 
Section~\ref{sec:1Dprojection} provides some motivation for, and the derivation of, 
our interface projection scheme in one dimension, showing the origin of the added-mass terms in the
equation of motion for the rigid body.
In Section~\ref{sec:analysis} this approximation is incorporated into a partitioned FSI
scheme for a one-dimensional FSI model problem. The stability of this new added-mass scheme, as well as the traditional
coupling scheme, is analyzed using normal mode theory.
Section~\ref{sec:1Dresults} provides numerical confirmation of the theoretical results for the
one-dimensional problem, demonstrating the expected convergence rates and stability properties.
Extension of the algorithm to multiple space
dimensions is presented in Section~\ref{sec:multiDimensionalInterface} showing the derivation of 
added-mass tensors. The time-stepping procedure for the overlapping grid FSI algorithm is
summarized in Section~\ref{sec:multidimensionalTimeStepping}. Results for two-dimensional
problems are presented in Section~\ref{sec:numericalResults}. 
These include (1) a smoothly receding rigid piston with known solution, (2) a smoothly
accelerated ellipse which is compared to the traditional algorithm, (3) a
shock-driven zero mass ellipse, and (4) a shock impacting a zero-mass body with a complex boundary.
The last two examples, which also demonstrate the
use of adaptive mesh refinement (AMR), are particularly challenging and interesting.
Concluding remarks are given in
Section~\ref{sec:conclusions}. In~\ref{sec:testProblem} we derive the 
exact solutions used in the numerical verification of the one-dimensiomal model problem.
Finally in~\ref{sec:addedMassMatrices} we present the form of the added mass matrices
for a number of simple shapes in two and three dimensions.

\section{Rigid bodies and compressible flow in multiple space dimensions} \label{sec:multiD}

In this section we define the governing equations for the fluid domains and the
rigid bodies. The equations are presented in three space dimensions which serves
as a general model. Simplifications to one and two space dimensions, as well as
linearization, will be performed later as appropriate.

\subsection{The Euler equations for an inviscid compressible fluid} \label{sec:NewtonEuler}

We consider the evolution of a compressible inviscid fluid with an embedded rigid body. The governing equations
for the fluid domain $\Omega \subset \Real^3$  are the compressible Euler equations 
\begin{align}
   \partial_t \wv + \grad\cdot \fv(\wv) & = 0, \qquad \xv \in \Omega, \quad t>0, 
\end{align}
where $\wv=[\rho,\rho\vv,\rho \Energy]^T$ is the vector of conserved variables (density, momentum, energy),
$\vv$ is the velocity, and
$\fv=[\rho\vv,\rho\vv\otimes\vv+p\Iv,(\rho \Energy+p)\vv]^T$ is
the flux.  The total energy is given by $\rho
\Energy=p/(\gamma-1)+\frac{1}{2}\rho\vert\vv\vert^2$ assuming an ideal gas with a
constant ratio of specific heats.

\subsection{The Newton-Euler equations for the motion of a rigid body} \label{sec:NewtonEuler}

The equations of motion for the rigid body are the Newton-Euler equations which can be written as
\begin{align}
   \dotxvcm &= \vvcm , \label{eq:centerOfMassPositionEquation} \\
   \mrb \dotvvcm &= \Forcev , \label{eq:centerOfMassVelocityEquation} \\
   A \dot \omegav &= - \OmegaMatrix A \omegav + \Torquev , \label{eq:angularVelocityEquation} \\
   \dot E &= \OmegaMatrix E .  \label{eq:axesOfInertiaEquation}
\end{align}
Here $\mrb$ is the mass of the body, $\xvcm(t)\in\Real^3$ is the position of the center of mass,
and $\vvcm(t)\in\Real^3$ is the velocity of the center of mass. 
The moment of inertia matrix $A\in\Real^{3\times 3}$ is defined by 
\begin{align*}
   A(t) &= \int_{B(t)} \rho_b(\xv)\Big[ \yv^T\yv I - \yv\yv^T \Big] \, d\xv , \quad \yv = \xv-\xvcm, 
\end{align*}
where $\rho_b(\xv)$ defines the mass density of the body and $B(t) \subset
\Real^3$ defines the region occupied by the body. The inertia matrix is
symmetric and positive semi-definite (positive definite if $\rho_b(\xv)>0$) and
can be written in terms of the orthogonal matrix $E\in\Real^{3\times 3}$, whose
columns are the principle axes of inertia, $\ev_i(t)$, and the diagonal matrix
$\Lambda$ whose diagonal entries are the moments of inertia, $I_i$, 
\begin{align*}
A &= E \Lambda E^T, \quad E =[\ev_1~ \ev_2~ \ev_3], \quad A\ev_i = I_i \ev_i, \quad \Lambda ={\rm diag}(I_1,I_2,I_3), \quad \ev_i^T\ev_j=\delta_{ij}.
\end{align*}
The angular momentum of the body is $\hv=A\omegav$ where $\omegav(t)\in\Real^3$ is the angular velocity. The matrix $\OmegaMatrix$
in~\eqref{eq:angularVelocityEquation} is the angular velocity matrix given by 
\begin{align}
  \OmegaMatrix & = \rm{Cross}(\omegav) = \begin{bmatrix}
    0 & -\omega_3 & \omega_2 \\
    \omega_3 & 0 & -\omega_1 \\
    - \omega_2 & \omega_1  & 0 
\end{bmatrix}     , \quad \text{( i.e. $\OmegaMatrix\av = \omegav\times\av$)} .   \label{eq:Wdef}
\end{align}
The total force and torque on the body are given by 
\begin{align}
   \Forcev &= \int_{\partial B} \fv_s\, ds + \fv_b, \quad \text{($\fv_s$ = surface forces, $\fv_b$= body force)}, \label{eq:bodyForce}\\
   \Torquev &= \int_{\partial B} (\xv-\xvcm)\times\fv_s\, ds + \gv_b, \quad \text{(torque)},  \label{eq:bodyTorque}
\end{align}
Given $\Forcev(t)$ and $\Torquev(t)$, along with initial conditions, $\xvcm(0)$, $\vvcm(0)$, $\omegav(0)$, and $E(0)$, 
equations~\eqref{eq:centerOfMassPositionEquation}-\eqref{eq:axesOfInertiaEquation} can be solved
to determine $\xvcm(t)$, $\vvcm(t)$, $\omegav(t)$,  and $E(t)$ as a function of time.

The motion of a point $\rv(t)$ attached to the body is given by a translation together with a rotation about the
initial center of mass,
\begin{align*}
    \rv(t) &= \xvcm(t) + R(t) (\rv(0)-\xvcm(0)) , 
\end{align*}
where $R(t)$ is the rotation matrix given by
\begin{align}
   R(t) &= E(t) E^{T}(0).    \label{eq:rotationMatrix}
\end{align}
The velocity of this point is 
\begin{align*}
    \dot\rv(t) &= \vvcm(t) +\OmegaMatrix R(t) (\rv(0)-\xvcm(0)) , \\
               &= \vvcm(t) +\OmegaMatrix (\rv(t)-\xvcm(t)) , \\
               &= \vvcm(t) +\omegav\times(\rv(t)-\xvcm(t)), 
\end{align*}
Letting $\yv = \yv(\rv) \equiv \rv(t)-\xvcm(t)$ it follows that the velocity of the point $\rv$ can be written in the form
\begin{align}
    \dot\rv(t) &= \vvcm(t) - Y \omegav, \label{eq:rDot}
\end{align}
where $Y(t)$ is the matrix 
\begin{align}
  Y & = \rm{Cross}(\yv) = \begin{bmatrix}
    0 & -y_3 & y_2 \\
    y_3 & 0 & -y_1 \\
    - y_2 & y_1  & 0 
\end{bmatrix} . \label{eq:Ymatrix}
\end{align}

\subsection{The coupling conditions for rigid bodies and inviscid compressible flow}\label{sec:coupling}

On an interface between a fluid and a solid, the normal component of the fluid velocity must match the normal component of
the solid velocity (the inviscid equations allow slip in the tangential direction).
Let $\rv=\rv(t)$ denote a point on the surface of the body $B$, and $\nv=\nv(\rv)$
the outward normal to the body, then
\begin{align}
  \nv^T\dot{\rv}(t) = \nv^T\vv(\rv(t),t).
\end{align}
In addition, the surface force per-unit-area at each point on the body is given by the local force per-unit-area
exerted by the fluid, 
\begin{align}
   \fv_s(\rv(t)) = -\nv \, p(\rv(t),t).
\end{align}

\section{A partitioned FSI algorithm for the one-dimensional Euler equations and a rigid body -- added mass terms}\label{sec:1Dprojection}

In the recent series of articles~\cite{banks11a,banks12a,sjogreen12}, a stable
interface projection scheme was developed for the problem of coupling a
compressible fluid and a deformable elastic solid of arbitrary density. 
The key result from~\cite{banks11a,banks12a} can be
distilled from the consideration of a
one-dimensional Riemann problem consisting of a linearized compressible fluid
(equations~\ref{eq:linearizedEuler}) on the right with state
$(\rho_0,v_0,\sigma_0)$, and a linear elastic solid on the left with state
$(\rhos_0,\vs_0,\sigmas_0)$. Arguments based on characteristics were used to
show that for positive times the interface values $(v_I,\sigma_I)$ are given in
terms of an impedance weighted average of the fluid and solid states,
\begin{align}
   v_I & =  { \zs \vs_0 + \zf v_0 \over \zs + \zf } + { \sigma_0 - \sigmas_0 \over \zs + \zf } \label{eq:linearizedFSRImpedanceV}
            , \\
  \sigma_I &= { \zs^{-1} \sigmas_0 + \zf^{-1} \sigma_0 \over \zs^{-1} + \zf^{-1} } 
                  +  {  v_0 - \vs_0 \over \zs^{-1} + \zf^{-1} }. \label{eq:linearizedFSRImpedanceSigma}
\end{align}
Here $\zs=\rhos c_p$ is the solid impedance based on the speed of sound, $c_p$,
for compression waves in the solid,
while $\zf=\rho c$ is the fluid impedance based on the speed of sound, $c$,
in the fluid. In~\cite{banks11a,banks12a} it was 
found that using a projection to impose (\ref{eq:linearizedFSRImpedanceV}) and (\ref{eq:linearizedFSRImpedanceSigma}) 
as interface conditions resulted in a scheme that remained stable, even in the
presence of {\em light} solids when the traditional FSI coupling scheme fails. 
See~\cite{banks11a,banks12a,sjogreen12} for further details. 

The present situation of a rigid body can be considered
through a limit process where $c_p$ becomes large compared to $c$, and the
elastic body becomes increasingly rigid. Taking the
formal limit $\zs/\zf \rightarrow \infty$ in
equations~\eqref{eq:linearizedFSRImpedanceV}-\eqref{eq:linearizedFSRImpedanceSigma},
with $\zf$ fixed, results in\footnote{This limit process could be quite complex and we are
speaking here on informal grounds for motivational purposes.}
\begin{align}
   v_I & =  \vs_0              , \label{eq:interfaceV_limit}\\
  \sigma_I &= \sigma_0 +z ( v_0 - \vs_0 ). \label{eq:interfaceSigma_limit}
\end{align}
Thus for a rigid body, the interface surface stress is equal to the stress from the
fluid plus $z$ times the difference of the fluid velocity and the velocity of the body. The dependence 
of the interface stress, $\sigma_I$,  on the velocity of the body, $\vs_0$,  has thus been exposed.

These interface conditions can be derived more directly by considering the {\em
Riemann-like} problem, shown in Fig.~\ref{fig:fluidRigidBodyCartoon}, that consists of a rigid body of mass $\mrb$ adjacent to a
compressible fluid governed by the linearized Euler equations. Using
characteristic theory, we can write an explicit equation for the motion for the
rigid body in terms of the initial conditions. This process introduces an added
mass term into the equations, and the motion of the body is seen to be well
defined even when $\mrb=0$. The equations are then written in an alternative form
as an interface projection that is localized in space and time. This form can be
used to generalize the approach to multiple dimensions.

Consider then the solution to the linearized one-dimensional Euler equations for an inviscid compressible fluid,
in the moving domain $x>\rb(t)$ as shown in Fig.~\ref{fig:fluidRigidBodyCartoon}, 
\begin{align}
&  \left\{
  \begin{aligned}
  \partial_t\rho + \vLin \partial_x\rho + \rhoLin \partial_x v &=0 \\
  \partial_t v + \vLin \partial_x v - (1/\rhoLin)\partial_x\sigma &=0 \\
  \partial_t\sigma + \vLin \partial_x\sigma - \rhoLin \cLin^2 \partial_x v &= 0
 \end{aligned} 
 \right.
, ~~\text{for $x>\rb(t)$} ,  \label{eq:linearizedEuler}  \\
&  [\rho(x,0),v(x,0),\sigma(x,0)] = [\rho_0(x), v_0(x), \sigma_0(x)] .
\end{align}
Here $\sigma=-p$ is the fluid stress. 
The equations have been linearized about the constant state $[\rhoLin, \vLin, \pLin]$. 
The linearized speed of sound is $\cLin=\sqrt{\gamma \pLin/\rhoLin}$ and the 
the initial conditions are given by $[\rho_0(x), v_0(x), \sigma_0(x)]$.
The fluid is coupled to a rigid body of mass $\mrb$ whose motion is governed by Newton's law of motion
for the velocity, $\vcm$, and the position, $\xcm$, of the center of mass, 
\begin{align}
  \mrb \dotvcm &= \sigma(\rb(t),t) \Area + f_b ,  \label{eq:NewtonEuler1d}\\
    \dotxcm &= \vcm .
\end{align}
Here $\Area$ is the cross-sectional area of the body, $f_b$ is an external body force and $\rb=x_b+\width/2$ defines the point
on the body that lies next to the fluid ($\width$ being the constant width of the body).

%
%
{
\newcommand{\fsi}{G}
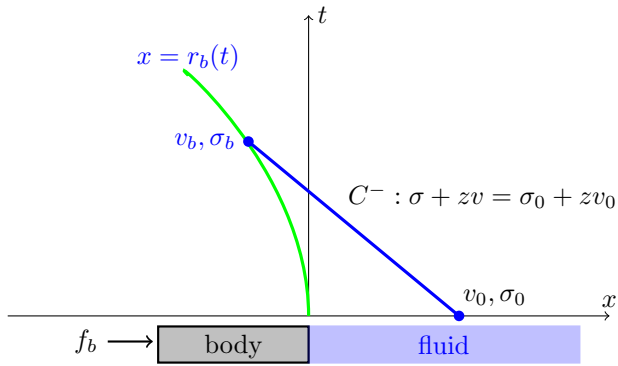
\begin{figure}[hbt]
\begin{center}
\begin{tikzpicture}[scale=4]
\useasboundingbox (-.1,0.0) rectangle (1.,1.);  
\draw[->] (-1.,0) -- (1.,0) node[anchor=south] {$x$};
\draw[->] (0,0) -- (0,1) node[anchor=west] {$t$} ;
%
\draw[very thick,green] (0,0) .. controls (0,.5) and (-.5,.9) .. (-.4,.8) node[anchor=south,yshift=-2pt,blue] {$x=\rb(t)$};
%
\draw[very thick,blue] (.5,0) -- (-.2,.58) node[anchor=east,xshift=-1pt,blue] {$v_b,\sigma_b$};
\fill[blue] (-.2,.58) circle (.5 pt);
\draw (.1,.4) node[anchor=west] {$C^-: \sigma + z v = \sigma_0+z v_0$};
\draw (.5,0.) node[anchor=west,xshift=-2pt,yshift=7pt] {$v_0,\sigma_0$}; 
\fill[blue] (.5,0.) circle (.5 pt);
\begin{scope}[yshift=-1pt]
\newcommand{\pcolour}{black}
\newcommand{\xap}{-.5}\newcommand{\xbp}{0.}
\newcommand{\yap}{-.12}\newcommand{\ybp}{0.0}
\draw[thick,\pcolour,fill=\pcolour,opacity=0.25] (\xap,\yap) rectangle (\xbp,\ybp);
\draw[thick,\pcolour] (\xap,\yap) rectangle (\xbp,\ybp);
\draw[thick,black] (-.25,-.065) node {body};
\draw[thick,->,xshift=-0.5pt] (-.65,-.05) node[anchor=east] {$f_b$} -- (\xap,-.05) ;
\newcommand{\fcolour}{blue}
\newcommand{\xaf}{0.0}\newcommand{\xbf}{.9}
\newcommand{\yaf}{-.12}\newcommand{\ybf}{0.0}
\draw[thick,\fcolour,fill=\fcolour,opacity=0.25] (\xaf,\yaf) rectangle (\xbf,\ybf);
\draw[thick,blue] (.45,-.06) node {fluid};
\end{scope}
\end{tikzpicture}
\end{center}
\caption{The $x$-$t$ diagram for the one-dimensional fluid/rigid-body problem. 
The interface between the rigid body and fluid follows the curve $x=r_b(t)$. The
characteristic variable $\sigma+z v$ in the fluid is constant along the $C^-$ characteristic curve
and provides a relation between the solid velocity, $v_b$, and stress on the body, $\sigma_b$,
in terms of previous fluid values along the characteristic.} \label{fig:fluidRigidBodyCartoon}
\end{figure}
}

From the theory of characteristics\footnote{These characteristic relations are found by seeking linear combinations of the equations~\eqref{eq:linearizedEuler}
for which the equations reduce to ordinary differential equations along space-time {\em characteristic} curves~\cite{Whitham74}.},
the variable 
$\charVar=\sigma + z v$ is constant along the $C^-$ characteristic $dx/dt = - \charSpeed = \vLin -\cLin$.
Therefore, for a point $\rb(t)$ on the body, $\charVar(\rb,t)=\charVar(\rb+ \charSpeed t,0)$, and thus
\begin{align}
  \sigma(\rb,t) + z v(\rb,t)=\sigma_0(\rb+ \charSpeed t)+z v_0(\rb+ \charSpeed t).
\end{align}
Using the interface condition $v(\rb,t)=\vb(t)$ it follows that the stress on the body is
\begin{align}
   \sigma(\rb,t) &= \sigma_0(\rb+ \charSpeed t) + z\big( v_0(\rb+ \charSpeed t) - v_b \big). \label{eq:stressOnBodyFromChars}
\end{align}

Substituting~\eqref{eq:stressOnBodyFromChars} into~\eqref{eq:NewtonEuler1d} gives an equation
for the motion of the body that only depends on the initial data in the fluid and the external body force,
\begin{align}
  \mrb \dotvb &= \sigma_0(\rb+ \charSpeed t) \Area  + z \Area \big( v_0(\rb+ \charSpeed t) - v_b \big) + f_b(t) , 
                  \label{eq:rigidBody1d} \\
     \dotrb &= \vb .
\end{align}
This equation can be written in the form,
\begin{align}
  \mrb \dotvb + z \Area \vb &= \sigma_0(\rb+\charSpeed t) \Area + z \Area v_0(\rb+\charSpeed t) + f_b(t) , 
           \label{eq:rigidBody1dII}  \\
     \dotrb &= \vb , \label{eq:rigidBody1dIIx}
\end{align}
where the {\em added mass term} $z \Area \vb$ has been moved to the left-hand side. 
Note that equations~\eqref{eq:rigidBody1dII}-\eqref{eq:rigidBody1dIIx}
can be used to solve for $\vb$ {\em even when $\mrb=0$} (provided $z \Area >0$).
By using an ODE integration scheme that treats the added mass term $z \Area  \vb$ implicitly, equation~\eqref{eq:rigidBody1dII}
can be used to evolve the rigid body with a time step that need not go to zero as $\mrb$ goes to zero. 

In practical implementation, it is often beneficial to localize~\eqref{eq:stressOnBodyFromChars} in space and time.
Using $\charVar(r_b,t)=\charVar(r_b+s\epsilon,t-\epsilon)$ along the $C^-$ characteristic 
gives
\begin{align}
   \sigma(\rb,t) &= \sigma(\rb+s\epsilon,t-\epsilon) + z\big( v(\rb+s\epsilon,t-\epsilon) - \vb(t) \big), \label{eq:stressOnBodyFromCharsLocal}
\end{align}
and letting $\epsilon\rightarrow 0$ leads to the relation
\begin{align}
   \sigma(\rb,t) &= \sigma(\rb+,t-) + z\big( v(\rb+,t-) - \vb(t) \big). \label{eq:stressOnBodyFromCharsLocal}
\end{align}
Here $\sigma(\rb+,t-)$ and $v(\rb+,t-)$ denote the stress and velocity in the fluid at a point
which lies an infinitesimal distance backward along the $C^-$ characteristic.
Equation~\eqref{eq:stressOnBodyFromCharsLocal} is in a form that can be used in an
interface projection strategy and can be generalized
to a multidimensional problem as is done in Section~\ref{sec:multiDimensionalInterface}. 
Furthermore, notice the similarity of (\ref{eq:stressOnBodyFromCharsLocal}) to equation (\ref{eq:interfaceSigma_limit}).
This hints at the close connection between (\ref{eq:stressOnBodyFromCharsLocal}) and the projection
schemes evaluated in~\cite{banks11a,banks12a,sjogreen12} for coupling compressible fluids and deformable bodies.

\section{Normal mode stability analysis of the one-dimensional FSI model problem} \label{sec:analysis}
In order to understand the stability of a numerical scheme that uses the new interface conditions
\eqref{eq:stressOnBodyFromCharsLocal}, consider the one-dimensional model problem of a rigid 
body confined on either side by an
inviscid compressible fluid, as shown in Fig.~\ref{fig:modelFig}.
As in~\cite{banks11a} we can linearize and freeze
coefficients about a reference state to arrive at a problem where the equations
of acoustics govern the two fluids, and Newtonian mechanics govern the motion of
the solid. 
As shown in Fig.~\ref{fig:modelFig}, the body has a width of $\width$ and its cross-sectional area is assumed to be $1$.
Note that the equations for the fluids are defined in fixed reference coordinates, $x<-\width/2$ and
$x>\width/2$.
\newcommand{\leftState}{$\uLE{L}$}
\newcommand{\rbState}{{$\vRB$}}
\newcommand{\rightState}{$\uLE{R}$}

\begin{figure}[hbt]
\begin{center}
\begin{tikzpicture}[scale=1]
\useasboundingbox (-1,-1.25) rectangle (13,2);  
\draw[thick,red,fill=red,opacity=0.25] (0,0) rectangle (4,1);
\draw[thick,red] (0,0) rectangle (4,1);
\draw[thick,black] (2,0.5) node {fluid (acoustics)};
\draw[thick,blue,fill=blue,opacity=0.25] (8,0) rectangle (12,1);
\draw[thick,blue] (8,0) rectangle (12,1);
\draw[thick,black] (10,0.5) node {fluid (acoustics)};
\draw[thick,green,fill=green,opacity=0.25] (4,0) rectangle (8,1);
\draw[thick,green] (4,0) rectangle (8,1);
\draw[thick,black] (6,0.5) node {solid (rigid body)};
\begin{scope}[yshift=-1.25cm]
\draw[black] ( 2,0) node {$\uLE{L}_i$};
\draw[black] ( 6,0) node {\rbState};
\draw[black] (10,0) node {$\uLE{R}_i$};
\end{scope}
\draw[thick,black,<->] (-1,1.25)--(13,1.25);
\draw[black] (13,1.5) node {$x$};
\draw[thick,black] (4,1.15)--(4,1.35);
\draw[thick,black] (6,1.15)--(6,1.35);
\draw[thick,black] (8,1.15)--(8,1.35);
\draw[black,xshift=-4pt] (4,1.75) node {$-\width/2$};
\draw[black] (6,1.75) node {$0$};
\draw[black] (8,1.75) node {$\width/2$};
\def\ya{-.25}
\begin{scope}[xshift=8cm,yshift=-.25cm]
\draw[-,thick,black,,yshift=-.025cm] 
   (-.5,0) -- (4.25,0) 
   \foreach \x in {-.5,.5,...,4.}{ (\x,-.05) -- (\x,.05) }
   (-.5,\ya) node {$0$}
   (.5,\ya) node {$1$}
   (1.5,\ya) node {$2$}
   (2.5,\ya) node {$3$}
   ( 3.5,\ya) node {$\ldots$}; 
\end{scope}
\begin{scope}[xshift=0cm,yshift=-.25cm]
\draw[-,thick,black,,yshift=-.025cm] 
   (-.5,0) -- (4.5,0) 
   \foreach \x in {.5,1.5,...,4.5}{ (\x,-.05) -- (\x,.05) }
   (4.5,\ya) node {$0$}
   (3.5,\ya) node {$-1$}
   (2.5,\ya) node {$-2$}
   (1.5,\ya) node {$-3$}
   ( 0.5,\ya) node {$\ldots$}; 
\end{scope}
\end{tikzpicture}
\end{center}
\caption{Schematic of the one-dimensional FSI model problem used in the stability analysis. A solid rigid body
is embedded between a fluid domain on the left and a fluid domain on the right.
The boundaries of the rigid
body are located mid-way between the ghost points of the fluid grids with index $i=0$ and the first grid point inside the domain with
index $i=-1$ on the left and $i=1$ on the right.} \label{fig:modelFig}
\end{figure}
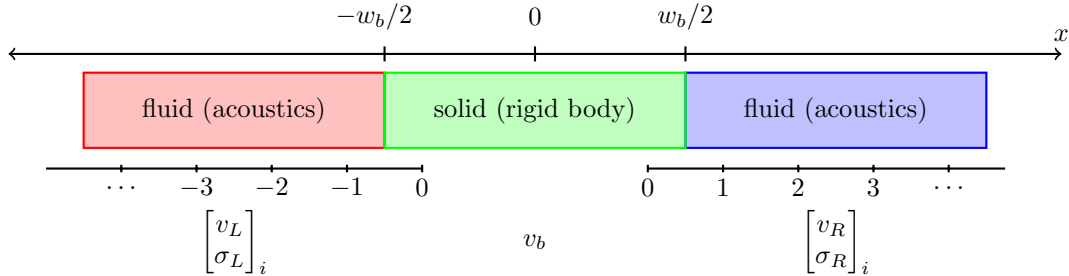
More specifically, the governing equations for the fluid in the left domain are given by
\begin{equation}
  \frac{\partial}{\partial t}  \begin{bmatrix}
    \uL\\
    \sigmaL
  \end{bmatrix} -
  \begin{bmatrix}
    0 \quad& \frac{1}{\rhoL} \\
    \kappaL \quad& 0
  \end{bmatrix}
  \frac{\partial}{\partial x}\begin{bmatrix}
    \uL\\
    \sigmaL
  \end{bmatrix}  =0,
  \quad\hbox{for} \quad x < -\frac{\width}{2},
  \label{eq:LELeft}
\end{equation}
while those for the fluid in the right domain are
\begin{equation}
  \frac{\partial}{\partial t}\begin{bmatrix}
    \uR\\
    \sigmaR
  \end{bmatrix} -
  \begin{bmatrix}
    0 \quad& \frac{1}{\rhoR} \\
    \kappaR \quad& 0
  \end{bmatrix}
  \frac{\partial}{\partial x}\begin{bmatrix}
    \uR\\
    \sigmaR
  \end{bmatrix} =0,
  \quad\hbox{for} \quad x > \frac{\width}{2}. 
  \label{eq:LERight}
\end{equation}
The motion of the rigid body is governed by
\begin{equation}
 \mass \dot{v}_b=\force, 
 \label{eq:RBMotion}
\end{equation}
where the force exerted on the rigid body by the fluid is
\begin{equation}
  \force =\left.\sigmaR\right|_{x=\width/2}-\left.\sigmaL\right|_{x=-\width/2}.\label{eq:ic1}
\end{equation}
The system is closed using interface conditions at $x=\pm \width/2$ which enforce continuity of velocity, namely
\begin{align}
  \left.\uL \right|_{x=-\width/2} & = \vRB,  \label{eq:ic2}\\
  \left.\uR \right|_{x=\width/2} & = \vRB.\label{eq:ic3}
\end{align}
Notice that the problem is posed in a moving reference frame (which we
call $x$), and the frame attached to the rigid body can be
calculated as
\begin{equation*}
    \hat{x} = x+\int_0^t \vRB(\tau) \, d\tau.
\end{equation*}

\subsection{A first-order accurate numerical discretization of the model problem}
\label{sec:simpleInterior}
This section describes the discretization of the governing equations~\eqref{eq:LELeft}-\eqref{eq:RBMotion}
to first-order accuracy.
As in~\cite{banks11a}, Godunov style upwind schemes will be used to discretize the fluid domains. 
We will analyze and demonstrate the properties of these schemes when combined with various discrete interface conditions. The finite difference grid for the discretization of the one-dimensional 
problem is outlined in Fig.~\ref{fig:modelFig}. 
Note that the left and right boundaries of the rigid body are located at the mid-point of computational cells.
This choice is made for convenience, but is not critical to the analysis.
The grid points to the left of the rigid body are denoted by
\begin{equation*}
  x_{L,i} = -\frac{\width}{2}+\left(i+\frac{1}{2}\right)\dxL,\quad i=\ldots, -2,-1,0,
\end{equation*}
and to the right by 
\begin{equation*}
  x_{R,i} = \frac{\width}{2}+\left(i-\frac{1}{2}\right)\dxR,\quad i=0,1,2,\ldots .
\end{equation*}
Ghost points, corresponding to index $i=0$ for both domains, will be used to enforce the interface conditions. 

Let $z_k=\rho_k c_k$ denote the acoustic impedance in domain $k=L,R$. 
The eigen-decomposition of the matrices in (\ref{eq:LELeft}) and (\ref{eq:LERight}) is given by 
\begin{equation}
   C_k \equiv
  \begin{bmatrix}
    0 & \frac{1}{\rho_k}\\
  \rho_kc_k^2 & 0
  \end{bmatrix}=
  R_k\Lambda_k R^{-1}_k, \quad
  R_k = c_k \begin{bmatrix}
    -1 & 1 \\
    z_k & z_k
  \end{bmatrix}, \quad
  \Lambda_k = 
   \begin{bmatrix}
    -c_k & 0\\
   0 & c_k
  \end{bmatrix}, \quad
  R^{-1}_k = \frac{1}{2 c_k z_k}
    \begin{bmatrix}
      -z_k & 1 \\
       z_k & 1 
  \end{bmatrix}. 
  \label{eq:eigenDecomp}
\end{equation}
Let $v_{k,i}^n\approx v_k(x_{k,i},t^n)$ and $\sigma_{k,i}^n\approx \sigma_k(x_{k,i},t^n)$ denote
discrete approximations to the velocity and stress at time $t^n=n\dt$. We also use the notation
\[
   \uLE{k}_i^{n} \equiv \begin{bmatrix} v_{k,i}^n \\ \sigma_{k,i}^n \end{bmatrix}.
\]
The first-order accurate upwind scheme is given by
\begin{equation}
    \uLE{k}_i^{n+1} =  
    \uLE{k}_i^{n} 
    +\Delta t R_k\Lambda_k^{-}R_k^{-1}
    \Dm \uLE{k}_i^n
    +\Delta tR_k\Lambda_k^{+}R_k^{-1}
      \Dp \uLE{k}_i^n  , 
  \label{eq:discreteInterior_fo}
\end{equation}
for $i=\ldots,-2,-1$ on the left and for $i=1,2\ldots$ on the right.
The negative and positive parts of the wave speed matrices are defined by
\begin{equation}
  \Lambda_k^{-} = 
  \left[\begin{array}{cc}
    -c_k & 0\\
   0 & 0
  \end{array}\right] \quad\hbox{and}\quad
  \Lambda_k^{+} = 
  \left[\begin{array}{cc}
    0 & 0\\
   0 & c_k
  \end{array}\right] , 
\end{equation}
respectively. The forward and backward divided difference operators are defined by $D_+u_i=(u_{i+1}-u_i)/\Delta x$ and 
$D_-u_i = D_+u_{i-1}$, where $\dx$ is taken for the appropriate domain.

The methods we consider can be presented using a unified notation.
Motivated by the discussion in Section~\ref{sec:1Dprojection}, the interface stresses at a time $t^{n+1}$ for the first-order scheme are defined by
\begin{align}
  \sigma_{I,L}^{n+1}  = \sigma_{L,-1}^{n+1}+{\alpha_{L}}\left(\vRB^{n+1}- v_{L,-1}^{n+1}\right)  \label{eq:genStressLeft} \\
  \sigma_{I,R}^{n+1}  = \sigma_{R,1}^{n+1}+{\alpha_{R}}\left(v_{R,1}^{n+1} - \vRB^{n+1}\right). \label{eq:genStressRight}
\end{align}
where $\alpha_L$ and $\alpha_R$ are parameters that can be used to obtain
various discrete interface conditions. The traditional interface coupling approach found in the
literature can be described in words as applying the velocity from the solid as
a boundary condition on the fluids, and applying the stress in the fluid to
derive the applied force on the body. This condition is achieved by setting
$\alpha_k = 0$. Our new projection interface condition is given by setting
$\alpha_k= z_k$.

The solution state in the ghost cells at $t^{n+1}$ is defined to first-order accuracy by imposing continuity of the velocity at the interfaces
\begin{align}
  v_{L,0}^{n+1} = \vRB^{n+1}, \label{eq:vLeftFO}\\
  v_{R,0}^{n+1} = \vRB^{n+1}, \label{eq:vRightFO}
\end{align}
and extrapolation of the stress to first-order accuracy as
\begin{align}
  \sigma_{L,0}^{n+1} = \sigma_{I,L}^{n+1}, \label{eq:stressLeftFO}\\
  \sigma_{R,0}^{n+1} = \sigma_{I,R}^{n+1}. \label{eq:stressRightFO}
\end{align}
The rigid body equations~\eqref{eq:RBMotion} are advanced in time with the backward Euler scheme,
\begin{equation}
  \mass\vRB^{n+1} = \mass\vRB^n+{\Delta t} \force^{n+1}
  \label{eq:BE}
\end{equation}
where the force at $t^{n+1}$, $\force^{n+1}$, is defined as
\begin{equation}
  \Force^{n+1} = \sigma_{I,R}^{n+1}-\sigma_{I,L}^{n+1}.
  \label{eq:force1D}
\end{equation}

The backward Euler method is used here in order to simplify the analysis. Used
in isolation, the backward-Euler scheme is unconditionally stable for any $\dt$
independent of $\mass$ provided $\mass>0$. We will show, however, that the fully
coupled FSI problem has a time-step restriction that depends on $\mass$ for the
{\em traditional} interface coupling scheme. For the new interface projection
scheme we show that there is no dependence of the stable time step on $\mass$.
The backward-Euler scheme is, of course, only first-order accurate.  For
higher-order accuracy one can use implicit Runge-Kutta schemes, as described in
Section~\ref{sec:multidimensionalTimeStepping} where we extend the scheme to
multiple space dimensions.  Note that while implicit schemes may be more
expensive per time-step than explicit schemes, they are only used to solve
the rigid body equations which consist of just a few ODEs. As an alternative to
implicit schemes, one can consider using an explicit scheme with a sub-cycling
approach (i.e. taking multiple sub-steps with a smaller value for $\dt$).  Some
remarks on these issues will be provided in subsequent discussions.

In summary, to advance one time level from $t^{n}$ to $t^{n+1}$ using the first-order accurate scheme, the following steps can be followed
\begin{algorithm}\label{alg:firstOrder}
~~\\
\begin{enumerate}
\item Compute $\uLE{L}_i^{n+1}$ for $i=\ldots,-2,-1$ and $\uLE{R}_i^{n+1}$ for $i=1,2,\ldots$ by (\ref{eq:discreteInterior_fo}).
\item Set $\Force^{n+1} = \sigma_{R,1}^{n+1}+\alpha_R(v_{R,1}^{n+1}-\vRB^{n+1}) - 
      \sigma_{L,-1}^{n+1} - \alpha_L(\vRB^{n+1}- v_{L,-1}^{n+1})$, and 
     solve (\ref{eq:BE}) for $\vRB^{n+1}$, 
\begin{equation}
\vRB^{n+1} = \Big[\mass+{\Delta t}(\alpha_L+\alpha_R)\Big]^{-1}\Big[\mass\vRB^n + 
   {\Delta t}\Big(  \sigma^{n+1}_{R,1} +\alpha_R v_{R,1}^{n+1} - 
                            (\sigma^{n+1}_{L,-1}-\alpha_L v_{L,-1}^{n+1})\Big)\Big].
\label{eq:eqforv}
\end{equation}
\item Define the ghost point values at the new time $t^{n+1}$ by the velocity boundary conditions (\ref{eq:vLeftFO}) and (\ref{eq:vRightFO}), along with the stress extrapolations (\ref{eq:stressLeftFO}) and (\ref{eq:stressRightFO}).
\end{enumerate}
\end{algorithm}

\subsection{Normal mode analysis of the first-order accurate scheme}\label{sec:normalMode}
Next, we analyze the stability of the interface discretizations, and investigate
how the choice of $\alpha_L$ and $\alpha_R$ affect the behavior of the overall
numerical method. To simplify the presentation, assume $c_L=c_R=c$,
$\rhoL=\rhoR=\rho$, $\dxL=\dxR=\dx$, and $\alpha_L=\alpha_R=\alpha$. In addition
set $z=z_L=z_R$. These assumptions are purely for convenience and clarity, and
do not materially change the results of the analysis. We pursue a stability
analysis via the normal mode ansatz of Gustafsson Kreiss and Sundstr{\"o}m~\cite{gustafsson72}.

As was done in~\cite{banks11a}, we seek normal mode solutions of the form
\begin{equation}
  \uLE{k}_i^n = \amp^n
  \begin{bmatrix}
    \tilde{v}_k\\
    \tilde{\sigma}_k
  \end{bmatrix}_i, \qquad
  \vRB^n = \amp^n \vRBtilde, \qquad \mbox{for $k=L,R$}, \label{eq:normalmode}
\end{equation}
where $\tilde{v}_{k,i}$ and $\tilde{\sigma}_{k,i}$ are bounded functions of space, 
and $\amp$, the amplification factor, is a complex scalar with $|\amp|>1$. 
If such a non-zero solution can be found (for given values of the parameters $\lambda$, $\dt$, $\mass$, $z$, etc.) then
there are solutions that grow in time and we say that the scheme is unstable for those parameters. We note that more
general definitions of stability allow some bounded growth in time, but for our purposes here we use this more
restrictive definition.
Characteristic normal modes are denoted by
\begin{equation}
  \begin{bmatrix}
    a_k\\
    b_k
  \end{bmatrix}_i = R_k^{-1}
  \begin{bmatrix}
    \tilde{v}_k\\
    \tilde{\sigma}_k
  \end{bmatrix}_i
  = 
    \frac{1}{2 c  z }  \begin{bmatrix}
    \tilde{\sigma}_k - z \tilde{v}_k \\
    \tilde{\sigma}_k + z \tilde{v}_k
  \end{bmatrix}_i.
  \label{eq:char_normal_modes}
\end{equation}
Insertion of (\ref{eq:normalmode}) into the finite difference scheme (\ref{eq:discreteInterior_fo}) leads to 
\begin{equation}
  \left.
  \begin{array}{lcl}
    \amp a_{L,i} & = & a_{L,i} - \lambda \left(a_{L,i}-a_{L,i-1}\right) \medskip  \\
    \amp b_{L,i} & = & b_{L,i} + \lambda \left(b_{L,i+1}-b_{L,i}\right) 
  \end{array}
  \right\} \qquad \hbox{for $i=\ldots,-3,-2,-1$} 
  \label{eq:diff_left_first}
\end{equation}
and
\begin{equation}
  \left.
  \begin{array}{lcl}
    \amp a_{R,i} & = & a_{R,i} - \lambda \left(a_{R,i}-a_{R,i-1}\right)  \medskip  \\
    \amp b_{R,i} & = & b_{R,i} + \lambda \left(b_{R,i+1}-b_{R,i}\right)
  \end{array}
  \right\} \qquad \hbox{for $i=1,2,3,\ldots$} 
  \label{eq:diff_right_first}
\end{equation}
where $0 < \lambda = {c\dt}/{\dx} \le 1.$
Define the quantity 
\[
  r=\frac{\amp-1+\lambda}{\lambda} . 
\]
We see that $|r| > 1$ by rewriting $|r|^2>1$ in terms of the polar variables $R$ and $\theta$, where $\amp = R e^{i\theta}$. By simple algebraic 
manipulations, $|r|^2 > 1$ can be rewritten as
\[
  (R-1)^2 + 2\lambda(R-1) + 2R(1-\lambda)(1-\cos\theta) > 0,
\]
which is true since $R>1$ and $\lambda<1$. 

For the two components on characteristics coming in from infinity, the solution to the difference equations (\ref{eq:diff_left_first}) and (\ref{eq:diff_right_first}) is
\begin{alignat*}{3}
  a_{L,i} & = r^{-(i+1)} a_{L,-1}, \qquad&& \hbox{ for $i=\ldots,-3,-2,-1$}, \\
  b_{R,i} & =  r^{(i-1)} b_{R,1},  \qquad&& \hbox{ for $i=1,2,3,\ldots$ }.
\end{alignat*}
The assumption of boundedness as $i\rightarrow \pm\infty$ gives $a_{L,i} = 0$ for $i\ldots,-3,-2,-1$, and $b_{R,i}=0$ for $i=1,2,3,\ldots$. Note that $a_{L,0}$ and $b_{R,0}$ do not play a role in the difference equations (\ref{eq:diff_left_first}) and (\ref{eq:diff_right_first}), but their values can be determined algebraically using the interface conditions 
\begin{align*}
  a_{L,0} & = \frac{\alpha-z}{2 z} (\vRBtilde/c -b_{L,-1}) , \\
  b_{R,0} & = \frac{z-\alpha}{2 z} (\vRBtilde/c+ a_{R,1}).
\end{align*}
The remainder of the solution to difference equations (\ref{eq:diff_left_first}) and (\ref{eq:diff_right_first}) is given by 
\begin{alignat}{3}
  a_{R,i} & = r^{-i} a_{R,0}, \qquad&& \hbox{ for $i=0,1,2,3,\ldots$}, \label{eq:alsol} \\
  b_{L,i} & =  r^i b_{L,0},  \qquad&& \hbox{ for $i=\ldots,-3,-2,-1,0$ }. \label{eq:brsol}
\end{alignat}
The solutions (\ref{eq:alsol}) and (\ref{eq:brsol}) are bounded because $|r| > 1$. The definition of the characteristic normal modes on the interior yields
\begin{equation}
  \begin{bmatrix}
    \tilde{v} \\ \tilde{\sigma}
  \end{bmatrix}_{L,i}  = 
  c \begin{bmatrix}
    1 \\ z 
  \end{bmatrix}
  r^{i}b_{L,0}\qquad \hbox{ for $i=\ldots,-3,-2,-1$ } ,
  \label{eq:ulsol}
\end{equation}
and
\begin{equation}
  \begin{bmatrix}
    \tilde{v} \\ \tilde{\sigma}
  \end{bmatrix}_{R,i}  = 
  c \begin{bmatrix}
    -1 \\ z
  \end{bmatrix}
  r^{-i}a_{R,0} \qquad \hbox{ for $i=1,2,3,\ldots$}.
  \label{eq:ursol}
\end{equation}
The three undetermined constants $b_{L,0}$, $a_{R,0}$, and $\vRBtilde$ are defined by application of the interface conditions (\ref{eq:vLeftFO})-(\ref{eq:stressRightFO}) and the rigid body integrator (\ref{eq:BE}). This leads to the linear system of equations
\begin{equation}
  \begin{bmatrix}
    \displaystyle{1 + \frac{\alpha-z}{2 z r}} & 0 & \displaystyle{-\frac{\alpha+z}{2z}} \medskip \\
    0 & \displaystyle{1 + \frac{\alpha-z}{2 z r}} & \displaystyle{\frac{\alpha+z}{2 z}} \medskip \\
    \displaystyle{\frac{\amp\dt}{r}\left(z-\alpha\right)} & \displaystyle{-\frac{\amp\dt}{r}\left(z-\alpha\right)} & \displaystyle{\mass(\amp-1)+2\dt\alpha \amp}
  \end{bmatrix}
  \begin{bmatrix}
    b_{L,0}  \\ a_{R,0} \medskip \\ \vRBtilde/c
  \end{bmatrix}
  = 0.
  \label{eq:eig}
\end{equation}
The system (\ref{eq:eig}) is an eigenvalue problem for $\amp$, in
the sense that if there is an $\amp$ such that the determinant of the system is zero, then there exists 
a non-trivial solution of the form (\ref{eq:normalmode}). If, furthermore, $|\amp|>1$, then the 
solution (\ref{eq:normalmode}) grows in time. 
\begin{theorem}
The numerical scheme using the interior discretizations (\ref{eq:discreteInterior_fo}), 
interface conditions (\ref{eq:vLeftFO})-(\ref{eq:stressRightFO}), rigid body integrator (\ref{eq:BE}) 
and projections (\ref{eq:genStressLeft})-(\ref{eq:genStressRight}) with $\alpha=z$ has no 
eigenvalues $\amp$ with $|\amp|>1$ for $\lambda \le 1$ and $\mass \ge 0$.
  \label{thm:projectedFO}
\end{theorem}
\begin{proof}
For $\alpha=z$ the eigenvalue problem (\ref{eq:eig}) reduces to 
\begin{equation}
  \begin{bmatrix}
    \displaystyle{1} & 0 & -1  \\
    0 & \displaystyle{1} & 1  \\
    0 & 0 & \displaystyle{\mass(\amp-1)+2\dt z \amp}
  \end{bmatrix}
  \begin{bmatrix}
    b_{L,0} \medskip \\ a_{R,0} \medskip \\ \vRBtilde/c
  \end{bmatrix}
     = 0.
\end{equation}
The determinant is zero when $\amp= \mass/(\mass+2\dt z)$.
By assumption, $\dt > 0$ and $z>0$ and so $|\amp| < 1$. 
\end{proof}

\begin{theorem}
The numerical scheme using the interior discretizations (\ref{eq:discreteInterior_fo}), interface 
conditions (\ref{eq:vLeftFO})-(\ref{eq:stressRightFO}), rigid body integrator (\ref{eq:BE}) and 
projections (\ref{eq:genStressLeft})-(\ref{eq:genStressRight}) with $\alpha=0$ has no eigenvalues $\amp$ with
$|\amp|>1$ when 
\begin{equation}
\Delta t < \mass(4-\lambda)/(z \lambda)
\label{eq:trestr}
\end{equation}
for $\lambda \leq 1$. Conversely, if 
$\Delta t > \mass(4-\lambda)/(z \lambda)$, then there are eigenvalues with $|\amp|>1$ for $\lambda \leq 1$.
\label{thm:traditionalFO}
\end{theorem}
\begin{proof}
For $\alpha=0$, the eigenvalue problem (\ref{eq:eig}) reduces to
\begin{equation}
  \begin{bmatrix}
    \displaystyle{1-\frac{1}{2r}} & 0 & \displaystyle{-\frac{1}{2}} \medskip \\
    0 & \displaystyle{1-\frac{1}{2r}} & \displaystyle{\frac{1}{2}} \medskip \\
    \displaystyle{\frac{\amp z \dt}{r}} & \displaystyle{-\frac{\amp z \dt}{r}} & \displaystyle{\mass(A-1)}
  \end{bmatrix}
  \begin{bmatrix}
    b_{L,0}  \\ a_{R,0}  \\ \vRBtilde/c
  \end{bmatrix}
   = 0.
\end{equation}
The zero determinant condition is solved to give three roots $\amp_1=1-\lambda/2$ and 
\begin{equation}
\amp_{2,3}=1-\frac{\lambda}{4}-\frac{z \xi \lambda}{2} \pm 
                \sqrt{ \left(1-\frac{\lambda}{4}-\frac{z \xi \lambda}{2}\right)^2 -1 +\frac{\lambda}{2} }
\label{eq:roots23}
\end{equation}
where $\xi={\dt}/{\mass}$. Clearly, $|\amp_1|<1$ for $\lambda \leq 1$. In the case
$$
\left(1-\frac{\lambda}{4}-\frac{z \xi \lambda}{2}\right)^2 -1 +\frac{\lambda}{2} < 0,
$$
$\amp_2$ and $\amp_3$ are complex conjugate, and 
$$
 |\amp_2|^2=|\amp_3|^2 = \left(1-\frac{\lambda}{4}-\frac{z \xi \lambda}{2}\right)^2 - 
\left(1-\frac{\lambda}{4}-\frac{z \xi \lambda}{2}\right)^2 + 1 - \frac{\lambda}{2} = 1 -\frac{\lambda}{2} < 1.
$$
When $\amp_2$ and $\amp_3$ are real, rewriting (\ref{eq:roots23}) as
$$
A_{2,3}=1-(\frac{\lambda}{4}+\frac{z \xi \lambda}{2}) \pm 
                \sqrt{ \left(\frac{\lambda}{4}+\frac{z \xi \lambda}{2}\right)^2 -z\xi\lambda} , 
$$
shows directly that both roots are are always $ <1$, hence $|\amp|<1$ 
if and only if
$$
 -1 < 1-(\frac{\lambda}{4}+\frac{z \xi \lambda}{2}) - 
             \sqrt{ \left(\frac{\lambda}{4}+\frac{z \xi \lambda}{2}\right)^2 -z\xi\lambda} ,
$$
which is equivalent to
\begin{equation}
 \sqrt{ \left(\frac{\lambda}{4}+\frac{z \xi \lambda}{2}\right)^2 -z\xi\lambda} 
            < 2-(\frac{\lambda}{4}+\frac{z \xi \lambda}{2}).
\label{eq:negcond}
\end{equation}
The necessary condition that the right hand side is positive is equivalent to
\begin{equation}
  z\xi\lambda < 4 - \frac{\lambda}{2}.
\label{eq:poscond}
\end{equation}
Assume (\ref{eq:poscond}) holds and square both sides of (\ref{eq:negcond}) to obtain
\begin{equation}
  z\xi\lambda < 4 - \lambda.
\label{eq:finalcond}
\end{equation}
Hence, $|\amp|<1$ exactly when (\ref{eq:finalcond}) holds. The proof is completed by observing that
(\ref{eq:finalcond}) is equivalent to (\ref{eq:trestr}).
\end{proof}

{\em Remark:}
Theorem~\ref{thm:traditionalFO} indicates the traditional coupling scheme with $\alpha=0$ has a time step restriction
that can be more strict than that for the fluid domains alone. 
Even though the rigid body is formally 
integrated with the Backward Euler scheme, which would be unconditionally stable when used in isolation, the coupled 
formulation does not include the full dependence of the forcing $\Force^{n+1}$ on $\vb^{n+1}$. In the case of light 
bodies, i.e., bodies with small $\mass$, the time step restriction for the scheme with $\alpha=0$, can be severe. 
Another way to state the result is that for any fixed grid resolution, there exists some 
sufficiently small mass for which the solution will have exponential growth in time.
In fact, it is easy to see that for $\alpha=0$  and 
fixed $\dt$, the limit of small mass yields 
$r \sim 1 - {\lambda}/{4}-{z\lambda\dt}/{\mass}$ and 
therefore $\lim_{\mass\to0}|\amp|=\infty$. 
Note, however, that the traditional coupling scheme with $\alpha=0$ is stable, for any
finite mass $\mass$, provided the time step satisfies the conditions given in~\ref{thm:traditionalFO}.

{\em Remark:} Theorem~\ref{thm:projectedFO} shows that the time step restriction (\ref{eq:trestr}) can be avoided 
by switching to the interface conditions with $\alpha=z$.

{\em Remark:} 
The structure of the eigenvalue problem in the proof of Theorem~\ref{thm:projectedFO} suggests why the
choice $\alpha=z$ is in some sense optimal. 
When $\alpha=z$, the rigid body mode is decoupled from the fluid modes and stability follows for any $\mass\ge 0$.
On the other hand, for $\alpha=0$,
the eigenvalue problem (\ref{eq:eig}) represents a coupled system and the question of stability is summarized in
Theorem~\ref{thm:traditionalFO}. 
\par
{\em Remark:} For choices of $\alpha$ other than zero or $z$, the stability of the numerical scheme 
varies somewhat. The determinant condition can be used to produce an expression for $\amp$, but it is somewhat
difficult to interpret. We provide no further discussion about other choices of the parameter $\alpha$.

\subsection{A second-order accurate numerical discretization of the model problem}
We look now at the formulation and stability of a second-order accurate version of the projection interface scheme. 
For the discretization of the fluid domains we choose the second-order accurate Lax-Wendroff scheme, 
\begin{equation}
    \uLE{k}_i^{n+1} =  
    \uLE{k}_i^{n} 
    +\Delta t C_k D_0 \uLE{k}_i^n
    +\frac{\Delta t^2}{2} C_k^2\Dp\Dm \uLE{k}_i^n\qquad k=L,R ,
  \label{eq:discreteInterior_so}
\end{equation}
where $D_0= (\Dp+\Dm)/2$ is the centered difference operator, and
$C_k$ has been defined previously in~\eqref{eq:eigenDecomp}.
The Lax-Wendroff scheme is a good model, since many 
non-linear schemes of TVD type are designed to approximate the Lax-Wendroff scheme in the parts of
the computational domain where the solution is smooth. 
The projection coupling conditions can be implemented to second-order accuracy
as follows. Define interface stresses on the left and right at any time $t^n$ by
\begin{align}
  \sigma_{I,L}^n &= \frac{3\sigma_{L,-1}^n-\sigma_{L,-2}^n}{2}+\alpha_L\left(\vRB^n-\frac{3v_{L,-1}^n-v_{L,-2}^n}{2} \right), \label{eq:stressLeftSO}\\
  \sigma_{I,R}^n & = \frac{3\sigma_{R,1}^n-\sigma_{R,2}^n}{2}+\alpha_R\left(\frac{3v_{R,1}^n-v_{R,2}^n}{2} -\vRB^n\right) \label{eq:stressRightSO}.
\end{align}
These are obtained by extrapolation from domain interiors, and subsequent projection. The force at any time level $t^n$ is defined as before using (\ref{eq:force1D}), and a second-order accurate trapezoidal integration for the solid is then defined
\begin{equation}
  \mass\frac{\vRB^{n+1}-\vRB^n}{\dt} = \frac{1}{2}\left(\force^{n+1}+\force^{n}\right).
  \label{eq:trap}
\end{equation}
The velocity from the solid is applied 
as a boundary condition on the fluids to second-order accuracy by setting the average $(v_{L,0}+v_{L,-1})/2$ 
equal to $\vRB$ (and similarly at the right interface), or equivalently 
\begin{align}
  v_{L,0}^{n+1} & = 2\vRB^{n+1}-v_{L,-1}^{n+1} , \label{eq:vLeftSO} \\
  v_{R,0}^{n+1} & = 2\vRB^{n+1}-v_{R,1}^{n+1} \label{eq:vRightSO}.
\end{align}
Extrapolation of the stress to the ghost cells gives
\begin{align}
  \sigma_{L,0}^{n+1} & = 2\sigma_{I,L}^{n+1}-\sigma_{L,-1}^{n+1} , \label{eq:sigmaLeftSO} \\
  \sigma_{R,0}^{n+1} & = 2\sigma_{I,R}^{n+1}-\sigma_{R,1}^{n+1} \label{eq:sigmaRightSO}.
\end{align}

In summary, to advance one time level from $t^{n}$ to $t^{n+1}$ using the second-order accurate scheme, the following steps can be followed
\begin{algorithm}\label{alg:secondOrder}
~~\\
\begin{enumerate}
\item Compute $\uLE{L}_i^{n+1}$ for $i=\ldots,-2,-1$ and $\uLE{R}_i^{n+1}$ for $i=1,2,\ldots$ using (\ref{eq:discreteInterior_so}).
\item Define $F^{n+1}$ using the computed solution at $t^{n+1}$ and solve (\ref{eq:trap}) to yield
\begin{align}
  \vRB^{n+1} = & 
    \left[\mass+\frac{\dt \alpha_L}{2} \,+\right. \left.\frac{\dt \alpha_R}{2}\right]^{-1}\left[\left(\mass-\frac{\dt \alpha_L}{2}-\frac{\dt \alpha_R}{2}\right)\vRB^{n}+\right.\nonumber \\ 
  & \frac{\dt}{2}\left(\frac{3\sigma_{R,1}^{n+1}-\sigma_{R,2}^{n+1}}{2}+\frac{3\sigma_{R,1}^n-\sigma_{R,2}^n}{2}\right)-
  \frac{\dt}{2}\left(\frac{3\sigma_{L,-1}^{n+1}-\sigma_{L,-2}^{n+1}}{2}+\frac{3\sigma_{L,-1}^n-\sigma_{L,-2}^n}{2}\right)+ \nonumber \\
  & \left. \frac{\alpha_R\dt}{2}\left(\frac{3v_{R,1}^{n+1}-v_{R,2}^{n+1}}{2}+\frac{3v_{R,1}^n-v_{R,2}^n}{2}\right)+
  \frac{\alpha_L\dt}{2}\left(\frac{3v_{L,-1}^{n+1}-v_{L,-2}^{n+1}}{2}+\frac{3v_{L,-1}^n-v_{L,-2}^n}{2}\right)\right].
  \label{eq:RBTrap}
\end{align}
\item Define the ghost point values at the new time $t^{n+1}$ for the fluid domains using equations (\ref{eq:vLeftSO}) -- (\ref{eq:sigmaRightSO}).
\end{enumerate}
\end{algorithm}

\subsection{Normal mode analysis of the second-order accurate scheme}
For the analysis, we make the same assumptions as in Sec.~\ref{sec:normalMode}, that the grid spacings and wave
speeds are the same on both sides of the body and that $\alpha_L=\alpha_R=\alpha$.
First, we decompose (\ref{eq:discreteInterior_so}) into characteristic components, 
and obtain the two scalar equations on each side of the body,
\begin{equation}
a_{k,i}^{n+1} = a_{k,i}^n - c \Delta t D_0 a_{k,i}^n + \frac{c^2\Delta t^2}{2}\Dp\Dm a_{k,i}^n \qquad 
b_{k,i}^{n+1} = b_{k,i}^n + c \Delta t D_0 b_{k,i}^n + \frac{c^2\Delta t^2}{2}\Dp\Dm b_{k,i}^n
\label{eq:lwchareq}
\end{equation}
where $k=L,R$,  $i=\ldots,-2,-1$ for $k=L$, and $i=1,2,\ldots$ for $k=R$.
The normal modes are found by inserting $a_i^n =\amp^n r^i$ and $b_i^n = \amp^n r^i$ into (\ref{eq:lwchareq}). 
This leads to the characteristic equation
\begin{equation}
   \frac{1}{2}(\nu + \nu^2)r^2 + ( 1-\amp - \nu^2)r + \frac{1}{2}(\nu^2-\nu) = 0,
\label{eq:lwstability}
\end{equation}
where $\nu = c\Delta t/\Delta x$ for the $b$ characteristic component 
and $\nu=-c\Delta t/\Delta x$ for the $a$ characteristic component. 
The assumption $c = c_L = c_R$ gives the same characteristic equation on either side of the body.
There are four roots, two for the $-c$ characteristic, that we denote $r_1^-$ and $r_2^-$, and two roots 
for the $c$ characteristic,
that we denote $r_1^+$ and $r_2^+$. It is well-known, see e.g., \cite{gustafsson72}, that for the 
equation $u_t = cu_x$ under the CFL-condition $\lambda<1$, the two roots of (\ref{eq:lwstability}) satisfy 
\begin{alignat}{2}
  |r_1^+| \leq 1-\delta &\qquad |\amp| \geq 1 , \nonumber \\
  |r_2^+| > 1 & \qquad |\amp| \geq 1, \quad \amp\neq 1 , \label{eq:root1} \\
   r_2^+ = 1, & \qquad \amp=1  \nonumber,
\end{alignat}
for some $\delta >0$ when $c>0$, and
\begin{alignat}{2}
  |r_1^-| < 1 &\qquad |\amp| \geq 1, \quad  \amp\neq 1 ,  \nonumber \\
  r_1^- = 1, & \qquad \amp=1, \label{eq:root2}  \\ 
  |r_2^-| \geq 1+\delta & \qquad |\amp| \geq 1, \nonumber
\end{alignat}
when $c<0$.
For the model problem (\ref{eq:discreteInterior_so}), there are thus four roots.
From (\ref{eq:root1}), (\ref{eq:root2}) and the condition of boundedness
at infinity, it follows that the $r_1^+$ and $r_1^-$ components are zero for $i<0$ and that
the $r_2^+$ and $r_2^-$ components are zero for $i>0$.
Hence, the normal mode solutions to the left and to the right of the body can be written
\begin{equation}
  \begin{bmatrix}
    \tilde{v} \\ \tilde{\sigma}
  \end{bmatrix}_{L,i}  = 
  c 
  \begin{bmatrix} 
    -1 \\ z 
  \end{bmatrix}
  (r_2^-)^{i}a_{L,0} + 
   c 
  \begin{bmatrix}
    1 \\ z 
  \end{bmatrix}
    (r_2^+)^{i}b_{L,0}
\qquad \hbox{ for } i\le 0 ,
  \label{eq:ulsollw}
\end{equation}
and
\begin{equation}
  \begin{bmatrix}
    \tilde{v} \\ \tilde{\sigma}
  \end{bmatrix}_{R,i}  = 
  c
  \begin{bmatrix} 
    -1 \\ z
  \end{bmatrix}
  (r_1^-)^{i}a_{R,0} +  
  c
  \begin{bmatrix}
    1 \\ z
  \end{bmatrix}
 (r_1^+)^{i}b_{R,0} \qquad 
\qquad \hbox{ for } i\ge 0,
  \label{eq:ursollw}
\end{equation}
respectively. 
\par
The solutions (\ref{eq:ulsollw}) and (\ref{eq:ursollw}) inserted into the interface 
conditions (\ref{eq:vLeftSO}), (\ref{eq:vRightSO}), (\ref{eq:sigmaLeftSO}), and 
(\ref{eq:sigmaRightSO}) together with (\ref{eq:RBTrap}) give five equations for 
the five unknowns $a_{R,0}, b_{R,0}, a_{L,0}, b_{L,0}$, and $\vRB$. Fully written out these equations are
\begin{alignat}{2}
\left(1+\frac{1}{r_2^-}\right)a_{L,0} - \left(1+\frac{1}{r_2^+}\right)b_{L,0} + \frac{2}{c} \vRB &= 0 ,\label{eq:nmso1}\\
    \left(1+ r_1^-\right)a_{R,0}      -       \left(1+r_1^+\right)b_{R,0}     + \frac{2}{c} \vRB &= 0 ,\label{eq:nmso2}\\ 
\left(1+\frac{1}{r_2^-} -2\left(\frac{\alpha}{z} + 1\right)\left(\frac{3}{2}\frac{1}{r_2^-} - \frac{1}{2}\frac{1}{(r_2^-)^2}\right) \right)a_{L,0} +
\left(1+\frac{1}{r_2^+} +2\left(\frac{\alpha}{z} - 1\right)\left(\frac{3}{2}\frac{1}{r_2^+}-\frac{1}{2}\frac{1}{(r_2^+)^2}\right)\right)b_{L,0} -
\frac{2\alpha}{z} \frac{\vRB}{c} &= 0 ,\label{eq:nmso3}\\
\left(1+r_1^- +2\left(\frac{\alpha}{z} - 1\right)\left(\frac{3}{2}r_1^- - \frac{1}{2}(r_1^-)^2\right) \right)a_{R,0} +
\left(1+r_1^+ -2\left(\frac{\alpha}{z} + 1\right)\left(\frac{3}{2}r_1^+ - \frac{1}{2}(r_1^+)^2\right) \right)b_{R,0} +
\frac{2\alpha}{z} \frac{\vRB}{c} &= 0, \label{eq:nmso4}\\
\left(\frac{\alpha}{z} + 1\right)\left(\frac{3}{2}\frac{1}{r_2^-} - \frac{1}{2}\frac{1}{(r_2^-)^2}\right)a_{L,0}
+ \left(1-\frac{\alpha}{z}\right)\left(\frac{3}{2}\frac{1}{r_2^+}-\frac{1}{2}\frac{1}{(r_2^+)^2}\right)b_{L,0} \hspace{2in} & \nonumber \\
+ \left(1-\frac{\alpha}{z}\right)\left(\frac{3}{2}r_1^- - \frac{1}{2}(r_1^-)^2\right)a_{R,0} 
-\left(1+\frac{\alpha}{z}\right)\left(\frac{3}{2}r_1^+ - \frac{1}{2}(r_1^+)^2\right)b_{R,0} + 
  \left( \frac{\amp-1}{\amp+1}\frac{2\mass }{\Delta t z} + \frac{2\alpha}{z}\right)\frac{\vRB}{c} & = 0 .\label{eq:nmso5}
\end{alignat}
For the case $\alpha=z$ the system (\ref{eq:nmso1})-(\ref{eq:nmso5}) becomes
\begin{alignat}{2}
\left(1+\frac{1}{r_2^-}\right)a_{L,0} - \left(1+\frac{1}{r_2^+}\right)b_{L,0} + \frac{2}{c} \vRB &= 0, \label{eq:nmso1b}\\
\left(1+ r_1^-\right)a_{R,0}      -       \left(1+r_1^+\right)b_{R,0}     + \frac{2}{c} \vRB &= 0, \label{eq:nmso2b}\\ 
\left(1-\frac{5}{r_2^-} + \frac{2}{(r_2^-)^2} \right)a_{L,0} +
\left(1+\frac{1}{r_2^+}\right)b_{L,0} - \frac{2}{c} \vRB &= 0, \label{eq:nmso3b}\\
\left(1+r_1^- \right)a_{R,0} +
\left(1-5r_1^+ + 2(r_1^+)^2 \right)b_{R,0} + \frac{2}{c} \vRB &= 0, \label{eq:nmso4b}\\
\left(\frac{3}{r_2^-} - \frac{1}{(r_2^-)^2}\right)a_{L,0} - \left(3 r_1^+ - (r_1^+)^2\right)b_{R,0} +
  \left( \frac{\amp-1}{\amp+1}\frac{2\mass }{\Delta t z} + 2 \right)\frac{\vRB}{c} & = 0 .\label{eq:nmso5b}
\end{alignat}

\begin{theorem}
When $|\amp|\geq 1$ and $\lambda < 1$, the system (\ref{eq:nmso1b})--(\ref{eq:nmso5b}) only has the trivial 
solution $a_{L,0}=b_{L,0}=a_{R,0}=b_{R,0}=\vRB=0$. Hence, the numerical scheme using the interior 
discretizations (\ref{eq:discreteInterior_so}), interface conditions (\ref{eq:vLeftSO})-(\ref{eq:sigmaRightSO}), 
rigid body integrator (\ref{eq:RBTrap}) and projections (\ref{eq:stressLeftSO})-(\ref{eq:stressRightSO}) has
no exponentially growing modes for $\lambda < 1$ and $\mass \geq 0$.
 \label{thm:projectedSO}
\end{theorem}
\begin{proof}
Adding equations (\ref{eq:nmso1b}) and (\ref{eq:nmso3b}) gives 
$$
  2\left(1-\frac{1}{r_2^-}\right)^2a_{L,0} = 0.
$$
Because of (\ref{eq:root2}), 
$$
 \left|1-\frac{1}{r_2^-}\right| \geq 1 - \frac{1}{|r_2^-|} \geq 1-\frac{1}{1+\delta} = \frac{\delta}{1+\delta}>0 ,
$$
and consequently, $a_{L,0}=0$. Similarly, subtracting (\ref{eq:nmso2b}) from (\ref{eq:nmso4b}) and
using (\ref{eq:root1}) gives $b_{R,0}=0$. Equation (\ref{eq:nmso5b}) with $a_{L,0}=b_{R,0}=0$ gives
\begin{equation}
    \left( \frac{\amp-1}{\amp+1}\frac{2\mass }{\Delta t z} + 2 \right)\frac{\vRB}{c} = 0 .
\label{eq:nmvrbeq}
\end{equation}
A non-trivial solution exists if
$$
 \frac{\amp-1}{\amp+1}\frac{2\mass }{\Delta t z} + 2 = 0,
$$
which is equivalent to $\amp = (\mass- \dt z)/(\mass+\dt z)$.
Assuming that for $\dt > 0$, $z>0$, and $\mass \ge 0$, it follows that $|\amp|<1$, and hence that the only solution 
of (\ref{eq:nmvrbeq}) when $|\amp|\geq 1$ is $\vRB = 0$. Finally, the remaining equations 
$$
 \left(1+\frac{1}{r_2^+}\right)b_{L,0}=0 \qquad \hbox{and}\qquad  (1+r_1^-)a_{R,0}=0 ,
$$
have the unique solutions $b_{L,0}=a_{R,0}=0$, because (\ref{eq:root1}) and (\ref{eq:root2}) exclude
the possibility that $r_2^+=-1$ or $r_1^-=-1$ when $|\amp|\geq 1$.
\end{proof}

\newcommand{\graphWidth}{7.cm}

\section{Numerical demonstration of the theory for the FSI model problem}\label{sec:1Dresults}

We now present numerical results from solving the one-dimensional FSI problem introduced
in Section~\ref{sec:1Dprojection}. The aim is to demonstrate the
accuracy and stability of the new FSI projection algorithm
For this purpose we use the exact solution derived
in~\ref{sec:testProblem}. The problem consists of an initial Gaussian pulse in the fluid that moves
left to right and interacts with the rigid body.
The initial conditions for the velocity
and stress are given by
\[
  v(x,t=0) = \frac{c_L}{2}\exp\left(-\beta^2(x-x_0)^2\right), \quad  \sigma(x,t=0) =  -\frac{\kappaL}{2}\exp\left(-\beta^2(x-x_0)^2\right).
\]
The rigid body is initially at rest. 
The exact solution is defined by (\ref{eq:exactSol}), (\ref{eq:leftSol}) and
(\ref{eq:rightSol}). Throughout this section we use $\rhoL = 1$, $c_L=\sqrt{2}$, 
$\rhoR = 1$, $c_R=\sqrt{3}$, $\beta=10$ and $x_0=-1/2$. 
Note that the initial conditions (\ref{eq:test_u_ic}) and
(\ref{eq:test_v_ic}), and exact solutions (\ref{eq:exactSol}),
(\ref{eq:leftSol}) and (\ref{eq:rightSol}) may require differentiation with
respect to space and/or time in order to be used or compared with the dependent
variables of velocity and stress which we use.

\subsection{Easy case: rigid body with mass one}

\begin{figure}
\begin{center}
  \includegraphics[width=\graphWidth]{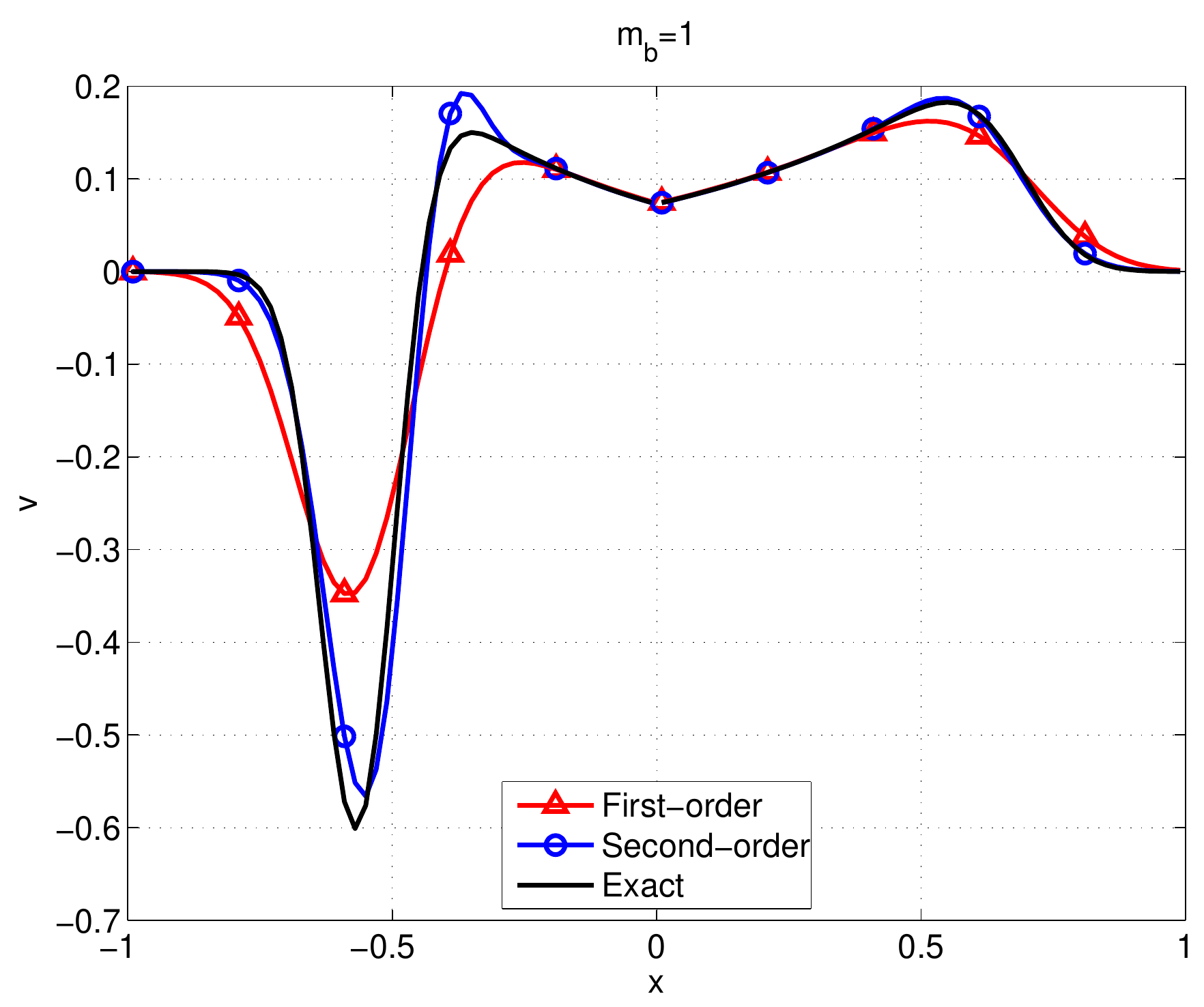}
  \includegraphics[width=\graphWidth]{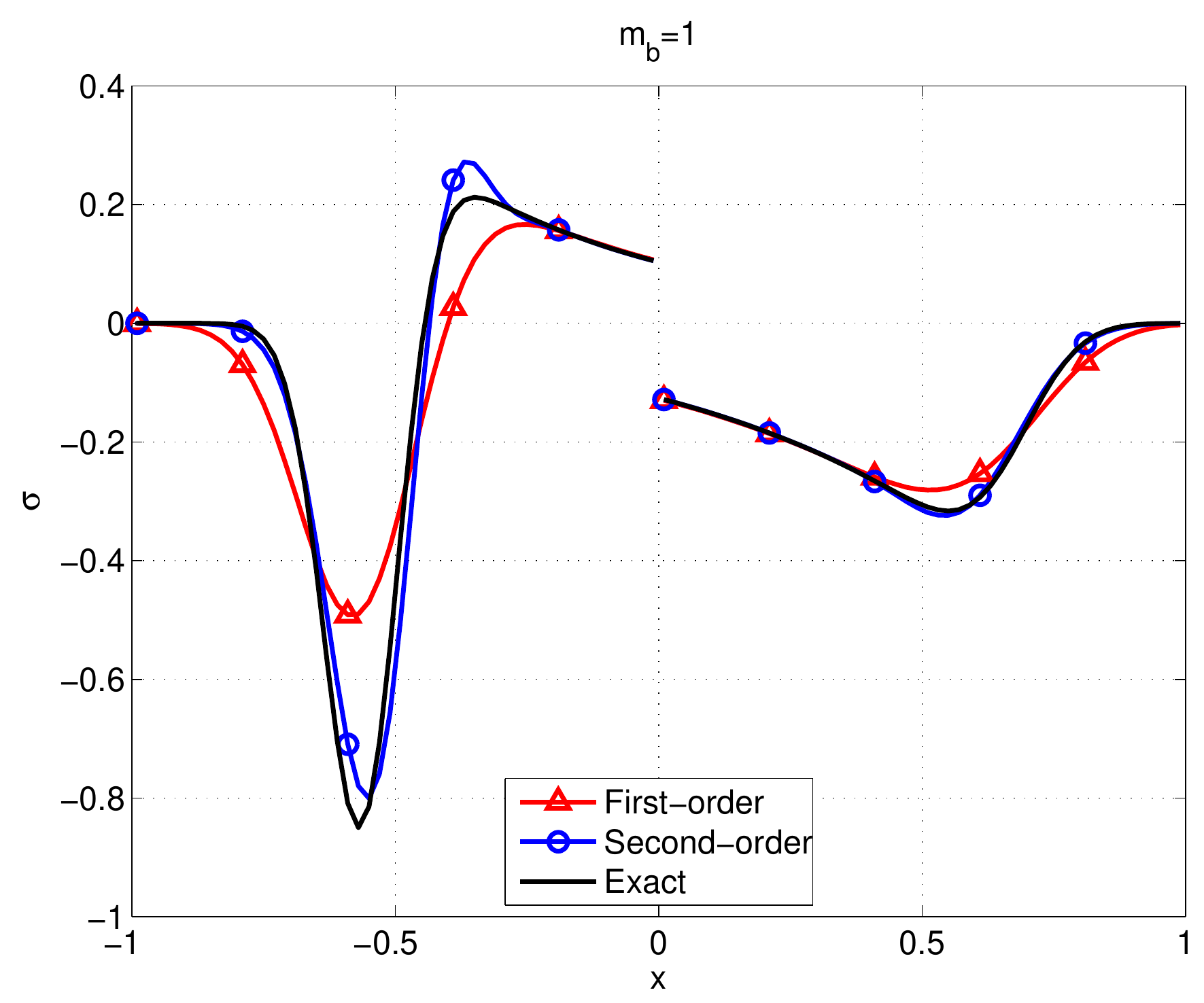}
  \includegraphics[width=\graphWidth]{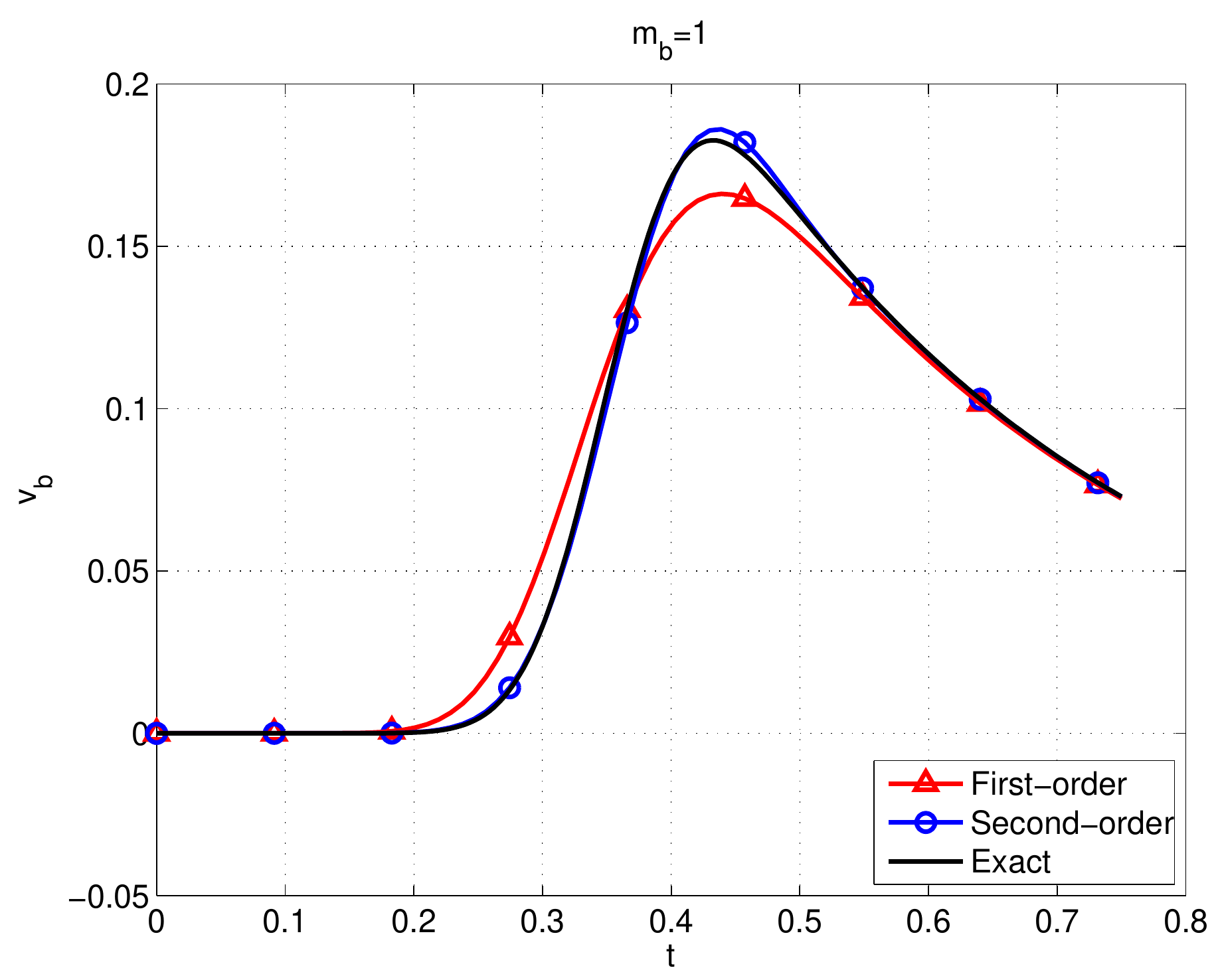}
  \includegraphics[width=\graphWidth]{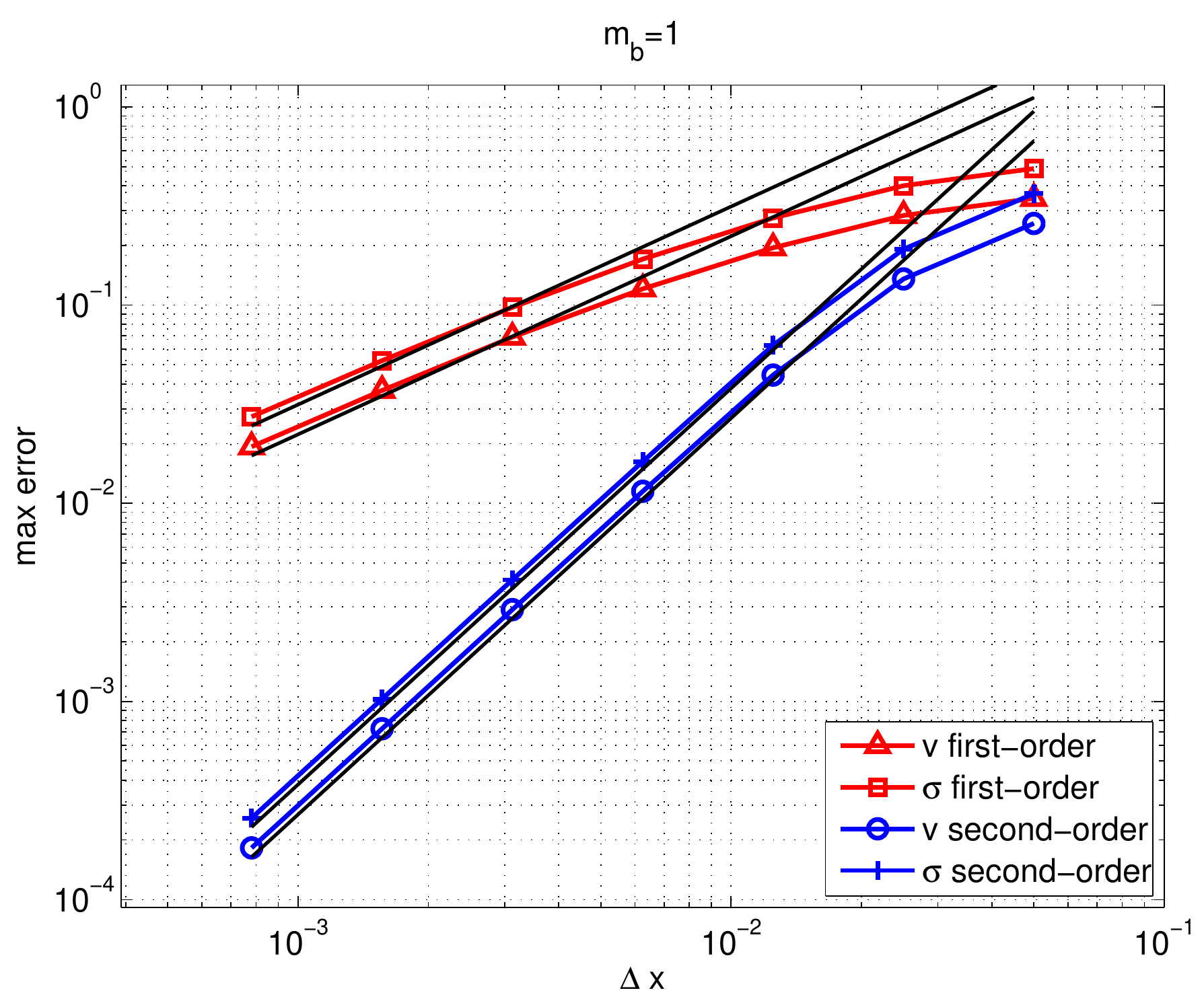}
  \caption{Results for the one-dimensional FSI problem with $\mass=1$ for the first- and second-order accurate schemes. 
           Top left: velocity at $t=0.75$. Top-right: stress at $t=0.75$. Bottom left: velocity of the rigid body, $v_b$ versus time. 
           Bottom right: convergence of the max-norms errors (reference lines of the corresponding order are displayed in black).
           The solutions are plotted in the reference domain $[-1,0]$ for the
           left domain and $[0,1]$ for the right domain with the rigid body of width $\width=0$ located at $x=0$.}
  \label{fig:sol_mass1}
\end{center}
\end{figure}

We begin our numerical results with a case where the CFL time-step constraint in the
fluids is dominant over the explicit ODE time-step constraint for the rigid body. This is the case
when
\[
  \max(z_L,z_R)<\mass\min\left(\frac{\cL}{\dxL},\frac{\cR}{\dxR}\right),
\]
which implies that time steps which satisfy the usual CFL stability constraint
in the fluid also satisfy the stability constraint associated with the ODE for
rigid body motion. As a result, the traditional interface coupling technique found in the
literature has no difficulty, and we are simply setting out to demonstrate
that the new interface projection technique remains accurate for this case.

Figure~\ref{fig:sol_mass1} shows simulation results for $\mass=1$ when using
the first-order accurate upwind scheme for the two fluid domains, the backward Euler
integrator for the rigid body evolution equation, and the interface projection scheme with $\alpha=z$
as defined by Algorithm~\ref{alg:firstOrder}.
In addition we show results  
using the second-order accurate Lax-Wendroff scheme for the fluid domains
together with the trapezoidal rule for integration of the rigid body
as defined by Algorithm~\ref{alg:secondOrder} with $\alpha=z$.

For both cases we use $\dxL=\dxR=1/50$. The exact solution and
numerical approximations for $v$ and $\sigma$ are displayed as functions of the reference coordinate $x$ 
at $t=0.75$, and the velocity of the rigid body is shown as a function of
time. The width of the body is taken as $\width=0$ (this has no influence on the results) 
so that the left and right reference domains meet at $x=0$.
The results from the first-order accurate scheme show predictably smeared out
solution profiles. The results from the second-order accurate scheme are in very good agreement
with the exact solution even at this coarse
resolution. Figure~\ref{fig:sol_mass1} also presents results from a grid convergence
study and shows the max-norm errors for this problem using the two algorithms. The
predicted convergence rates are convincingly demonstrated for both velocity and
stress.

{\em Remark:} For this case, one could also use the new projection scheme with a forward Euler rigid body integrator. Simulation results for this case reveal no unexpected behavior.

{\em Remark:} For this case, traditional coupling techniques without projection would not experience exponential blowup for the considered grids and time steps. Numerical results using the traditional scheme with $\alpha=0$ for this case are nearly identical to those in Figure~\ref{fig:sol_mass1} and are therefore not shown.

\subsection{Difficult case: very light rigid body with mass $10^{-6}$}
\label{sec:intermediate}
We now consider a case where the time-step restriction for the traditional interface algorithm 
is orders of magnitude smaller than the time-step restriction for the new interface projection algorithm.
The time-step restriction for the new projection algorithm 
depends only on the usual CFL time-step restrictions for each fluid domain separately; 
the coupling with the rigid body imposes no new constraint on the time-step 
since the backward Euler and trapezoidal methods are both A-stable.
\begin{figure}
\begin{center}
  \includegraphics[width=\graphWidth]{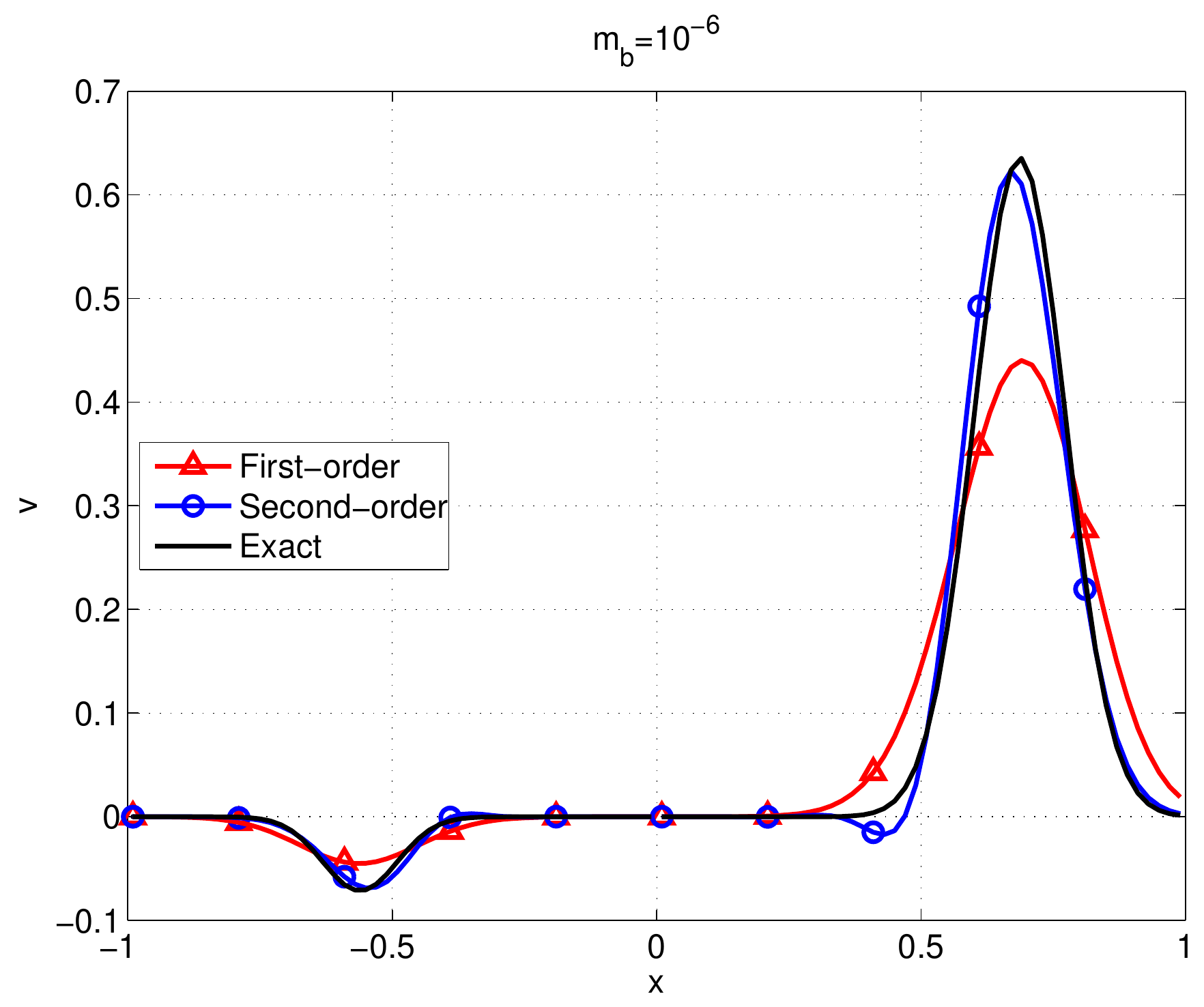}
  \includegraphics[width=\graphWidth]{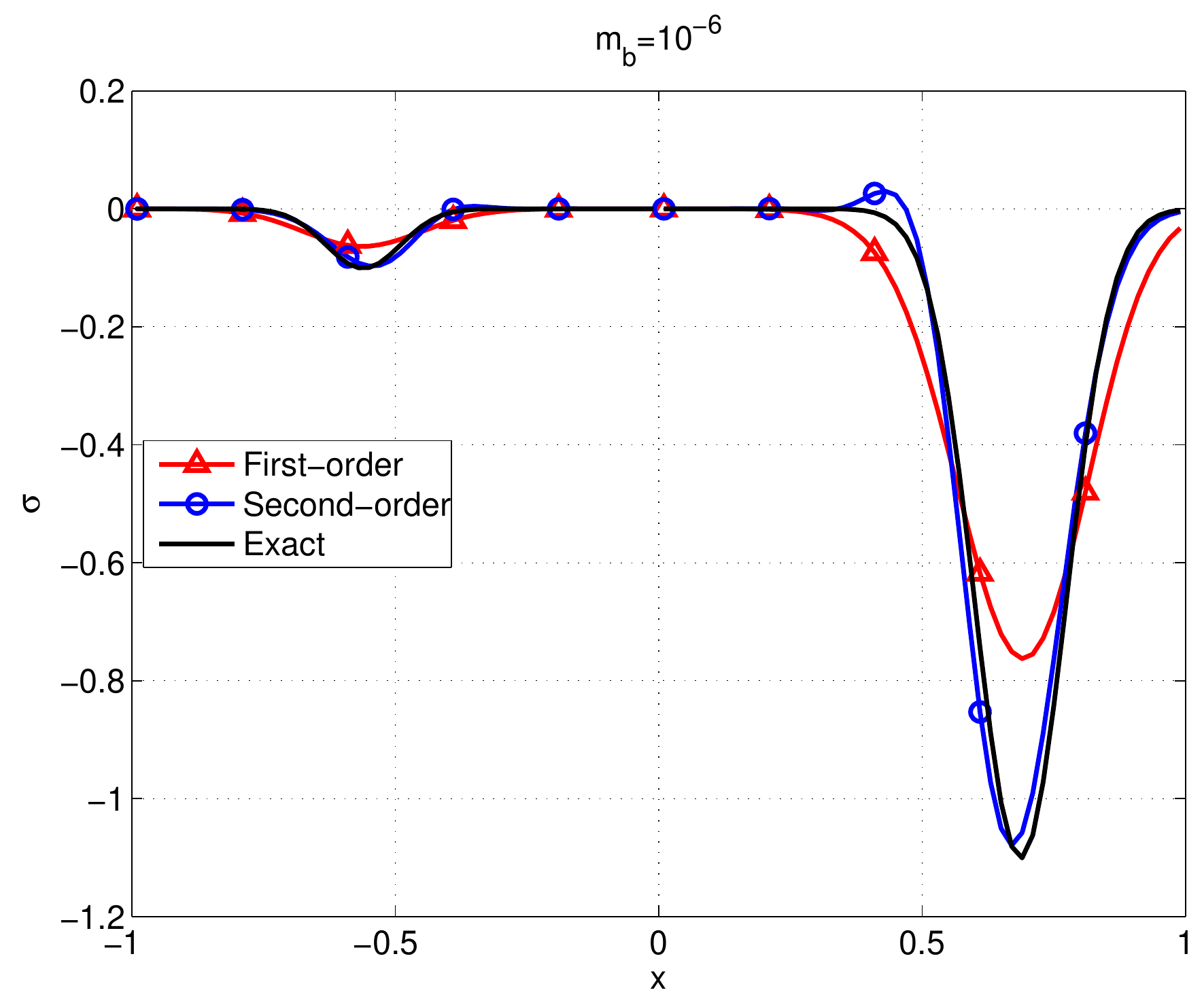}
  \includegraphics[width=\graphWidth]{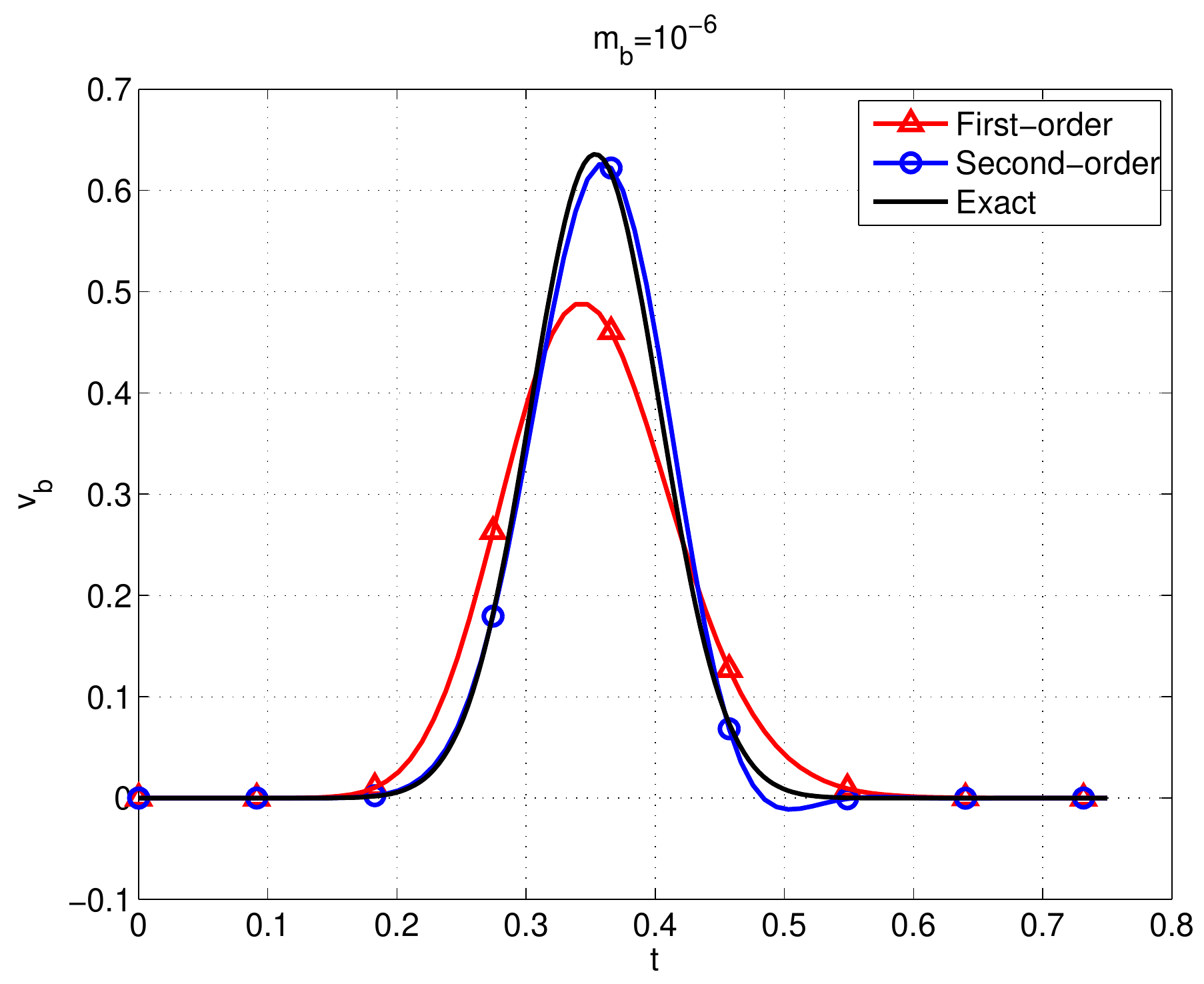}
  \includegraphics[width=\graphWidth]{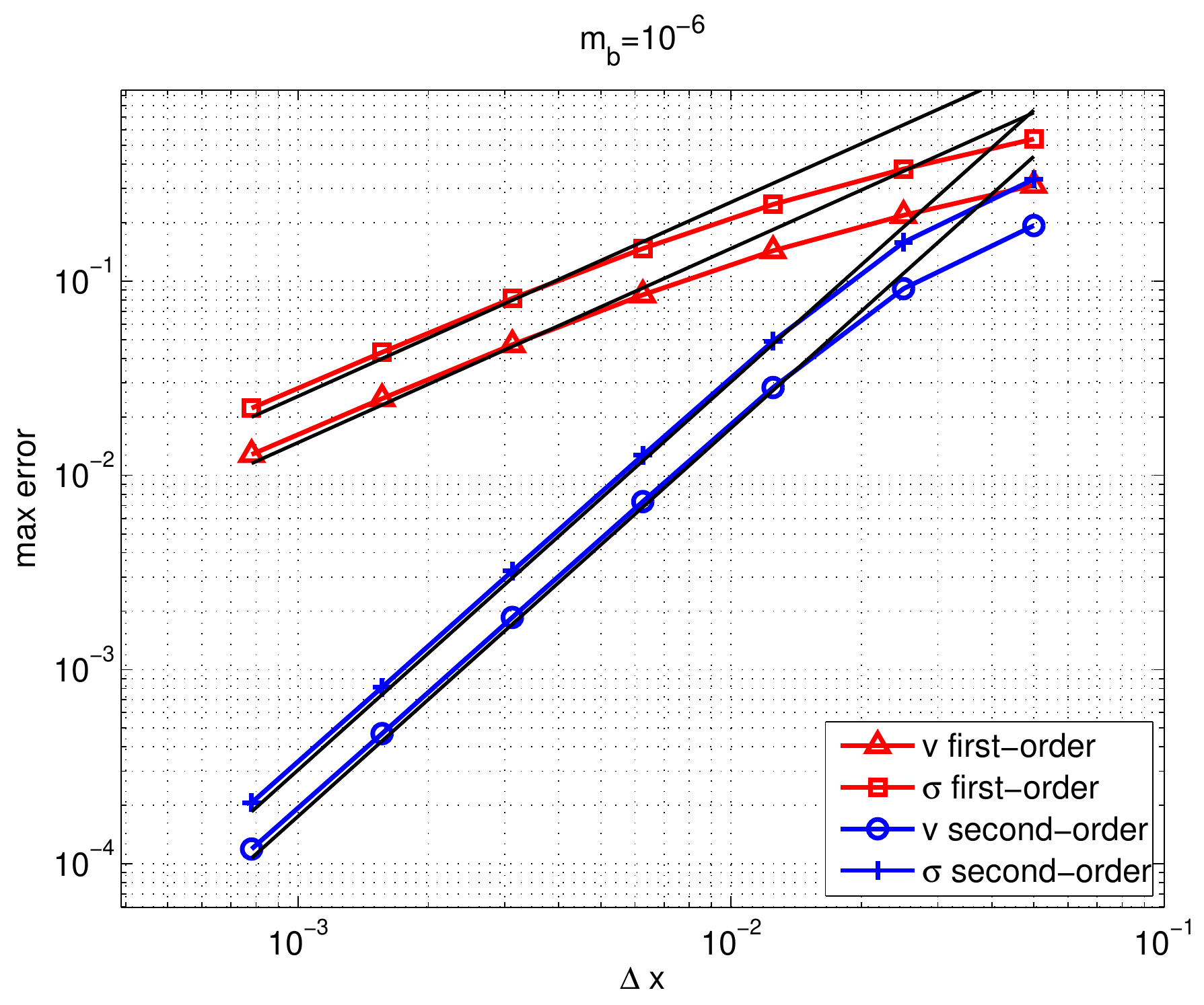}
  \caption{Results for the one-dimensional FSI problem with $\mass=10^{-6}$ for the first- and second-order accurate schemes. 
           Top left: velocity at $t=0.75$. Top-right: stress at $t=0.75$. Bottom left: velocity of the rigid body, $v_b$ versus time. 
           Bottom right: convergence of the max-norms errors. The solutions are plotted in the reference domain $[-1,0]$ for the
           left domain and $[0,1]$ for the right domain with the rigid body of width $\width=0$ located at $x=0$.}
  \label{fig:sol_mass1em6}
\end{center}
\end{figure}
We consider a rigid body with mass $\mass=10^{-6}$ and use the same grid spacings as
before, $\dxL=\dxR=1/50$. Figure~\ref{fig:sol_mass1em6} shows simulation results for this case
using the two new schemes. As the
figure shows, the results from the second-order accurate scheme are, as expected, superior
to those from the first-order accurate scheme.
The lower right graph in Figure~\ref{fig:sol_mass1em6} presents a convergence study. The expected rates
of convergence are again convincingly demonstrated. 

{\em Remark:} For this case, one could instead consider using an explicit rigid
body integrator together with the projection scheme. The rigid
body integration must respect the ODE time-step constraint and so subcycling can
be used. It is straightforward to estimate that for $\dxL=\dxR=1/50$, $14389$ subcycles are
required to obtain stability of a forward Euler integrator. 
The number of subcycles required for stability decreases, however, as $\dt$ decreases.
As a result, for $\dxL=\dxR=1/1280$ (the finest resolution in the
associated convergence studies), only $568$ subcycles are required. A sub-cycling
has been implemented and the results are nearly identical to the results
shown in Figure~\ref{fig:sol_mass1em6}.

{\em Remark:} Had the traditional algorithm with $\alpha=0$ been used, the
entire solution (both fluid domains and the solid domain) would have to be
integrated using a time-step which satisfies a constraint of the form
of~\eqref{eq:trestr}. For the
first-order scheme with Backward Euler rigid body integrator the constraint is
(\ref{eq:trestr}).  For other fluid discretizations and/or rigid body
integrators,the timestep restriction can be determined following the approach used in the proof of Theorem~\ref{thm:traditionalFO}.
Such a time-step
restriction can be quite severe and arises as a result of using a partitioned
algorithm without the interface projection.

\subsection{Rigid body with zero mass}

\begin{figure}
\begin{center}
  \includegraphics[width=\graphWidth]{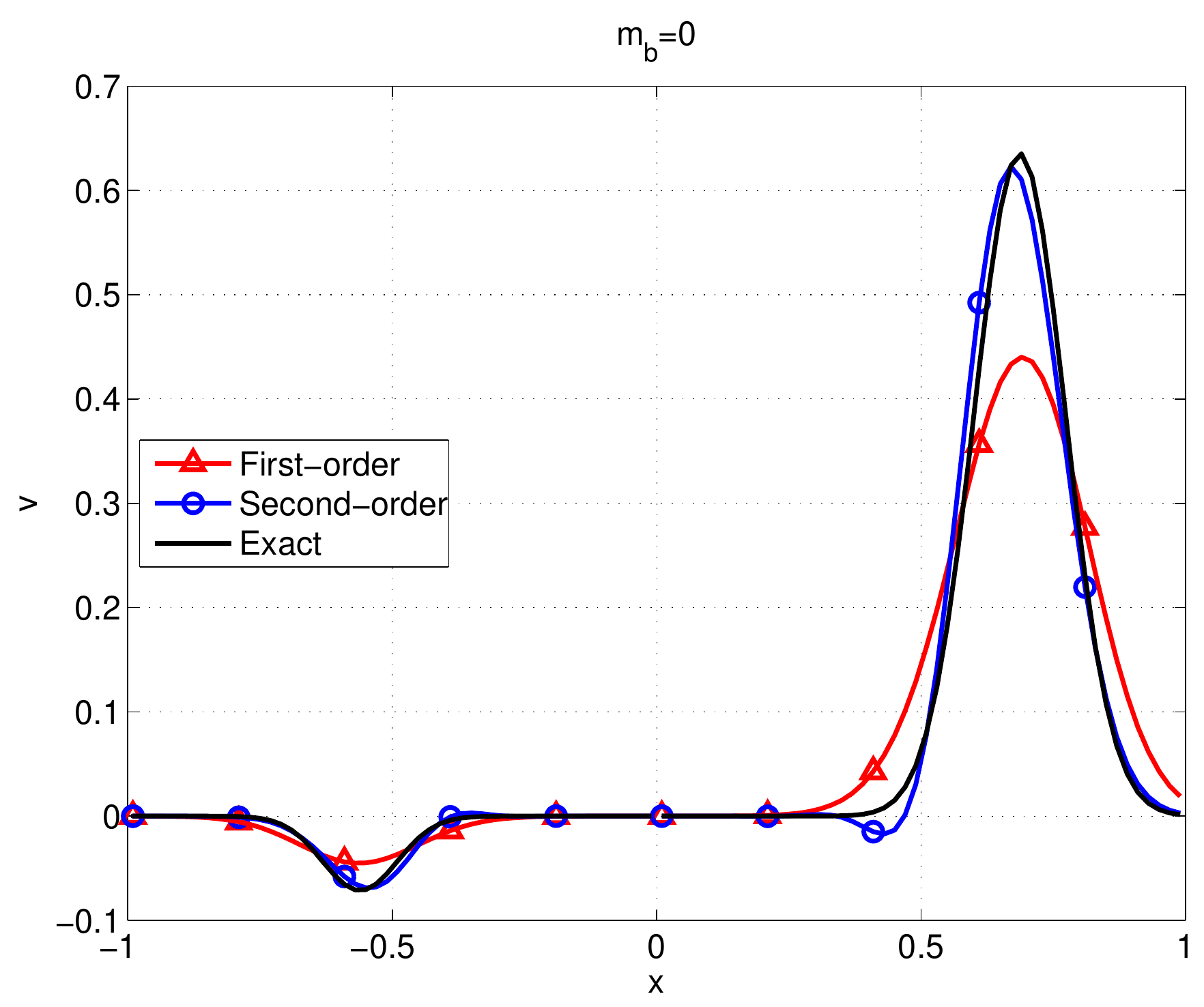}
  \includegraphics[width=\graphWidth]{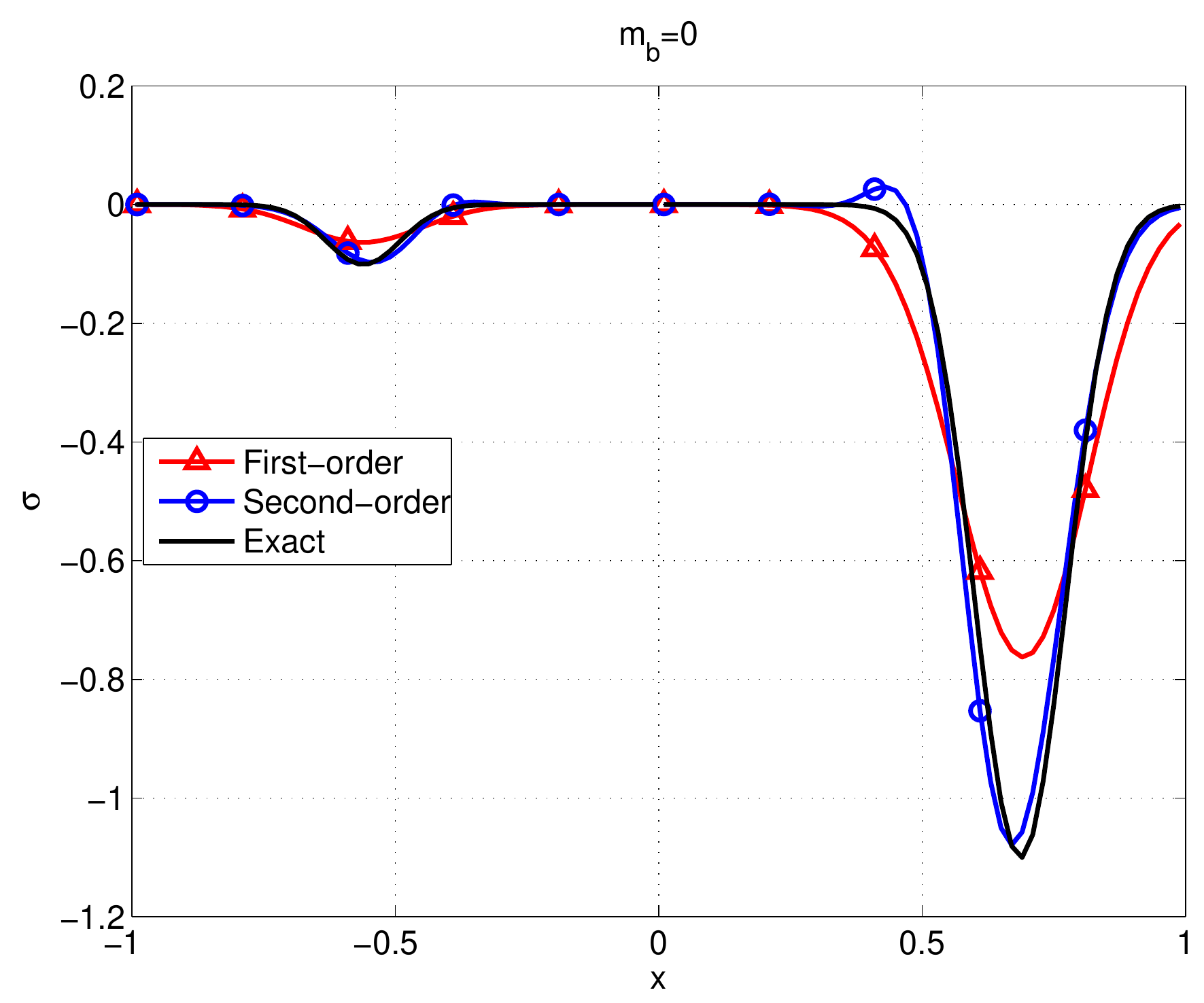}
  \includegraphics[width=\graphWidth]{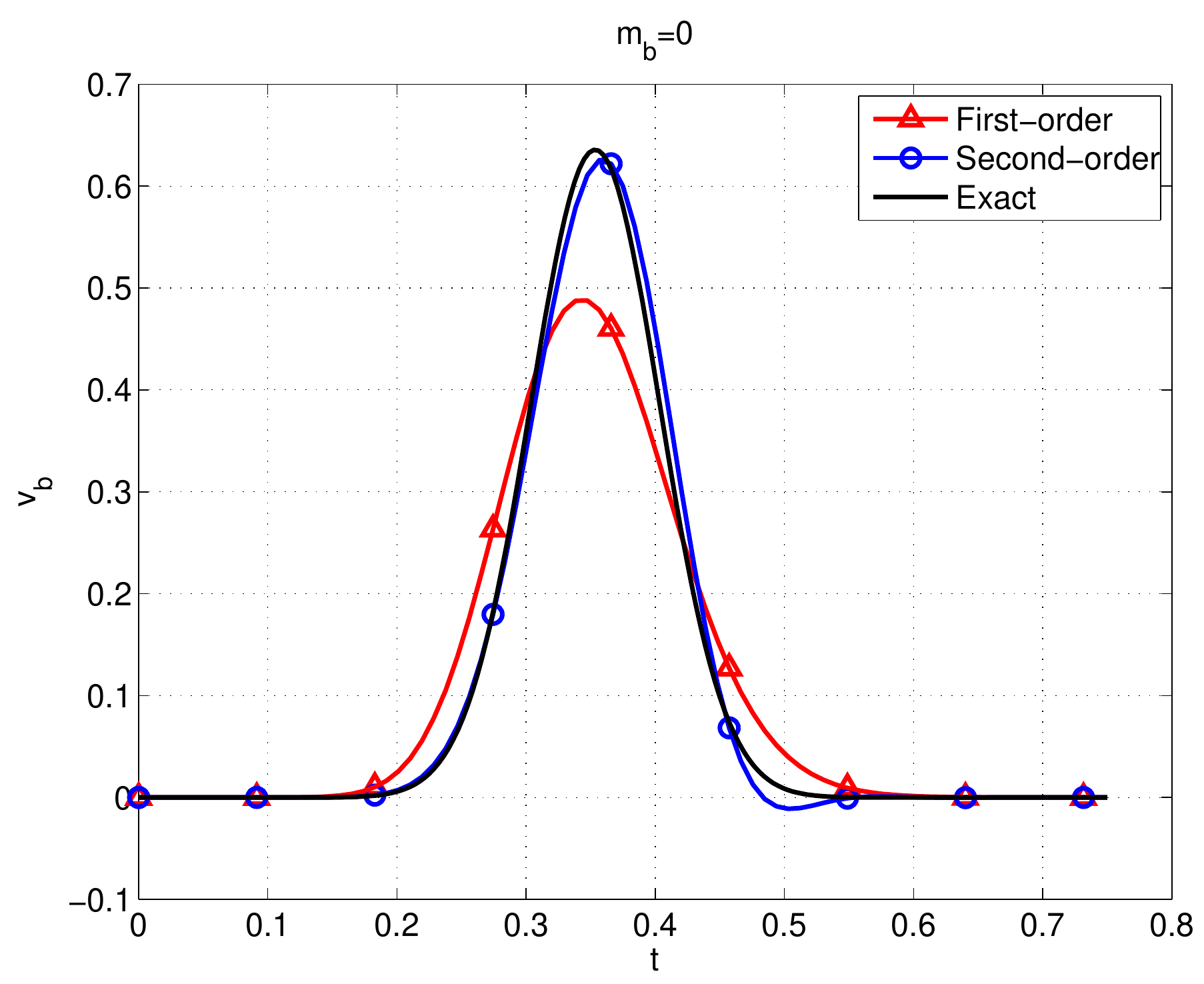}
  \includegraphics[width=\graphWidth]{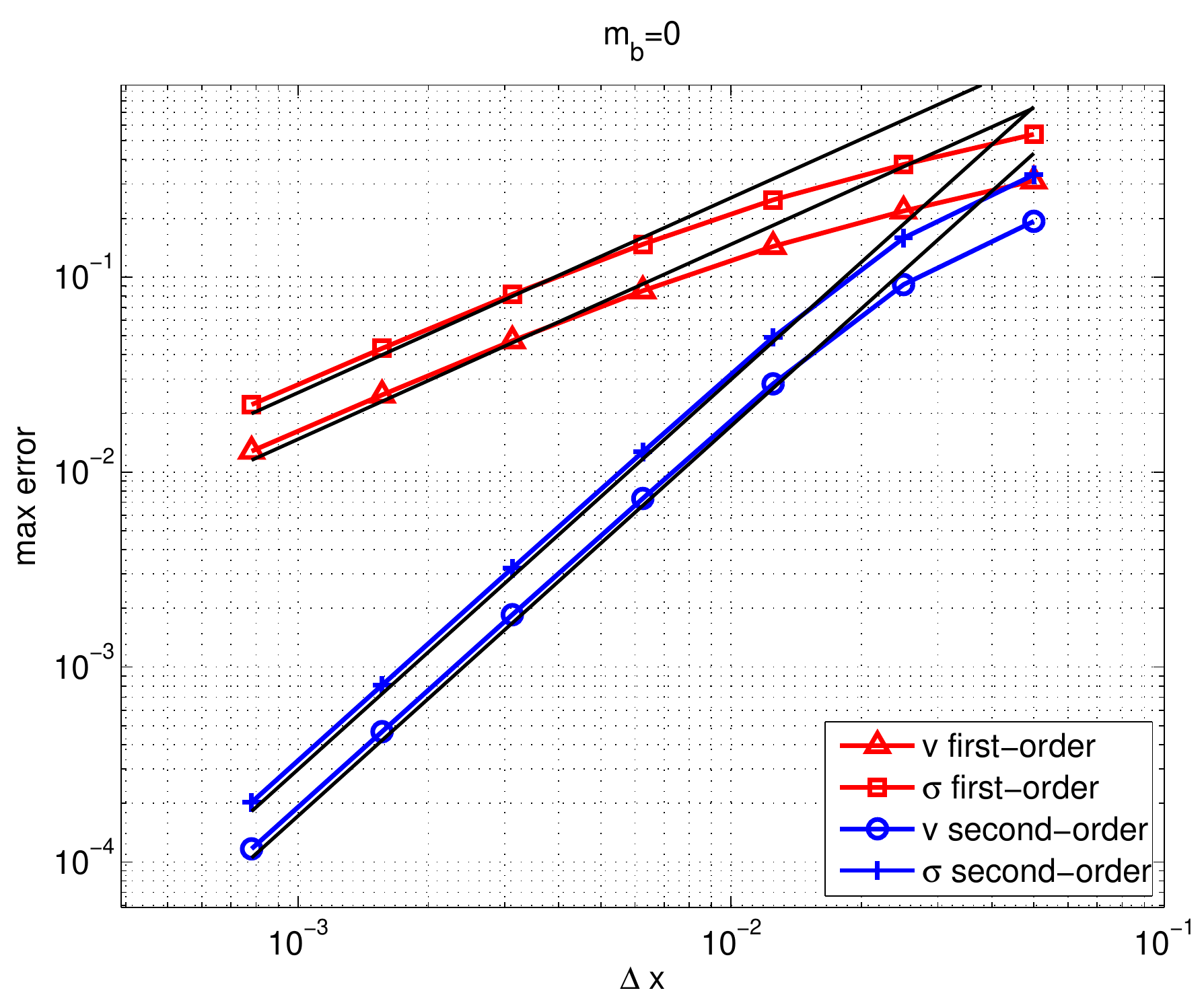}
  \caption{Results for the one-dimensional FSI problem with $\mass=0$ for the first- and second-order accurate schemes. 
           Top left: velocity at $t=0.75$. Top-right: stress at $t=0.75$. Bottom left: velocity of the rigid body, $v_b$ versus time. 
           Bottom right: convergence of the max-norms errors. The solutions are plotted in the reference domain $[-1,0]$ for the
           left domain and $[0,1]$ for the right domain with the rigid body of width $\width=0$ located at $x=0$.}
  \label{fig:sol_mass0}
\end{center}
\end{figure}

The new projection based FSI scheme remains well defined even when the 
mass of the rigid body is zero. This is apparent from the update equation
for the velocity of the rigid body, equation~\eqref{eq:eqforv} for
the first-order accurate scheme or equation~\eqref{eq:RBTrap} for the second-order
accurate scheme. The traditional partitioned algorithm is not well-defined for this case, since it would
require division by $m_b$, and so is not an option.
Figure~\ref{fig:sol_mass0} shows results for the first-order upwind
method with backward Euler rigid body integration, and the second-order upwind
method with trapezoidal rigid body integration. The exact solution is computed
for $\mass=0$ which yields essentially the same solution used for the two domain
model problem in~\cite{banks11a}, i.e. the solution behaves as if the
rigid body were not present. Figure~\ref{fig:sol_mass0} shows convergence
results where again the predicted rates of convergence are demonstrated. No
significant differences from the mass $m_b=10^{-6}$ case in
Section~\ref{sec:intermediate} are observed.

{\em Remark:} For this case it is impossible to satisfy the ODE stability constraint without using an A-stable integrator and so explicit rigid body integration with subcycling is not an option. Put another way, the explicit algorithm would require
an infinite number of subcycles.

\newcommand{\vr}{v_r}
\newcommand{\vvr}{\vv_r}
\newcommand{\pr}{p_r}
\section{The multi-dimensional interface approximation and added-mass matrices} \label{sec:multiDimensionalInterface}

In this section we extend the added-mass algorithm to multiple space dimensions.
Formula~\eqref{eq:stressOnBodyFromCharsLocal} relates the pressure and velocity of a point on the body to the nearby pressure
and velocity in the fluid. This relation is in a form amenable to multidimensional generalization.
Let $\rv=\rv(t)$ denote a point on the surface of the body $B$, and $\nv=\nv(\rv,t)$ the outward normal to the body, then 
in multiple space dimensions~\eqref{eq:stressOnBodyFromCharsLocal} becomes
\begin{align*}
    -p (\rv(t),t)\, \nv &= - p(\rv+,t-)\,\nv + z(\rv+,t-)\, \big[ \nv^T\big( \vv(\rv+,t-)-\vv(\rv,t)\big)\big]\,\nv,
\end{align*}
where $\vv(\rv,t)=\dot\rv$ is the velocity of the point.
To clarify the notation let $\pr=p (\rv,t)$ and $\vvr=\vv(\rv,t)$ denote the pressure and velocity on the body at point $\rv=\rv(t)$,
 and $z_f=z(\rv+,t-)$, $p_f=p(\rv+,t-)$ and $\vv_f=\vv(\rv+,t-)$ denote the impedance, pressure and velocity at the adjacent point in 
the fluid. This gives
\begin{align*}
    -\pr \nv &= - p_f \nv + z_f\, \big[\nv^T\big( \vv_f - \vvr \big)\big]\,\nv.
\end{align*}
Using equation~\ref{eq:rDot} for $\vvr=\dot{\rv}$ it follows that
\begin{align}
    -\pr \nv &= - p_f \nv + z_f\,\big[ \nv^T\big( \vv_f -\vvcm + Y\omegav \big)\big]\,\nv. 
            \label{eq:multidimensionalProjection} 
\end{align}
The {\bf key point} of~\eqref{eq:multidimensionalProjection} is that
is shows how the force exerted by the fluid on the body, $\fv_s=-\pr\nv$,
depends on the velocity of the center of mass, $\vvcm$, and the angular velocity, $\omegav$, of the body. 
Substituting~\eqref{eq:multidimensionalProjection} into the expressions~\eqref{eq:bodyForce}-\eqref{eq:bodyTorque} 
for $\Forcev$ and $\Torquev$ gives
\begin{align*}
    \Forcev & = \int_{\partial B} z_f \nv\nv^T( -\vvcm + Y \omegav )~ ds 
                + \int_{\partial B} -p_f\nv + z_f (\nv^T\vv_f)\nv ~ds + \fv_b, \\
    \Torquev & = \int_{\partial B} z_f  Y\nv\nv^T( -\vvcm + Y \omegav )~ds 
                + \int_{\partial B} \yv\times \big(-p_f\nv+ z_f (\nv^T\vv_f)\nv\big)\,ds + \gv_b.
\end{align*}
We write $\Forcev$ and $\Torquev$ in the form
\begin{align*}
    \Forcev & = -\Avv \vvcm - \Avw \omegav + \Forcevtilde ,\\ 
    \Torquev & =  -\Awv \vvcm - \Aww \omegav + \Torquevtilde
\end{align*}
where the {\em added-mass} matrices $A^{ij}$ are given by (using $Y^T=-Y$, where $Y$ is defined by~\eqref{eq:Ymatrix}),
\begin{alignat}{3}
    \Avv &= \int_{\partial B} z_f \nv \nv^T~ ds,    \qquad &  \Avw &= \int_{\partial B} z_f \nv (Y\nv)^T ~ds,  \label{eq:AddedMassMatrixI} \\
    \Awv &=\int_{\partial B} z_f (Y\nv)\nv^T ~ds \qquad &  \Aww &= \int_{\partial B} z_f Y\nv (Y\nv)^T ~ds , \label{eq:AddedMassMatrixII}
\end{alignat}
and $\Forcevtilde$ and $\Torquevtilde$ are given by 
\begin{align}
     \Forcevtilde & = \int_{\partial B} -p_f\nv + z_f (\nv^T\vv_f)\nv \,ds + \fv_b,           \label{eq:partialForce} \\
    \Torquevtilde & =  \int_{\partial B} \yv\times \big(-p_f\nv+ z_f (\nv^T\vv_f)\nv\big)\,ds + \gv_b. \label{eq:partialTorque}
\end{align}
Note that $\Avv$ and $\Aww$ are symmetric and positive semi-definite while $(\Avw)^T = \Awv$. 
Let $\Ma \in \Real^{6\times 6}$ denote the composite added mass matrix (tensor), 
\begin{align*}
   \Ma &= \begin{bmatrix}
    \Avv & \Avw \\
    \Awv & \Aww
\end{bmatrix} .
\end{align*}
This matrix is symmetric and positive semi-definite since for any vector $\wv = [\av~ \bv]^T \in \Real^6$, $\av\in \Real^3$, $\bv\in \Real^3$
\begin{align*}
  \wv^T \Ma \wv &=  [ \av~  \bv ]
  \begin{bmatrix}
    \Avv & \Avw \\
    \Awv & \Aww
   \end{bmatrix}
      \begin{bmatrix} \av \\ \bv \end{bmatrix} = \int z_f \Big( \| \nv^T\av\|^2  + 2(\nv^T\av)(Y\nv)^T \bv + \| (Y\nv)^T \bv\|^2 \Big) ds , \\
       &= \int z_f \Big( (\nv^T\av) + (Y\nv)^T\bv\Big)^2 \,ds.
\end{align*}
The rigid body equations of motion~\eqref{eq:centerOfMassPositionEquation}-\eqref{eq:axesOfInertiaEquation} can now be written in the form 
\begin{align}
  \begin{bmatrix}
      I & 0 & 0 & 0 \\
      0 & \mrb I & 0 & 0  \\
      0 &  0  & A & 0 \\
      0 & 0 & 0 & I  
  \end{bmatrix}
  \begin{bmatrix} \dotxvcm \\ \dotvvcm \\ \dot\omegav \\ \dot E \end{bmatrix}
 +
  \begin{bmatrix}
    0 & -I & 0 & 0 \\
    0 & \Avv & \Avw & 0  \\
    0 & \Awv & \Aww +\OmegaMatrix A & 0 \\
    0 & 0  & 0 & -\OmegaMatrix 
   \end{bmatrix}
  \begin{bmatrix} \xvcm \\ \vvcm \\ \omegav \\ E \end{bmatrix}
 = 
 \begin{bmatrix} 0 \\ \Forcevtilde \\ \Torquevtilde \\ 0 \end{bmatrix} . \label{eq:addedMassRigidBodyEquations}
\end{align}
We will refer to equations~\eqref{eq:addedMassRigidBodyEquations} as the {\em added-mass} Newton-Euler equations.

{\em Remark:} By solving equations~\eqref{eq:addedMassRigidBodyEquations} with
an implicit time-stepping scheme that treats the added-mass terms implicitly,
the rigid body equations can be advanced with a {\em large} time step even as
$\mrb$ and $A$ approach zero, {\em provided} $\Ma$ is nonsingular. This is
described in more detail in Section~\ref{sec:multidimensionalTimeStepping}.

{\em Remark:} In \ref{sec:addedMassMatrices} we present the form of the added-mass matrices for
some common body shapes.

\newcommand{\InteriorIndex}{\Ic_I}
\newcommand{\BodyIndex}{\Ic_B}
\newcommand{\GhostIndex}{\Ic_G}
\section{The multi-dimensional time-stepping algorithm} \label{sec:multidimensionalTimeStepping}

We make use of overlapping grids to treat multi-dimensional problems with moving rigid bodies.
Narrow boundary fitted grids lie next to the bodies and these move with the bodies 
(see the examples in Section~\ref{sec:numericalResults}). One or more stationary
background grids generally cover most of the domain. This approach results in high-quality grids
even as bodies undergo large motions. 
The time-stepping algorithm we use for FSI problems with rigid bodies is described in detail in~\cite{henshaw06},
while that for FSI problems with deforming solids is described in~\cite{banks12a}. In~\cite{henshaw06} the Newton-Euler 
equations for the rigid bodies are solved using a Leap-frog predictor step followed by a trapezoidal rule
corrector step. 

\newcommand{\StageI}{Predict} 
\newcommand{\StageII}{Body}
\newcommand{\StageIII}{Correct}
\newcommand{\StageIV}{Ghost}
\newcommand{\Gvdot}{\dot{\Gv}}

\begin{figure}[hbt]\tableFont 
\begin{center}
\newcommand{\strutt}{\rule{0pt}{10pt}}
\begin{tabular}{|r|c|c|c|} \hline 
\multicolumn{4}{|c|}{The FSI time stepping algorithm}  \strutt \\ \hline
~Stage~  & Condition         & Type                            &   Assigns               \strutt  \\ \hline 
 \StageI(a)   &  Predict body motion, moving grid  & extrapolation & $\xvcm^p,\vvcm^p,\omegav^p,\Ev^p, \Gv_\iv^p$ ~~   \strutt\\
 \StageI(b)   &  Advance fluid $\wv_\iv^n$, $\wv_\iv^p$,~    & PDE         & $\wv_\iv^n$,~$\iv\in\InteriorIndex$,~~  $\wv_\iv^p$,~$\iv\in\BodyIndex$ \strutt\\
 \StageII(a)  &  Compute added mass terms~\eqref{eq:AddedMassMatrixI}-\eqref{eq:partialTorque}  &             &  $A_{11}^p$, $A_{12}^p$, $A_{21}^p$, $A_{22}^p$, $\Fvtilde^p$, $\Gvtilde^p$                                             \strutt \\
 \StageII(b)  &  Advance rigid body~\eqref{eq:addedMassRigidBodyEquations}  & ODEs        &  $\xvcm^{n},\vvcm^{n},\omegav^{n},\Ev^{n}$  \strutt \\
\StageIII(a)  &  Project fluid on body~\eqref{eq:velocityProjection}-\eqref{eq:rhoProjection} & ~projection~ &  $\vv_{\iv}^{n}$, $p_\iv^{n}$, $\rho_\iv^{n}$, ~~$\iv\in\BodyIndex$     \strutt \\
 \StageIII(b) &  Correct moving grid   & projection & $\Gv_\iv^{n}$  \strutt\\
 \StageIV    & Assign fluid ghost values  & PDE, extrapolation &  $\wv_\iv^{n}$, ~~$\iv\in\GhostIndex$       \strutt \\  
   \hline 
\end{tabular}
\caption{The FSI time stepping algorithm for advancing the states of the fluid and rigid body. 
}
\label{tab:interfaceStages}
\end{center}
\end{figure}

For the new interface algorithm developed here, we choose a time-stepping method 
for the Newton-Euler equations~\eqref{eq:addedMassRigidBodyEquations} that
treats the added-mass terms implicitly so that the scheme is well-defined even 
in the limit of zero mass and/or moments of inertia. We use
diagonally implicit Runge-Kutta (DIRK) schemes for this
purpose~\cite{alexander77}.  DIRK schemes have very nice stability and accuracy
properties. The one-stage, first-order accurate DIRK scheme, which we denote by
DIRK1, is just the backward-Euler scheme. For the numerical results in
section~\ref{sec:numericalResults} we will use a two-stage third-order accurate
(A-stable) scheme, denoted by DIRK3, due to Crouzeiux (see~\cite{alexander77}
(2.2)).  In each stage of the DIRK scheme we solve an implicit approximation
to~\eqref{eq:addedMassRigidBodyEquations} by Newton's method.

%
%
\color{black}
\begin{figure}[hbt]
\begin{center}
\begin{tikzpicture}[scale=1]
\useasboundingbox (-5,-.2) rectangle (5.,.5);
  \filldraw[fill=gray!20,line width=2pt] (-5,-.5) -- (0,-.5) -- (0,.5) -- (-5,.5);
  \draw (-2,0) node {solid};
  \filldraw[fill=blue!10,line width=2pt] ( 5,-.5) -- (0,-.5) -- (0,.5) -- (5,.5);
  \draw (2,0) node {fluid};
  \draw[color=green,line width=2pt] (0,-.5) -- (0,.5);
\end{tikzpicture}
\end{center}
\begin{center}
\begin{tikzpicture}[scale=4]
\useasboundingbox (-1.5,.2) rectangle (1.5,0);  
%
\begin{scope}
\def\ya{.08}
\draw[-,thick,black,yshift=.025 cm] 
   (-.5,0) -- (1.25,0) 
   (1.1,-.1) node {$x$}
   \foreach \x in {-.5,-.25,...,1.}{ (\x,-.025) -- (\x,.025) }
   (-.5,\ya) node {$\wv^n_{-2}$}
   (-.25,\ya) node {$\wv^n_{-1}$}
   (  .0,\ya) node {$\wv^n_{0}$} 
   ( .25,\ya) node {$\wv^n_{1}$} 
   ( .50,\ya) node {$\wv^n_{2}$} 
   ( .75,\ya) node {$\ldots$}; 
\end{scope}
\end{tikzpicture}
\end{center}
\caption{
The fluid grids for two-dimensional problems have a grid point aligned with the boundary of
the rigid body. The solution on the boundary is $\wv^n_{0}$, while $\wv^n_{-2}$ and $\wv^n_{-1}$  denote
the values on the ghost points. For clarity, only one grid line is shown in the direction
normal to the boundary.
}
\label{fig:twoDimensionalGridCartoon}
\end{figure}
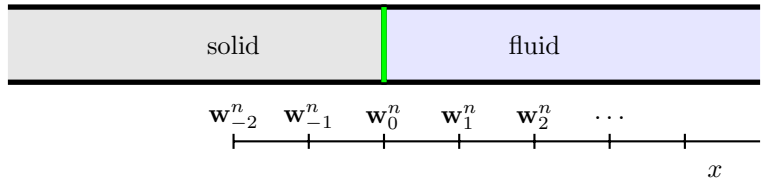

The FSI time stepping algorithm for advancing the fluid and rigid body is
outlined in Figure~\ref{tab:interfaceStages}.  
In a slight difference from the grid arrangement used for the analysis
in one-dimension as illustrated in Fig.\ref{fig:modelFig}, the two-dimensional grids have a grid point aligned
with the boundary of the body as shown in Fig.\ref{fig:twoDimensionalGridCartoon}. 
Let $\wv_\iv^n=(\rho_\iv^n,\vv_\iv^n,p_\iv^n)$ denote the discrete solution in space
and time for the state of the fluid at time $t^n$, where $\iv$ is a
multi-index. Let $(\xvcm^{n},\vvcm^{n},\omegav^{n},\Ev^{n})$ denote the discrete
approximation in time to the state of the rigid body. Let $\xv_\iv^n=\Gv_\iv^n$
denote the (moving) grid points on the fluid grid that lies next to the body and
$\dot\Gv_\iv^n$ the grid velocity (the fluid domain will actually be discretized
with multiple overlapping grids but for clarity we ignore these other grids in
the current discussion).

Suppose that we are given the full state of the discrete solution at time
$t^{n-1}$ and wish to determine the state at the next time step $t^n$.  In the
first stage of the time stepping algorithm, predicted values are obtained for the state of the
solid body at the new time, $(\xvcm^p,\vvcm^p,\omegav^p,\Ev^p)$. These values
can be obtained either from the Newton-Euler equations of motion or using 
extrapolation in time (for a second order accurate scheme we extrapolate using the current
level and two previous 
time levels~\footnote{To extrapolate at $t=0$ we would need the state of the body at $2$ previous times. 
Currently these values must be supplied when the problem is being setup. More generally one could implement a
predictor-corrector style time-stepping algorithm at startup that would obviate the need for negative time state values.}
). 
 From the predicted state of the body we
can obtain predicted values for the grid location, $\Gv_\iv^p$, and grid velocity, $\dot\Gv_\iv^p$; 
these values are needed to advance the fluid state.  Note that the grids move as a rigid body and do not deform.
In the second stage of the time
stepping algorithm we obtain the new values of the fluid state
$\wv_\iv^n=(\rho_\iv^n,\vv_\iv^n,p_\iv^n)$ at interior grid points,
$\iv\in\InteriorIndex$, and predicted values, $\wv_\iv^p=(\rho_\iv^p,\vv_\iv^p,p_\iv^p)$,
at points on the body surface, $\iv\in\BodyIndex$. 
These values are obtained using our high-order Godunov-based
upwind scheme~\cite{henshaw06}. At this stage no boundary conditions have been applied on the fluid at the body surface.
Given the predicted fluid states $\wv_\iv^p$ we can compute the partial body
forces~\eqref{eq:partialForce}-\eqref{eq:partialTorque} and the added mass
matrices~\eqref{eq:AddedMassMatrixI}-\eqref{eq:AddedMassMatrixII} using numerical
integration over the surface of the body.
Note that it is straightforward to compute the coefficients of the added mass matrices using numerical 
quadrature even for variable impedance and bodies of arbitrary shape.
 We then solve the added-mass Newton-Euler 
equations~\eqref{eq:addedMassRigidBodyEquations} (e.g. with a DIRK scheme) to determine
the corrected state of the rigid-body at the new time,
$(\xvcm^{n},\vvcm^{n},\omegav^{n},\Ev^{n})$.
 The predicted state of the fluid on the solid body is then
corrected by setting the fluid velocity equal to the 
(local) body velocity and the fluid pressure to equal its
projected value,
\begin{alignat}{3}
   \vv_\iv^{n} &= \vv_{b,\iv}^{n}, \quad&& \iv\in\BodyIndex, \label{eq:velocityProjection} \\
   -p_\iv^{n} &= -p_\iv^{p} + z^p\nv^T\big( \vv_\iv^p - \vv_{b,\iv}^{n}  \big), \quad&& \iv\in\BodyIndex . \label{eq:pressProjection} 
\end{alignat}
Here the local body velocity is $\vv_{b,\iv}=\vvcm^{n}+W^n(\rv^n_\iv-\xvcm^{n})$, where $\rv_\iv^n$ denotes
the location of a point on the body surface, and where $W^n$ is defined from $\omegav^n$ using~\eqref{eq:Wdef}.
After projecting the pressure, the density is corrected using 
\begin{equation}
 \rho_\iv^{n} = \rho_\iv^p \Big( p_\iv^{n}/p_\iv^{p}\Big)^{1/\gamma}, \qquad \iv\in\BodyIndex , \label{eq:rhoProjection}
\end{equation}
which ensures that the entropy of the predicted state equals that of the corrected state.
The fluid grid, $\Gv_\iv^n$, and grid velocity, $\dot{\Gv}_\iv^n$, at the new time are
corrected from the predicted values to match the new state of the rigid body.
In the final stage of the time stepping algorithm, the ghost values of fluid state that lie
adjacent to the body surface are updated using the appropriate
boundary conditions and compatibility conditions, see~\cite{henshaw06,banks12a} for more details.

\section{Numerical results in two space dimensions} \label{sec:numericalResults}

In this section we present numerical results in two-dimensions that demonstrate
the accuracy and stability of the added-mass interface algorithm when applied to
light rigid bodies. A pressure driven light piston problem is used to examine
the accuracy of the two-dimensional added-mass algorithm for an FSI problem with
an analytic solution. A smoothly accelerated light rigid body in the shape of an ellipse is used to
evaluate the scheme for a two-dimensional problem that includes the rotational
added-mass terms. Solutions using the new added-mass algorithm are compared to the old algorithm, which is necessarily run at a
small CFL number to avoid exponential blowup. Although the exact solution to this problem is not known,
a posteriori estimates of the errors are determined from solutions on a sequence of grids
of increasing resolution. In two final examples we simulate the impingement of Mach 2 shocks
on zero mass rigid bodies in 2D. We include two cases, the first an ellipse and the second
a complex body with appendages that we call a {\em starfish}. These cases demonstrate the 
robustness of the added-mass algorithm for very difficult situations. Solutions of these shock driven
ellipse problem are computed at varying grid resolutions and compared.  These
results include computations that use dynamic adaptive mesh refinement (AMR).

\subsection{Pressure driven light piston} \label{sec:pressureDrivenLightPiston}

%
%
{
\newcommand{\fsi}{G}
\newcommand{\figWidth}{5.13cm}
\newcommand{\trimfig}[2]{\trimPlot{#1}{#2}{0.05}{0.0}{.18}{.18}}
\begin{figure}[hbt]
\begin{center}
\begin{tikzpicture}[scale=4]
\useasboundingbox (-.5,0.0) rectangle (2.5,1.);  
\draw[->] (-1.,0) -- (1.,0) node[anchor=south] {$x$};
\draw[->] (0,0) -- (0,1) node[anchor=west] {$t$} ;
%
\draw[very thick,green] (0,0) .. controls (0,.5) and (-.5,.9) .. (-.4,.8) node[anchor=south,yshift=-2pt,blue] {$x=\fsi(t)$};
\draw[blue] (-.08,.3) -- (.5,.75) node[anchor=south west,yshift=-4pt] {$C^+$};
\fill[blue] (.21,.525) circle (.5 pt);
\draw[blue] (.21,.525) node[anchor=west] {~$(x,t)$};
\fill[blue] (-.068,.32) circle (.5 pt) node[anchor=east,blue] {~$(\fsi(\tau),\tau)$};
\draw[blue] (0,0) -- (.6,.5) node[anchor=west] {$x=a_0 t$};
\draw[blue] (.8,.25) node {$\begin{bmatrix} \rho_0 \\ v_0 \\ p_0 \end{bmatrix}$};
\begin{scope}[yshift=-1pt]
\newcommand{\pcolour}{black}
\newcommand{\xap}{-.5}\newcommand{\xbp}{0.}
\newcommand{\yap}{-.1}\newcommand{\ybp}{0.0}
\draw[thick,\pcolour,fill=\pcolour,opacity=0.25] (\xap,\yap) rectangle (\xbp,\ybp);
\draw[thick,\pcolour] (\xap,\yap) rectangle (\xbp,\ybp);
\draw[thick,black] (-.25,-.05) node {piston};
\draw[thick,->,xshift=-0.5pt] (-.65,-.05) node[anchor=east] {$f_b$} -- (\xap,-.05) ;
\newcommand{\fcolour}{blue}
\newcommand{\xaf}{0.0}\newcommand{\xbf}{.9}
\newcommand{\yaf}{-.1}\newcommand{\ybf}{0.0}
\draw[thick,\fcolour,fill=\fcolour,opacity=0.25] (\xaf,\yaf) rectangle (\xbf,\ybf);
\draw[thick,blue] (.45,-.05) node {fluid};
\end{scope}
\begin{scope}[xshift=1.4cm,yshift=-.05cm]
  \draw (0,0) node[anchor=south west,xshift=-8pt,yshift=+0pt] {\trimfig{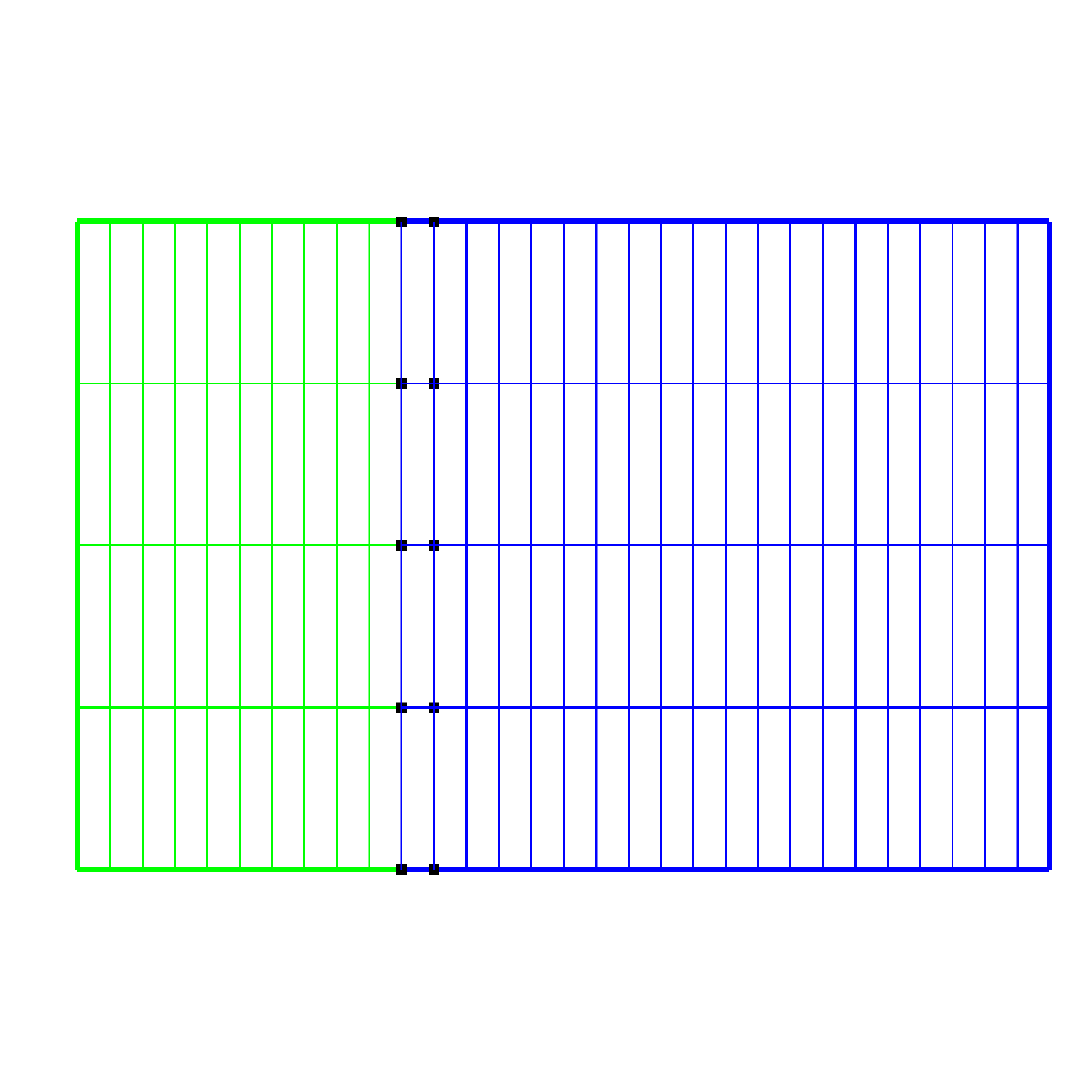}{\figWidth}};
  \draw (.2,.75) node[anchor=south west, draw,fill=white, rounded corners,thick] {\scriptsize piston face};
  \draw[->,thick,black] (.2,.75) -- (0.0,.5);
  \draw[-,thick,black,yshift=0pt] (0,0.0) -- (1.2,0);
  \draw[-,thick,black,yshift=0pt] (0,-.03) -- (0.0,.03) node[anchor=north,yshift=-6pt] {\scriptsize $0$};
  \draw[-,thick,black,yshift=0pt] (.4,-.03) -- (0.4,.03) node[anchor=north,yshift=-6pt] {\scriptsize $.5$};
  \draw[-,thick,black,yshift=0pt] (.8,-.03) -- (0.8,.03) node[anchor=north,yshift=-6pt] {\scriptsize $1.$};
  \draw[-,thick,black,yshift=0pt] (1.2,-.03) -- (1.2,.03) node[anchor=north,yshift=-6pt] {\scriptsize $1.5$};
  \draw[<-,black,xshift=1.3cm] (0,0.065) -- (0,.40) node[anchor=south] {$\Area$};
  \draw[->,black,xshift=1.3cm] (0,0.535) -- (0,.873);
\end{scope}
\end{tikzpicture}
\end{center}
\caption{Left: the $x$-$t$ diagram for the pressure driven piston problem with a receding piston.
       Right: overlapping grid $\Gp^{(2)}$ for the fluid region at $t=0.0$. The green grid moves with the piston. The blue background
grid does not move. The interpolation points are marked as black dots.} \label{fig:pressureDrivenPistonCartoon}
\end{figure}
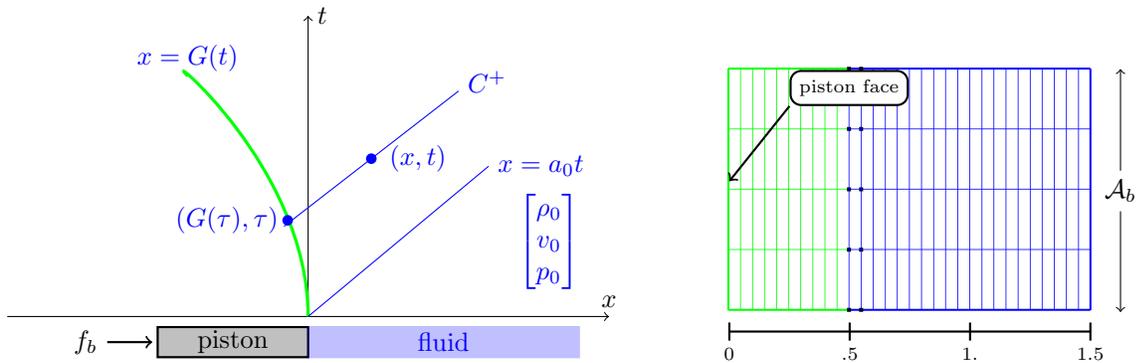
}

The geometry of the one-dimensional pressure driven piston problem is shown in Fig.~\ref{fig:pressureDrivenPistonCartoon}
A compressible fluid occupying the region $x>G(t)$ lies adjacent to a piston of mass $\mrb$ and cross-sectional
area $\Area$.  The
face of the piston that lies next to the fluid follows the curve $x=G(t)$ as time evolves. A body force $f_b(t)$ also
acts on the piston. The exact solution to this problem can be determined for a fluid that is initially at rest and the form
of this solution is given in~\cite{henshaw06}. When $f_b(t)=0$, the exact solution can be determined explicitly. For 
general $f_b(t)$, the case considered here, the exact solution can be accurately approximated by numerical integration of the appropriate
ordinary differential equations.

\newcommand{\pistonTable}{\tableFont
\begin{tabular}{|l|c|c|c|c|c|c|c|} \hline 
Grid  & $h_j$ &  $\eem_\rho$  & r & $\eem_u$  & r & $\eem_T$  & r \\ \hline 
$\Gp^{(8)}$  & 1/80  &  \num{6.3}{-5} &      & \num{1.2}{-4} &      & \num{3.1}{-5} &       \\ \hline
$\Gp^{(16)}$ & 1/160 &  \num{1.8}{-5} &  3.5 & \num{3.3}{-5} &  3.7 & \num{8.8}{-6} &  3.5  \\ \hline
$\Gp^{(32)}$ & 1/320 &  \num{4.2}{-6} &  4.2 & \num{8.5}{-6} &  3.9 & \num{2.2}{-6} &  3.9  \\ \hline
    rate     &       &        $1.95$  &      &       $1.94$  &      &       $1.89$  &       \\ \hline  
\end{tabular}
}
\begin{figure}[hbt]
\newcommand{\figWidth}{6.75cm}
\newcommand{\trimfig}[2]{\trimFig{#1}{#2}{0.0}{0.}{.0}{.0}}
\begin{center}
\begin{tikzpicture}[scale=1]
  \useasboundingbox (0,.75) rectangle (15.,5.75);  
  \draw (  0, 0) node[anchor=south west,xshift=-4pt,yshift=+0pt] {\trimfig{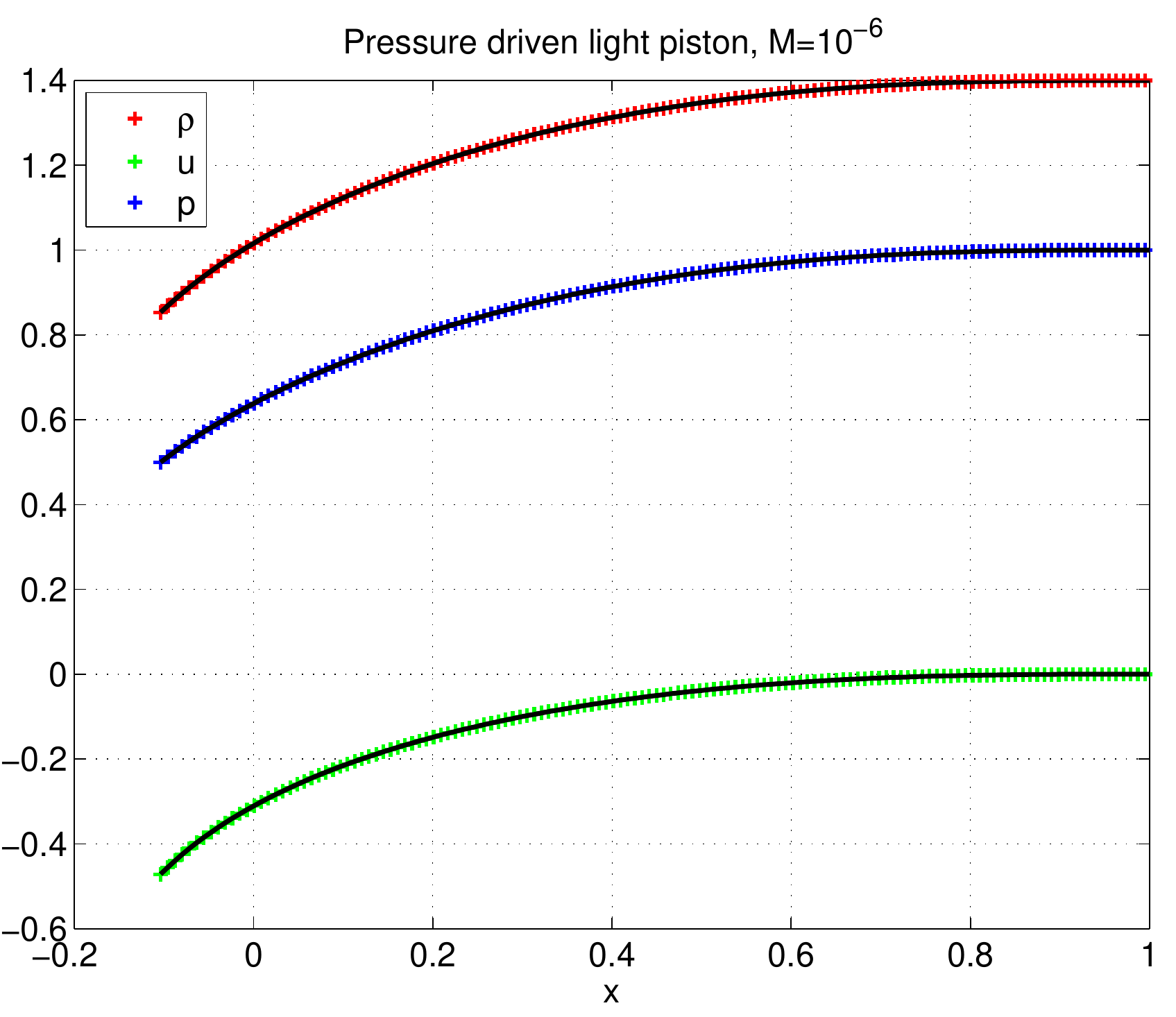}{\figWidth}};
  \draw (7.25, 2) node[anchor=south west,xshift=-4pt,yshift=+0pt] {\pistonTable};
\end{tikzpicture}
\end{center}
  \caption{Results for a pressure driven light piston of mass $\mrb=10^{-6}$. 
    Left: computed and exact solution at $t=1.$ using $\Gp^{(8)}$. Right: maximum errors
    and estimated convergence rates at time $t=1.$ }
  \label{fig:pressureDrivenLightPiston}
\end{figure}

We solve the pressure driven piston problem on a two-dimensional overlapping grid denoted by $\Gp^{(j)}$,
where $j$ denotes the grid resolution (see Figure~\ref{fig:pressureDrivenPistonCartoon}).
The grid spacing in the $x$-direction is chosen to be $\Delta
 x^{(j)} = 1/(10 j)$. The spacing in the $y$-direction is held fixed at $\Delta y  = 2/10$. 
A background Cartesian grid covers the domain $[-0.5,1.5]\times[0,1]$ and remains stationary. A second
Cartesian grid initially covers the domain $[0,0.5]\times[0,1]$ and moves over time according the
piston motion. 

\newcommand{\cRate}{\beta}
The pressure driven piston problem is solved for a piston of mass $\mrb=10^{-6}$. The initial conditions for
the fluid are $(\rho_0,v_0,p_0)=(1.4,0.,1)$ with $\gamma=1.4$. 
The body force is chosen to be $f_b(t)=p_0 \Area ( 1 - \frac{1}{2}t^3)$ which results in a piston that smoothly 
recedes to the left and for which we expect the numerical solution to be second-order accurate in the max-norm. 
The computed and exact solutions are shown in Fig.~\ref{fig:pressureDrivenLightPiston} for results using
grid $\Gp^{(8)}$ and these are in excellent agreement.
Figure~\ref{fig:pressureDrivenLightPiston} also gives the max-norm errors for solutions computed
on a sequence of grids of increasing resolution. 
The values in the columns labelled ``r'' give the ratio of the error on the current grid to that on the previous coarser grid,
a ratio of $4$ being expected for a second-order accurate method.
The convergence rate, $\cRate$, is estimated from a least-squares fit to the log of the error equation $e(h) = C h^\cRate$. 
The results show that the
solution is converging at close to second-order.

\newcommand{\Gre}{\Gc_{\rm re}}
\newcommand{\ds}{\Delta s}
\subsection{Smoothly accelerated ellipse} \label{sec:acceleratedEllipse}

In this example we consider a light rigid body in the shape of an ellipse that is accelerated by a smoothly varying 
body force. We compare the solution from the new added-mass algorithm to that from the old algorithm, the latter requiring
a very small time step to avoid exponential blowup when the mass of the body is small. 

\begin{figure}[hbt]
\newcommand{\figWidtha}{6.cm}
\newcommand{\figWidth}{7.cm}
\newcommand{\trimfiga}[2]{\trimFig{#1}{#2}{0.0}{0.}{.0}{.0}}
\newcommand{\trimfig}[2]{\trimFig{#1}{#2}{0.0}{0.}{.0}{.0}}
\begin{center}
\begin{tikzpicture}[scale=1]
  \useasboundingbox (0,.5) rectangle (14.,6.5);  
  \draw (0,0) node[anchor=south west,xshift=-4pt,yshift=+0pt] {\trimfiga{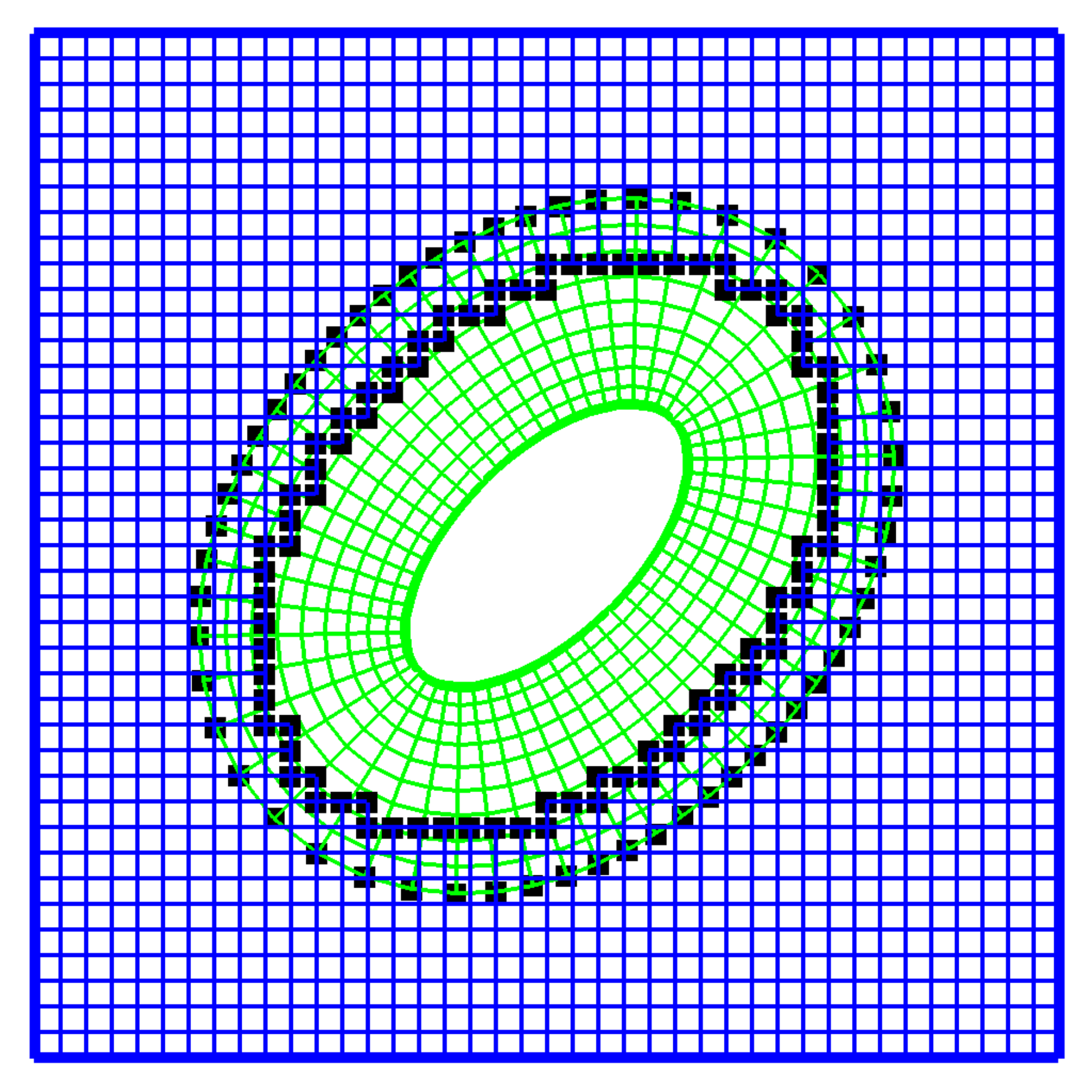}{\figWidtha}};
  \draw (7, 0) node[anchor=south west,xshift=-4pt,yshift=+0pt] {\trimfig{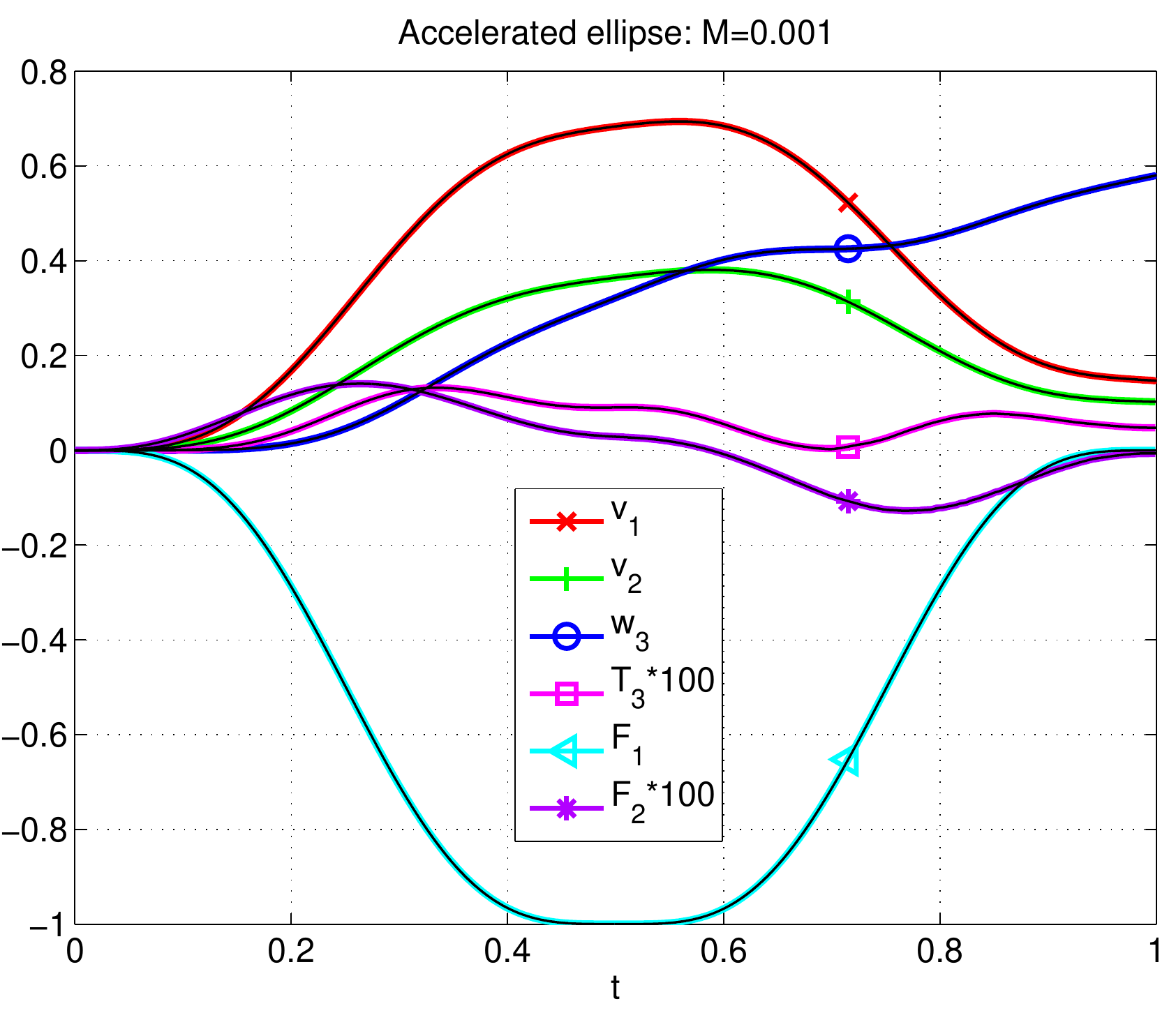}{\figWidth}};
\end{tikzpicture}
\end{center}
  \caption{Accelerated ellipse. 
        Left: overlapping grid $\Gre^{(1)}$ at time $t=0$. 
        Right: time histories of the rigid body velocity $(v_1,v_2)$, angular momentum $w_3$, torque $T_3$ and
        forces $(F_1,F_2)$ for an ellipse of mass $\mrb=10^{-3}$ and moment of inertial $I_3=10^{-3}$ using the old algorithm (black lines) and new algorithm (using grid $\Gre^{(2)}$). ($T_3$ and $F_2$ are scaled by a factor of $100$
for graphical purposes). The force shown on the body does not include the contribution from the external body force.}
  \label{fig:acceleratedEllipse}
\end{figure}

The overlapping grid for this {\em rotated-ellipse} problem is denoted by $\Gre^{(j)}$ where $j$
denotes the grid resolution (grid $\Gre^{(1)}$ is shown in
Fig.~\ref{fig:acceleratedEllipse}).  The grid consists on a stationary
background Cartesian grid for the region $[-2,2]\times[-2,2]$, with grid spacing
$\ds^{(j)}=1/(10 j)$. A narrow boundary fitted grid is located next to the
surface of the elliptical body, and this grid will move to follow the motion of
the body.  The surface of the body is defined by an ellipse, which has major and
minor axes of lengths $1.4$ and $0.7$, respectively, and which is rotated by
$\pi/4$ in the counterclockwise direction.  The boundary fitted grid extends $8$
grid lines in the normal direction (the grid in
Fig.~\ref{fig:acceleratedEllipse} shows an additional ghost line), and the
grid spacing in the normal direction is slightly clustered near the ellipse
surface. The number of points in the tangential direction is chosen so the grid
spacing is approximately $\ds^{(j)}$.

The ellipse is accelerated using a body force that smoothly ramps from zero to one on the time
interval $[0,\half]$ and then smoothly ramps back to zero over the interval $[\half,1]$.
In particular, the body force is in the $x$-direction and is given by
\begin{align}
   f_x(t) & = R(2 t) - R(2 t- 1) , ~~\text{where},~~
  R(t) = \begin{cases}
                       0 & \text{if $t\le 0$} \\ 
                       35 t^4 -84 t^5 + 70 t^6 -20 t^7 & \text{if $0<t<1$} \\
                       1 & \text{if $t\ge 0$}
                  \end{cases}.  \label{eq:smoothBodyForce}
\end{align}
Note that the ramp function $R$ has three continuous derivatives since the first three derivatives of $R(t)$  are zero at $t=0$ and $t=1$.

We consider an an ellipse of mass $\mrb=10^{-3}$ and moment of inertia $I_3=10^{-3}$.
The fluid is taken as an ideal gas with $\gamma=1.4$.
The ellipse and fluid are initially at rest with the initial fluid state given by $(\rho,v_1,v_2,p)=(1/\gamma,0,0,1)$. 
The smooth body force is given by \eqref{eq:smoothBodyForce}. 
The boundary conditions on the Cartesian grid, which have little influence for this problem,
are inflow on the left with all variables given, outflow on the right side  (all variables extrapolated)
and slip walls on the top and bottom.
For comparison, we solve this problem using both the old FSI algorithm and the new {\em added-mass} FSI algorithm.
The new algorithm is run at a CFL number of $0.9$. The old algorithm experiences exponential blowup at this CFL number and is instead run at a CFL number of $1/100$. 

In the right-hand side of Fig.~\ref{fig:acceleratedEllipse} we show the state
of the rigid body over time for the old and new algorithms. The body initially
accelerates upward and to the right as indicated by the components of the body
velocity and rotates in a counter-clockwise direction as indicated by the
angular velocity.  The forces on the body shown in
Fig.~\ref{fig:acceleratedEllipse} do not include the contributions from the
external body force and thus represent the force exerted by the fluid on the
body. The force $F_1$ indicates that the fluid pushes back on the body to nearly
balance the external force $f_x(t)$.  The results from the old and new algorithm are
nearly indistinguishable in this plot indicating that the new algorithm provides
an accurate approximation even with a time step that is nearly 100 times larger
than the old algorithm.

{
\newcommand{\figWidth}{6.4cm}
\newcommand{\trimfig}[2]{\trimPlot{#1}{#2}{.05}{.125}{.07}{.1}}
\newcommand{\cbHeight}{6.3cm}
\newcommand{\xcb}{6.7cm}
\newcommand{\ycb}{0.1cm}
\newlength{\ycbTop}%
\setlength{\ycbTop}{\ycb+\cbHeight}
\newlength{\ycbMid}%
\setlength{\ycbMid}{\ycb+\cbHeight*\real{.5}}
\newcommand{\xLabel}{6.5cm}
\newcommand{\yLabel}{6.5cm}
\newcommand{\trimfigcb}[2]{\includegraphics[height=#2, clip, trim=17cm 2.35cm 1.65cm 2.35cm]{#1}}
\newcommand{\timeFrame}[6]{%
 \begin{scope}[xshift=#5cm,yshift=#6cm]
  \draw ( 0.0,0) node[anchor=south west] {\trimfig{#1}{\figWidth}};
  \draw (\xLabel,\yLabel) node[draw,fill=white,anchor=east,xshift=+1pt,yshift=-4pt] {\scriptsize #2};
  \draw (\xcb,\ycb) node[anchor=south west,xshift= +0pt,yshift=+0pt] {\trimfigcb{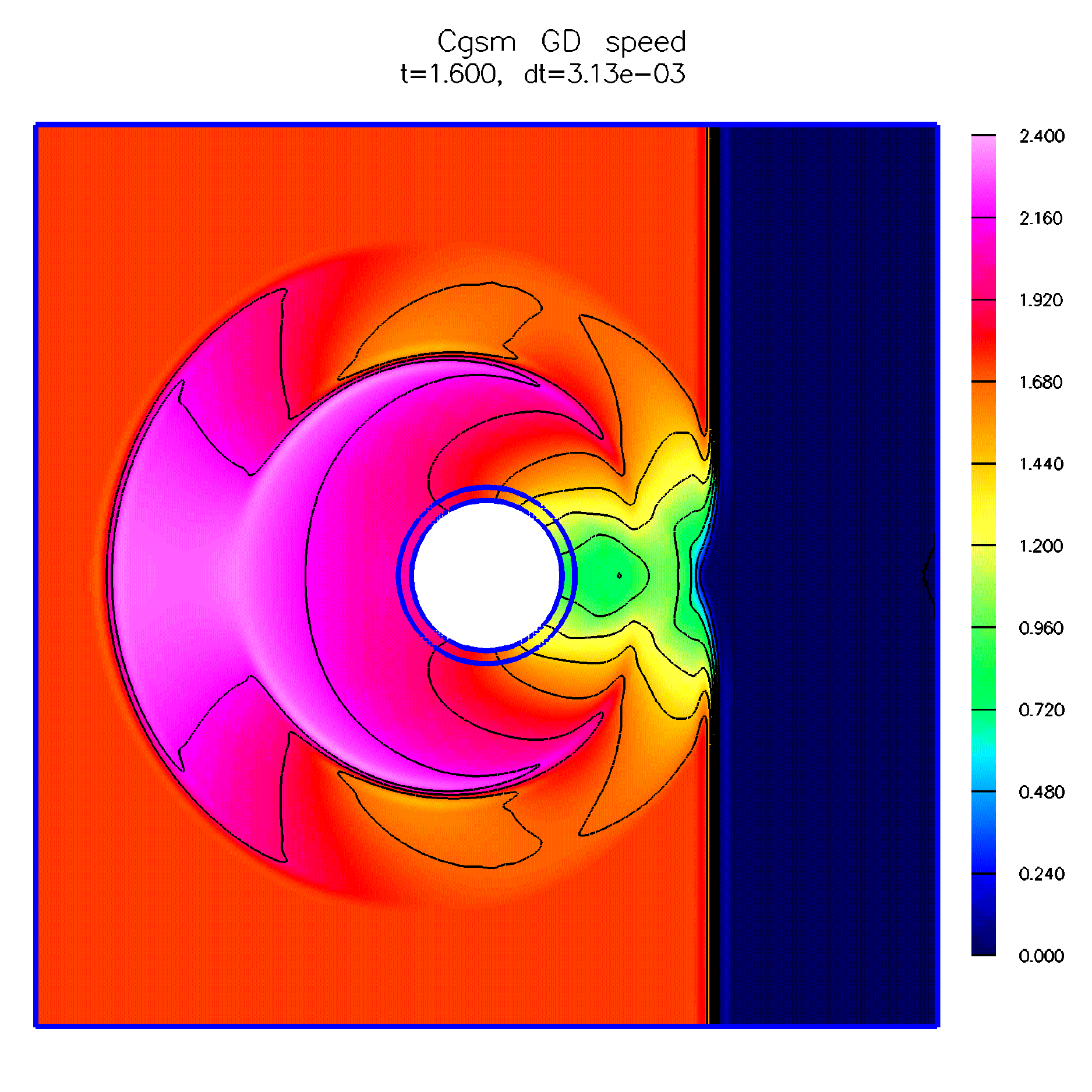}{\cbHeight}};
  \draw (\xcb,\ycb) node[anchor=south west,xshift= +8pt,yshift=+1pt] {\scriptsize $#3$};
  \draw (\xcb,\ycbTop) node[anchor=south west,xshift= +8pt,yshift=-4pt] {\scriptsize $#4$};
  \draw (\xcb,\ycbMid) node[anchor=      west,xshift=+10pt,yshift=2pt] {\small $p$};
 \end{scope}
}
%
%
\begin{figure}[hbt]
\begin{center}
\begin{tikzpicture}[scale=1]
  \useasboundingbox (0,.5) rectangle (16.,14.0);  
  \timeFrame{images/ellipseAccelMass1em3P0p5}{Added-mass algorithm: $t=0.5$}{.31}{1.96}{0}{7.0}
  \timeFrame{images/ellipseAccelMass1em3P1p0}{Added-mass algorithm: $t=1.0$}{.18}{1.59}{8}{7.0}
  \timeFrame{images/ellipseAccelMass1em3OldCfl01P0p5}{Old algorithm: $t=0.5$}{.31}{1.96}{0}{0}
  \timeFrame{images/ellipseAccelMass1em3OldCfl01P1p0}{Old algorithm: $t=1.0$}{.18}{1.59}{8}{0}
\end{tikzpicture}
\end{center}
  \caption{Accelerated ellipse: pressure at $t=0.5$ and $t=1.0$ for the old algorithm running at CFL number $10^{-2}$ (bottom) 
       and new {\em added-mass} algorithm running at CFL number $0.9$ (top) for grid $\Gre^{(16)}$. }
  \label{fig:acceleratedEllipseFlow}
\end{figure}
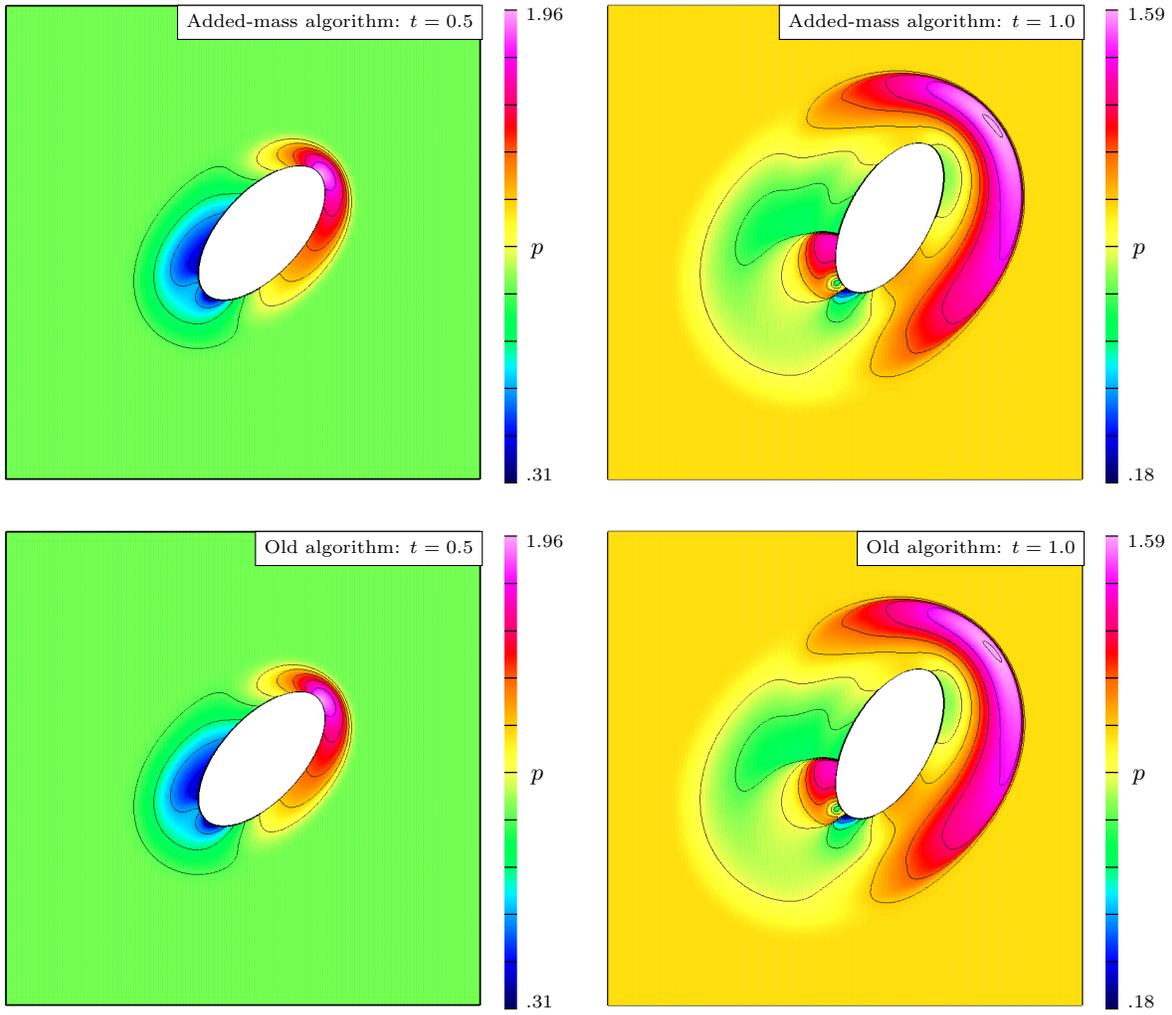
}

Fig.~\ref{fig:acceleratedEllipseFlow} shows contours of the pressure field at times $t=0.5$ and $t=1.0$ for 
both the old and new algorithms. The accelerating body generates a forward moving wave that steepens over
time and which has formed a shock by $t=1.0$. The solutions from the old and new algorithm are in excellent agreement with
almost no detectable differences. 
For a more quantitative evaluation of the accuracy we determine a-posteriori error estimates by solving the problem on a sequence of grids
of increasing resolution and using the error estimation approach described in~\cite{henshaw08,banks09a}. 
Fig.~\ref{tab:acceleratedEllipseConvergenceMaxNorm} shows the estimated max-norm errors and convergence rates at time $t=0.4$ when
the solution is still smooth. These results show that the solution is converging at close to second-order accuracy. We note that for
these results the slope-limiter was turned off in the Godunov method since this slope limiter can reduce the order of accuracy. 
Fig.~\ref{tab:acceleratedEllipseConvergence} shows the estimated $L_1$-norm errors and convergence rates at time $t=1.0$ when the
solution is no longer smooth. In this case the results show that the solution is converging at rates close to $1$, which are the
expected rates for problems with shocks. We note that the discrete $L_1$-norm of a grid function is computed in the usual way by summing the
absolute values of the values at each grid point and dividing by the total number of grid points~\cite{henshaw08}.

\renewcommand{\tableFontSize}{\small}
\begin{figure}[hbt]\tableFontSize
\begin{center}
\begin{tabular}{|c|c|c|c|c|c|c|c|c|c|} \hline
Grid ~$\Gc^{(j)}$ &~~$h_j$~~& $\eem_\rho$ & $r$ & $\eem_u$ & $r$ & $\eem_v$ & $r$  & $\eem_p$  & $r$ \\ \hline 
~$\Gre^{(8)} $~   & 1/40   & \num{8.0}{-3} &      & \num{5.3}{-3} &      & \num{3.4}{-3} &       & \num{8.3}{-3} &      \\ \hline	
~$\Gre^{(16)} $~  & 1/80   & \num{2.2}{-3} &  3.7 & \num{1.4}{-3} &  3.8 & \num{9.7}{-4} &  3.5  & \num{2.3}{-3} &  3.7 \\ \hline	
~$\Gre^{(32)}$~   & 1/160  & \num{5.9}{-4} &  3.7 & \num{3.7}{-4} &  3.8 & \num{2.8}{-4} &  3.5  & \num{6.2}{-4} &  3.7 \\ \hline	
\rateLabel        &        &     1.88      &      &     1.93      &      &     1.80      &       &     1.87      &      \\ \hline  
\end{tabular}
\caption{A posteriori estimated errors (max-norm) and convergence rates for the accelerated ellipse at $t=0.4$ (no slope limiter).
   The scheme converges at close to second-order accuracy in the max-norm when the solution is smooth.}
\label{tab:acceleratedEllipseConvergenceMaxNorm}
\end{center}
\end{figure}

\begin{figure}[hbt]\tableFontSize
\begin{center}
\begin{tabular}{|c|c|c|c|c|c|c|c|c|c|} \hline
Grid ~$\Gc^{(j)}$ &~~$h_j$~~& $\eem_\rho$ & $r$ & $\eem_u$ & $r$ & $\eem_v$ & $r$  & $\eem_p$  & $r$ \\ \hline 
~$\Gre^{(8)} $~   & 1/40 & \num{2.1}{-3} &      & \num{9.3}{-4} &      & \num{9.6}{-4} &       & \num{2.1}{-3} &      \\ \hline 
~$\Gre^{(16)} $~  & 1/80 & \num{9.9}{-4} &  2.1 & \num{4.3}{-4} &  2.1 & \num{4.6}{-4} &  2.1  & \num{9.6}{-4} &  2.2 \\ \hline 
~$\Gre^{(32)}$~   & 1/160& \num{4.7}{-4} &  2.1 & \num{2.0}{-4} &  2.1 & \num{2.2}{-4} &  2.1  & \num{4.5}{-4} &  2.2 \\ \hline 
\rateLabel        &      &     1.08      &      &     1.09      &      &     1.07      &       &     1.11      &     \\ \hline  
\end{tabular}
\caption{A posteriori estimated errors ($L_1$-norm) and convergence rates for the accelerated ellipse at $t=1.0$.
       The scheme converges at close to first-order accuracy in the $L_1$-norm when the solution is not smooth.}
\label{tab:acceleratedEllipseConvergence}
\end{center}
\end{figure}

\subsection{Shock driven zero mass ellipse} \label{sec:shockDrivenEllipse}

The shock driven ellipse problem consists of a Mach 2 shock that impacts an
ellipse of zero mass and zero moment of inertia.  This example demonstrates the robustness of the new
added-mass algorithm on a difficult problem for which the old rigid-body FSI
algorithm would fail for any time-step, no matter how small. We note that since
the mass and moments of inertial of the body are zero in the Newton-Euler
equations~\eqref{eq:addedMassRigidBodyEquations}, the linear and angular velocities of the body respond immediately to ensure the
net force on the body is zero; there is no damping in the response from the body's inertia.

The overlapping grid for this problem, $\Gre^{(j)}$ is the same as that used in Section~\ref{sec:acceleratedEllipse}. 
We use adaptive mesh refinement in some of the computations of this section. Let $\Gre^{(j\times 4)}$ denote the
AMR grid that has a base grid $\Gre^{(j)}$ with grid spacing $\Delta s^{(j)}\approx 1/(10j)$ together with
one level of refinement grids of refinement factor $4$. The effective resolution of the AMR grid $\Gre^{(j\times 4)}$ is
thus $\Delta s^{(j\times4)}\approx 1/(40j)$. We note that the AMR grids are added to both the background grid and
to the component grid around the ellipse, refer to~\cite{henshaw06} for further details of the moving-grid AMR approach.

The initial conditions in the fluid consist of a shock located at $x=-1$ with initial state $(\rho,u,v,p)=(2.6667,1.25,0,3.214256)$ ahead
of the shock and $(\rho,u,v,p)=(1,0,0,1.4)$ behind the shock. The boundary conditions are supersonic inflow (all variables specified)
on the left face of the background grid and supersonic outflow (all variables extrapolated) on the other faces of the background grid.

{
\begin{figure}[hbt]
\newcommand{\figWidth}{7.5cm}
\newcommand{\trimfig}[2]{\trimFig{#1}{#2}{0.0}{0.}{.0}{.0}}
\begin{center}
\begin{tikzpicture}[scale=1]
  \useasboundingbox (0,.75) rectangle (8.,6.5);  
  \draw ( 0, 0) node[anchor=south west,xshift=-4pt,yshift=+0pt] {\trimfig{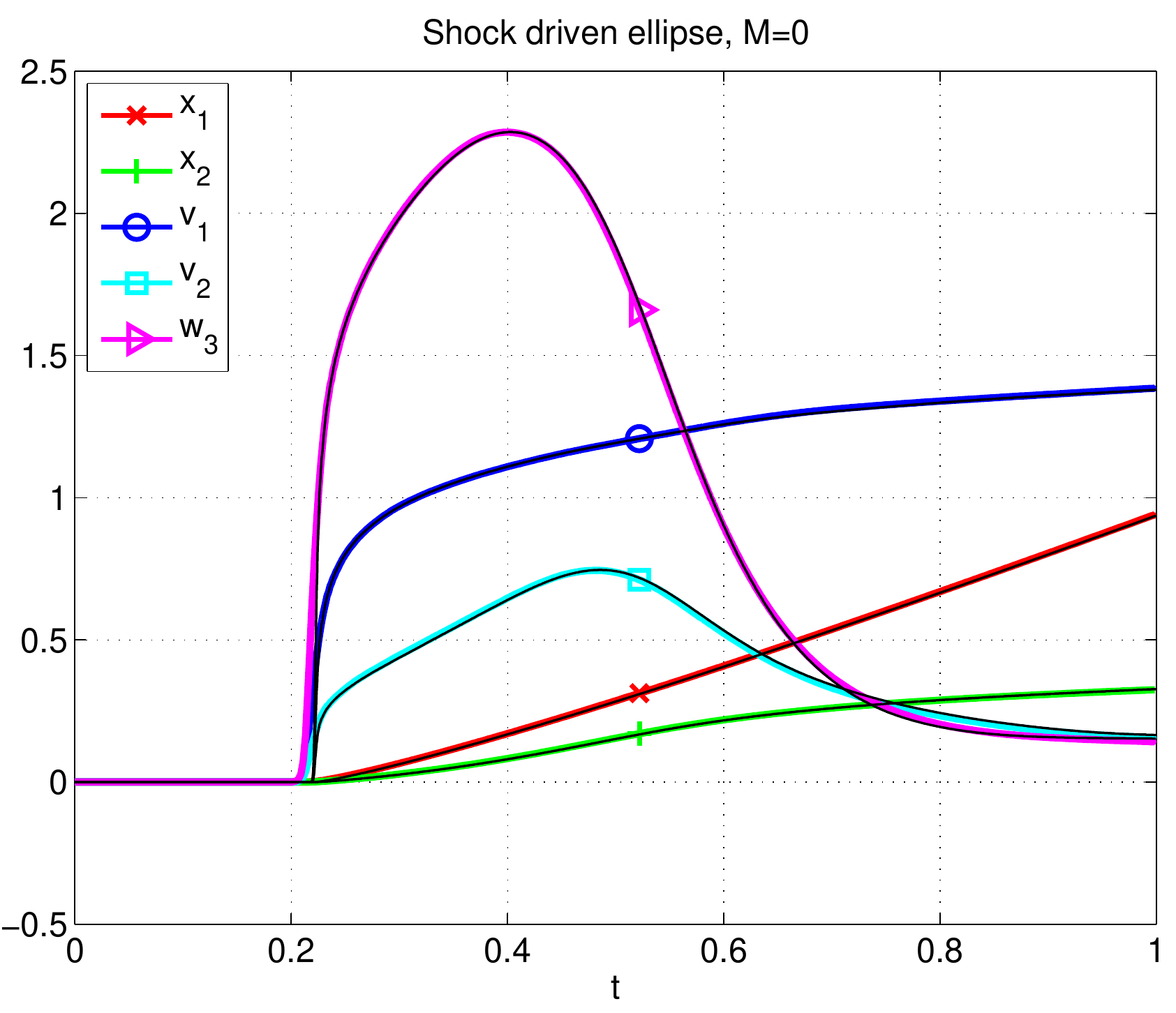}{\figWidth}};
%
\end{tikzpicture}
\end{center}
  \caption{Shock-drive ellipse: time histories of the center of mass, $(x_1,x_2)$, the velocity of the center of
  mass, $(v_1,v_2)$ and the angular velocity $w_3$. The colored lines are results from the coarse
  grid $\Gre^{(8)}$ while  the black lines are results using the finer grid $\Gre^{(32)}$.}   
  \label{fig:shockDrivenEllipseRigidBodyInfo}
\end{figure}
}

{
\newcommand{\schlierenName}{ellipseShockMassZero16l1r4Schlieren}
\newcommand{\amrName}{ellipseShockMassZero16l1r4AmrP}
\newcommand{\figWidth}{6.5cm}
\newcommand{\trimfig}[2]{\trimFig{#1}{#2}{.23}{.23}{.23}{.23}}
\newcommand{\cbHeight}{6.3cm}
\newcommand{\xcb}{6.7cm}
\newcommand{\ycb}{0.1cm}
\setlength{\ycbTop}{\ycb+\cbHeight}
\setlength{\ycbMid}{\ycb+\cbHeight*\real{.5}}
\newcommand{\xLabel}{6.5cm}
\newcommand{\yLabel}{6.5cm}
\newcommand{\trimfigcb}[2]{\includegraphics[height=#2, clip, trim=17cm 2.35cm 1.65cm 2.35cm]{#1}}
\newcommand{\timeFrame}[6]{%
 \begin{scope}[xshift=#5cm,yshift=#6cm]
  \draw ( 0.0,0) node[anchor=south west] {\trimfig{#1}{\figWidth}};
  \draw (\xLabel,\yLabel) node[draw,fill=white,anchor=east,xshift=+1pt,yshift=-4pt] {\scriptsize $t=#2$};
  \draw (\xcb,\ycb) node[anchor=south west,xshift= +0pt,yshift=+0pt] {\trimfigcb{images/colourBarLines}{\cbHeight}};
  \draw (\xcb,\ycb) node[anchor=south west,xshift= +8pt,yshift=+1pt] {\scriptsize $#3$};
  \draw (\xcb,\ycbTop) node[anchor=south west,xshift= +8pt,yshift=-6pt] {\scriptsize $#4$};
  \draw (\xcb,\ycbMid) node[anchor=      west,xshift=+10pt,yshift=2pt] {\small $p$};
 \end{scope}
}
%
%
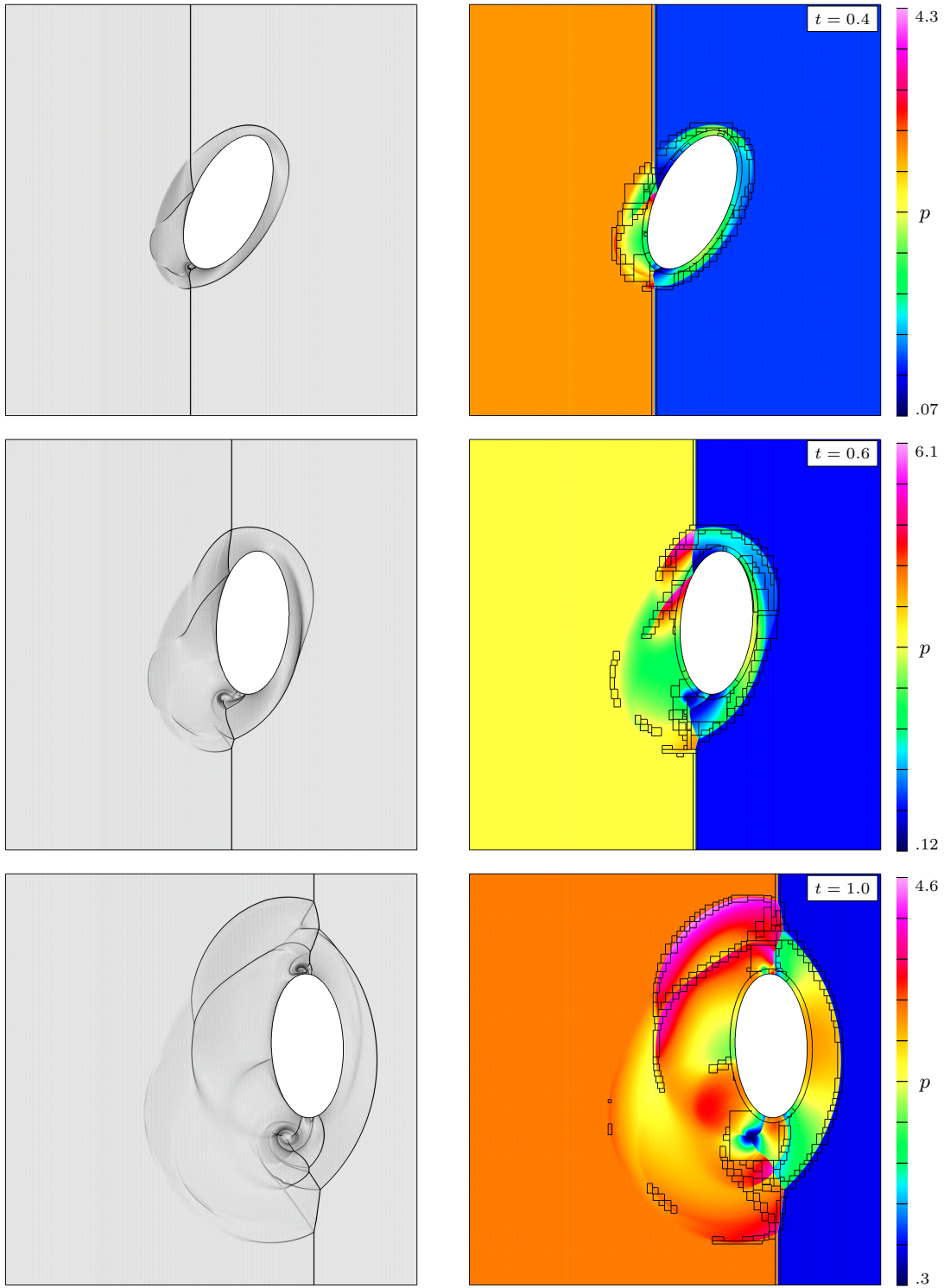
\begin{figure}[htb]
\begin{center}
\begin{tikzpicture}[scale=1]
  \useasboundingbox (0,.5) rectangle (14.,19.5);  
  \draw ( 0,13.4) node[anchor=south west,xshift=-4pt,yshift=+0pt] {\trimfig{images/\schlierenName4}{\figWidth}};
  \draw ( 0, 6.7) node[anchor=south west,xshift=-4pt,yshift=+0pt] {\trimfig{images/\schlierenName6}{\figWidth}};
  \draw ( 0, 0) node[anchor=south west,xshift=-4pt,yshift=+0pt] {\trimfig{images/\schlierenName10}{\figWidth}};
  \timeFrame{images/\amrName4}{0.4}{.07}{4.3}{7}{13.4}
  \timeFrame{images/\amrName6}{0.6}{.12}{6.1}{7}{6.7}
  \timeFrame{images/\amrName10}{1.0}{.3}{4.6}{7}{0}
\end{tikzpicture}
\end{center}
  \caption{Shock driven zero mass ellipse. Schlieren images (left column) and pressure contours (right column)
 at times $t=0.4$, $t=0.6$ and $=1.0$ using grid $\Gre^{(16\times 4)}$. The block boundaries of the refinement
   grids are shown superimposed on the pressure contours.}
  \label{fig:shockDrivenEllipse}
\end{figure}
}

Fig.~\ref{fig:shockDrivenEllipseRigidBodyInfo} compares the time history of the rigid body dynamics
from a coarse grid, $\Gre^{(8)}$, and finer grid, $\Gre^{(32)}$,  computation. The velocity and angular velocity are seen to rapidly
increase when the shock first hits the ellipse just after $t=0.2$. The ellipse is initially accelerated up and
to the right and experiences a rapid counter-clockwise rotation. After an initial rise, the angular velocity
decreases and approximately levels off at some positive value\footnote{We note that the long time behavior of the ellipse is of interest
but we do not pursue that line of investigation here.}. 
The results from the two computations are in excellent agreement.

Numerical schlierens and contours of the pressure field at different times are
shown in Fig.~\ref{fig:shockDrivenEllipse} (see~\cite{henshaw06} for a definition of the numerical schlieren function). 
 The computations were performed
with AMR using the grid $\Gre^{(16\times4)}$ (base grid $\Gre^{(16)}$ plus one
refinement level of refinement ratio $4$).   The solution at $t=0.4$ shows
the ellipse has undergone a rapid acceleration upward and to the right combined with a
rapid counter clockwise rotation. The impact of the incident shock on the
ellipse causes a shock to form in the region ahead of the body. By $t=1.0$, a
complex pattern of interacting shocks has formed in the regions above and below
the ellipse.  In Fig.~\ref{fig:shockDrivenEllipseCompare} we compare the
schlieren images of the solution at $t=1.0$ from grids of different
resolutions. These result show good agreement in the basic structure of the
solution, with additional fine scale features appearing as the grid is
refined. This is the expected behavior for inviscid computations.

{
\newcommand{\figWidth}{5.cm}
\newcommand{\trimfig}[2]{\trimPlotb{#1}{#2}{.4}{.15}{.15}{.15}}
\begin{figure}[hbt]
\begin{center}
\begin{tikzpicture}[scale=1]
  \useasboundingbox (0,.5) rectangle (16.,8.);  
  \draw ( 0.0, 0) node[anchor=south west,xshift=-4pt,yshift=+0pt] {\trimfig{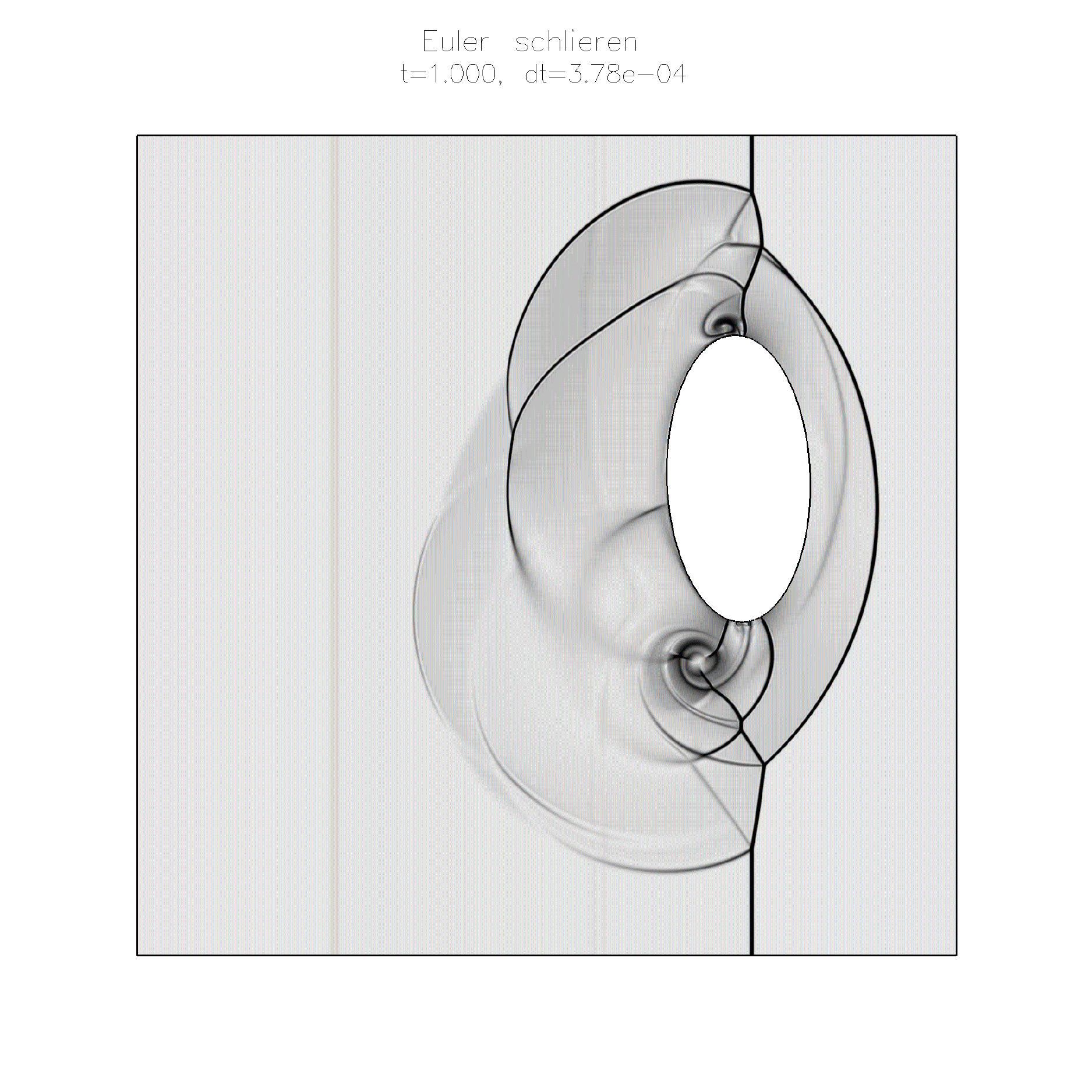}{\figWidth}};
  \draw ( 5.5, 0) node[anchor=south west,xshift=-4pt,yshift=+0pt] {\trimfig{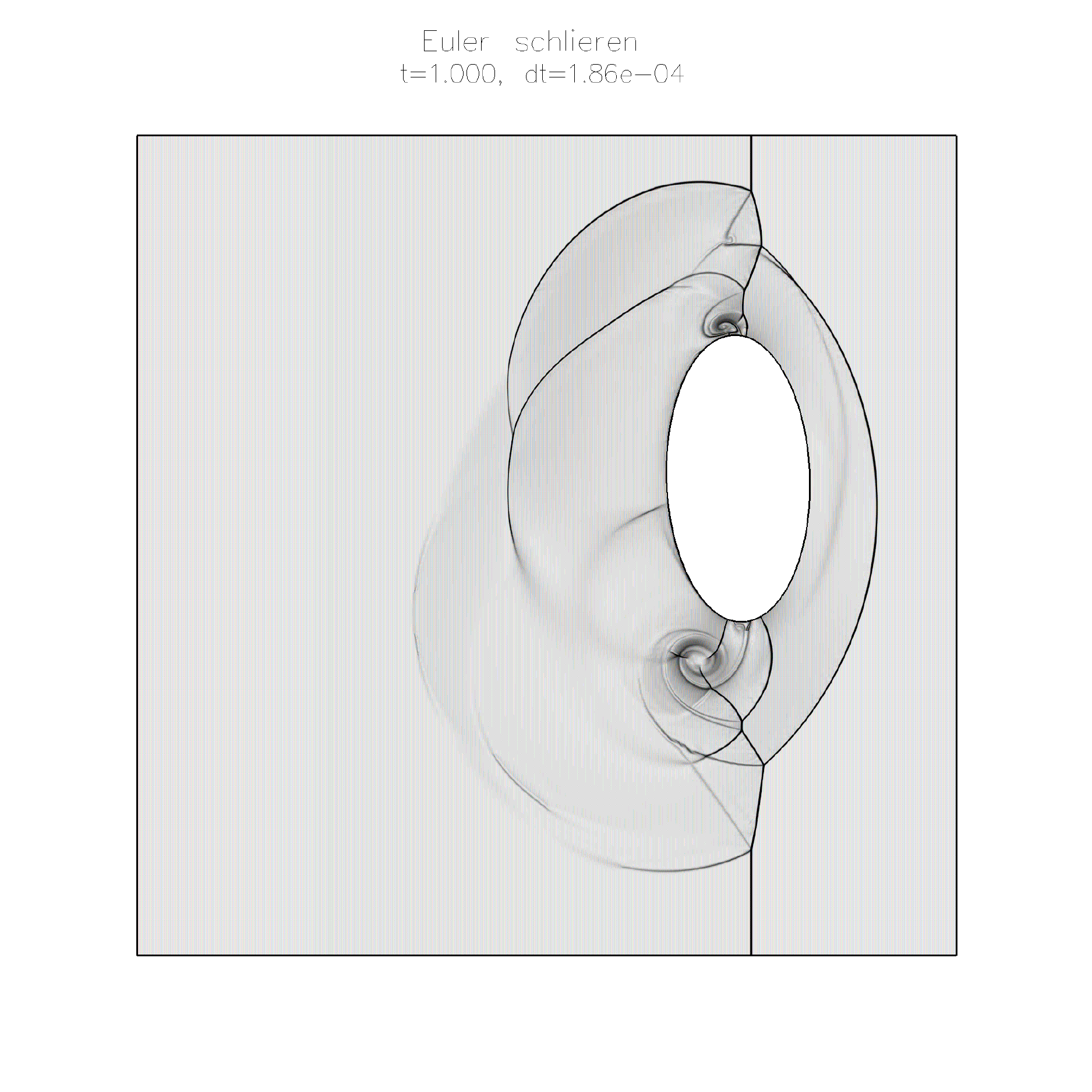}{\figWidth}};
  \draw (11.0, 0) node[anchor=south west,xshift=-4pt,yshift=+0pt] {\trimfig{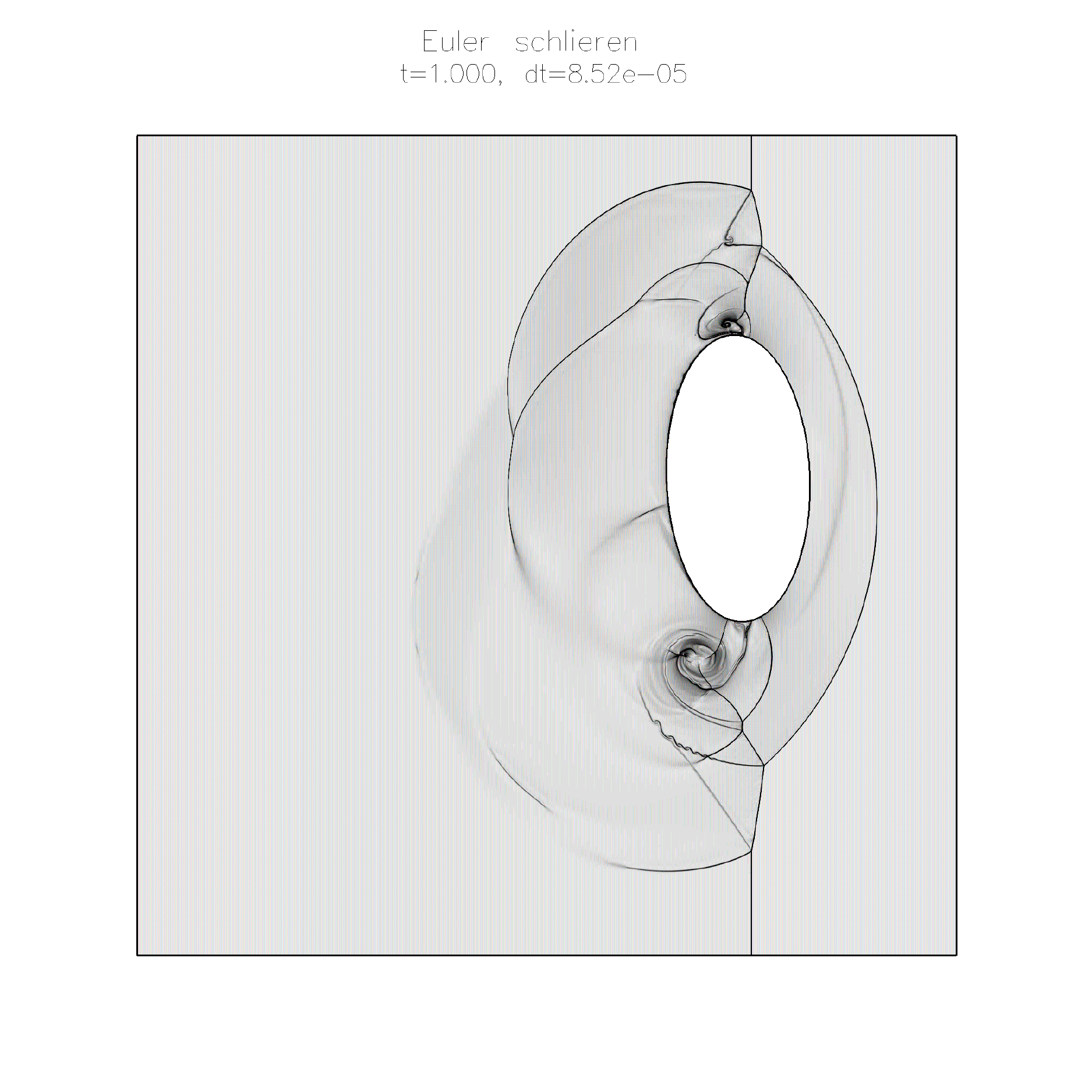}{\figWidth}};
\end{tikzpicture}
\end{center}
  \caption{Shock driven zero mass ellipse. A comparison of schlieren images of the solution at $t=1.0$ 
      computed on the coarse grid $\Gre^{(32)}$ (left),
      medium grid $\Gre^{(16\times4)}$ (middle) and fine (AMR) grid $\Gre^{(8\times4\times4)}$ (right).}
  \label{fig:shockDrivenEllipseCompare}
\end{figure}
}

\newcommand{\Gcb}{\Gc_{\rm sf}}
\newcommand{\Narm}{N_a}
\subsection{Shock impacting a zero mass body with complex boundary} \label{sec:shockComplexBody}

As a final case we consider a Mach 2 shock that impacts a zero mass body with
a complex boundary. This interesting example demonstrates that the new added-mass algorithm
is straight-forward to apply to bodies with complex shapes and that the algorithm remains robust in
the difficult regime of a zero mass body. 
The boundary of the {\em starfish} body is the two-dimensional curve $\xv_S(s)=[x_S(s),y_S(s)]^T$, defined by 
\begin{align}
   \xv_S(s) &= R(s)\, \begin{bmatrix} \cos\hat{\theta}(s) \\ \sin\hat{\theta}(s) \end{bmatrix}, \qquad s\in[0,1], \\
    R(s) &= r_a + r_b r(s), \\
    \hat{\theta}(s) &= \theta(s) + \alpha r(s)^2, \\
    r(s) &= \Big( \half\big[1+\sin(\Narm\, \theta(s))\big] \Big)^2, \\
    \theta(s) &= 2\pi s.
\end{align}
Here $\Narm$ is an integer that defines the number of {\em arms}, $r_a=0.4$ defines
the radius of the base of the arms and $r_b=0.6$ defines the length the arms. The parameter $\alpha$
controls the sweep of the arms and we take $\alpha=\pi/\Narm$. 

\begin{figure}[hbt]
\newcommand{\figWidtha}{7.cm}
\newcommand{\figWidth}{7.cm}
\newcommand{\trimfiga}[2]{\trimFig{#1}{#2}{0.0}{0.}{.0}{.0}}
\newcommand{\trimfig}[2]{\trimFig{#1}{#2}{0.0}{0.}{.0}{.0}}
\begin{center}
\begin{tikzpicture}[scale=1]
  \useasboundingbox (0,.5) rectangle (14.5,7.25);  
  \draw (0,0) node[anchor=south west,xshift=-4pt,yshift=+0pt] {\trimfiga{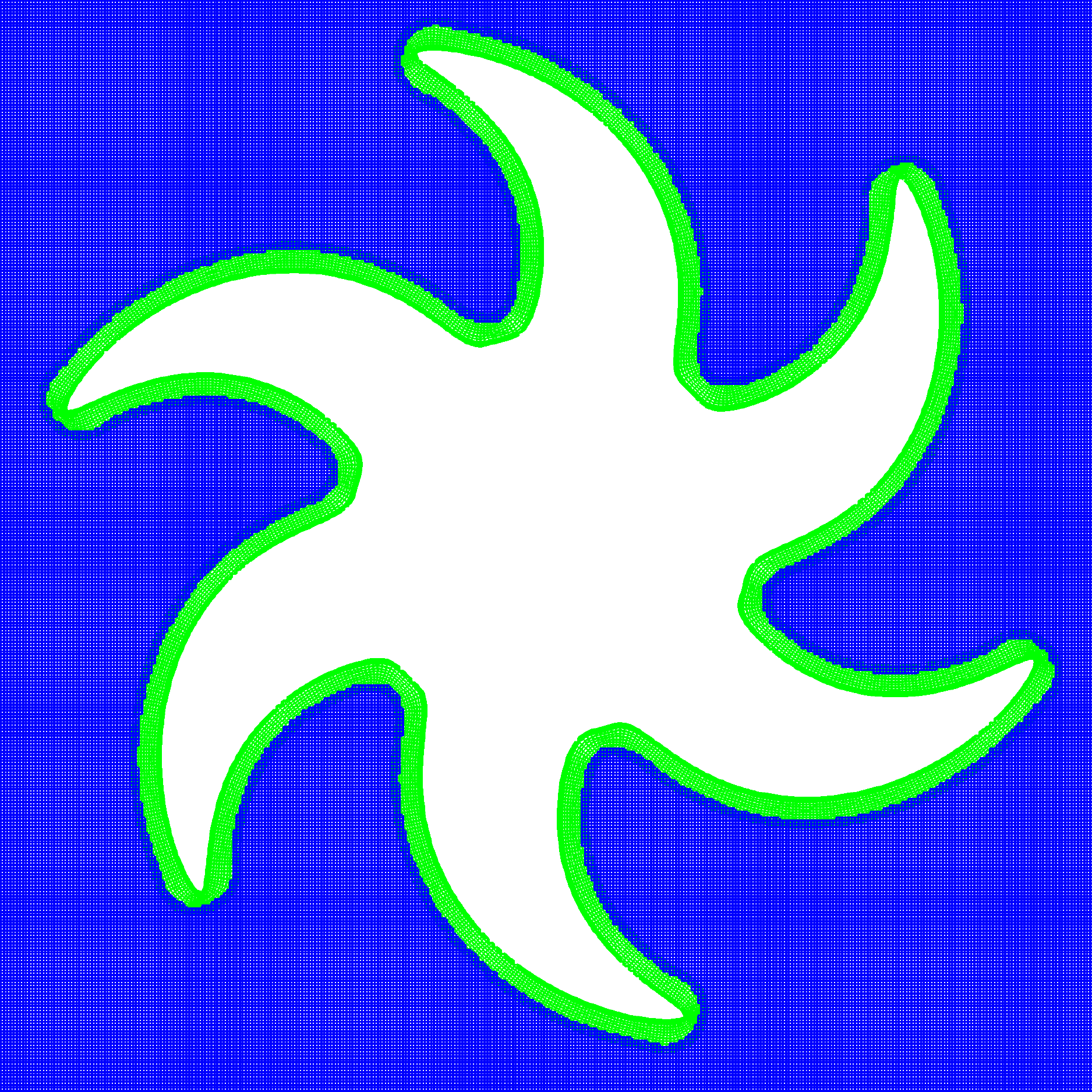}{\figWidtha}};
  \draw (7.5, 0) node[anchor=south west,xshift=-4pt,yshift=+0pt] {\trimfig{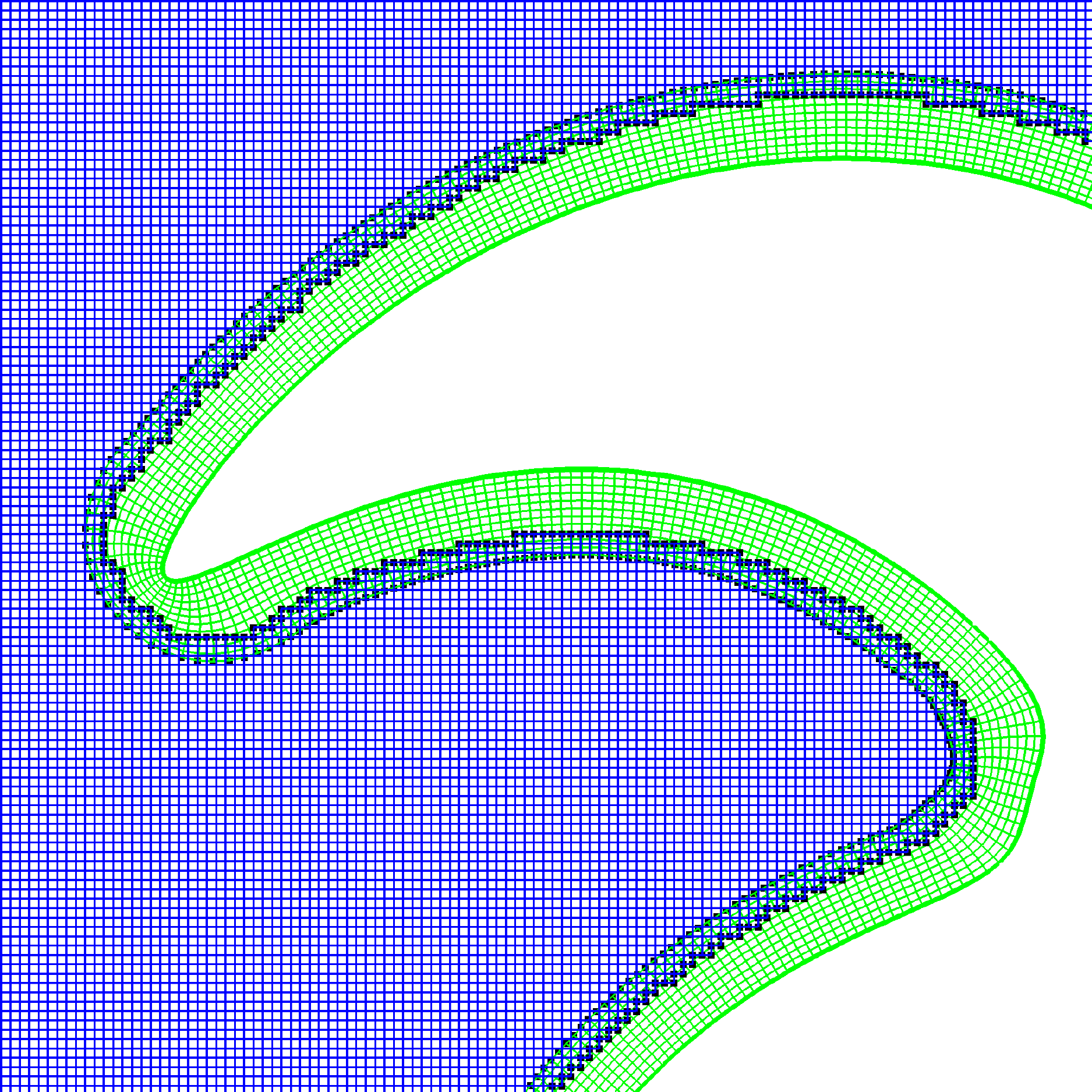}{\figWidth}};
\end{tikzpicture}
\end{center}
  \caption{Starfish grid. Left: overlapping grid $\Gcb^{(16)}$ at time $t=0$. Right: magnified view showing the smooth and high quality grid near the boundary. }

  \label{fig:starFishGrid}
\end{figure}

The overlapping grid for this problem, $\Gcb^{(j)}$, is shown in Fig.~\ref{fig:starFishGrid}.
The boundary curve is fit with a cubic spline. A volume grid is generated near the surface, to a distance of $0.05$, using the
hyperbolic grid generator in Overture~\cite{HyperbolicGuide}. The grid spacing is chosen to be approximately $\Delta s^{(j)}\approx 1/(10j)$.
As in the previous section, we use adaptive mesh refinement 
and let $\Gcb^{(j\times 4)}$ denote a grid with one level of refinement, with refinement factor 4.
The initial conditions in the fluid consist of a shock located at $x=-1.2$ with initial state $(\rho,u,v,p)=(2.6667,1.25,0,3.214256)$ ahead
of the shock and $(\rho,u,v,p)=(1,0,0,1.4)$ behind the shock. The boundary conditions are supersonic inflow (all variables specified)
on the left face of the background grid and supersonic outflow (all variables extrapolated) on the other faces of the background grid.

{
\newcommand{\figWidth}{8.1cm}
\newcommand{\trimfig}[2]{\trimPlotb{#1}{#2}{.1}{.02}{.1}{.1}}
\begin{figure}[hbt]
\begin{center}
\begin{tikzpicture}[scale=1]
  \useasboundingbox (0,.65) rectangle (16.5,15.5);  
  \draw ( 0.  ,7.6) node[anchor=south west,xshift=-4pt,yshift=+0pt] {\trimfig{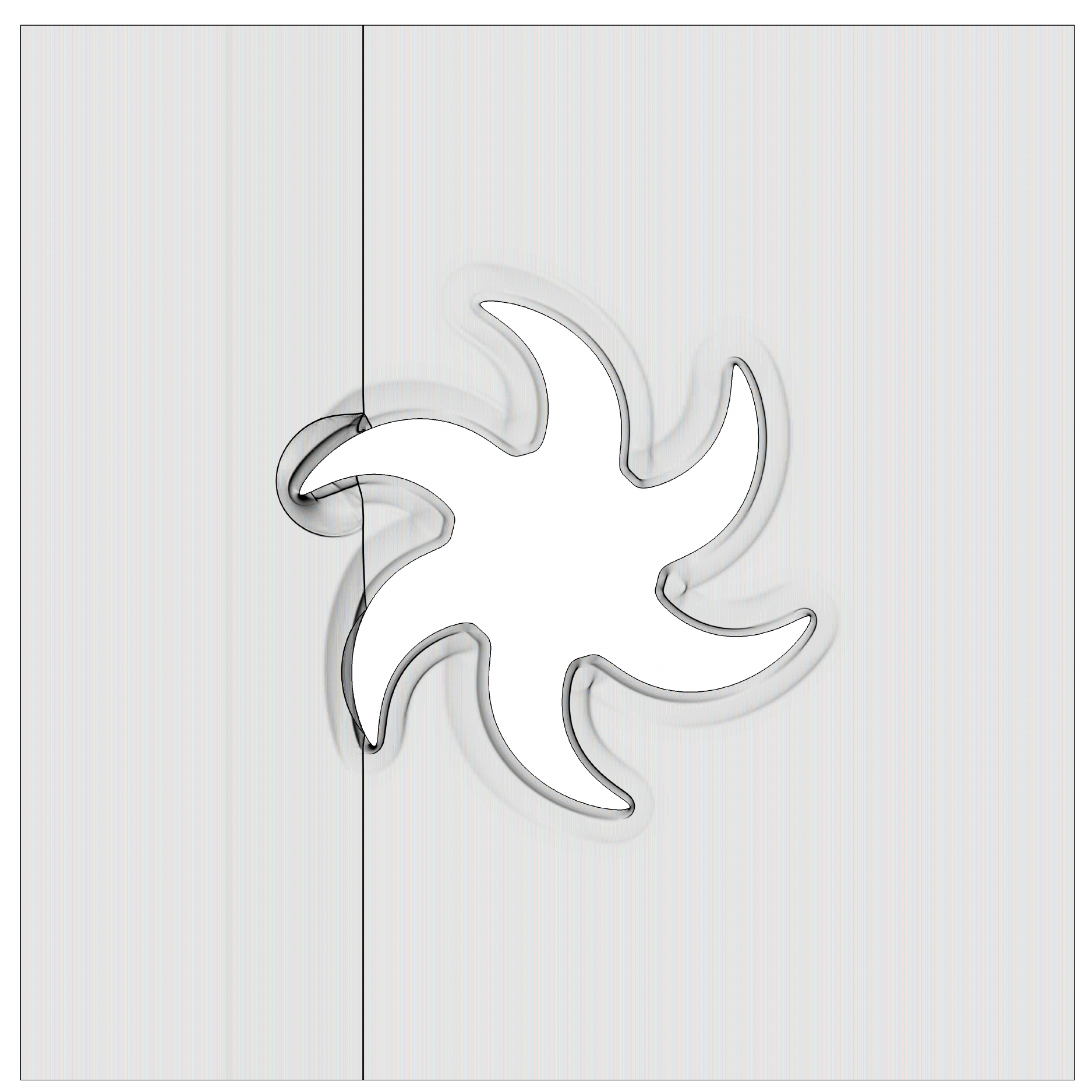}{\figWidth}};
  \draw ( 8.25,7.6) node[anchor=south west,xshift=-4pt,yshift=+0pt] {\trimfig{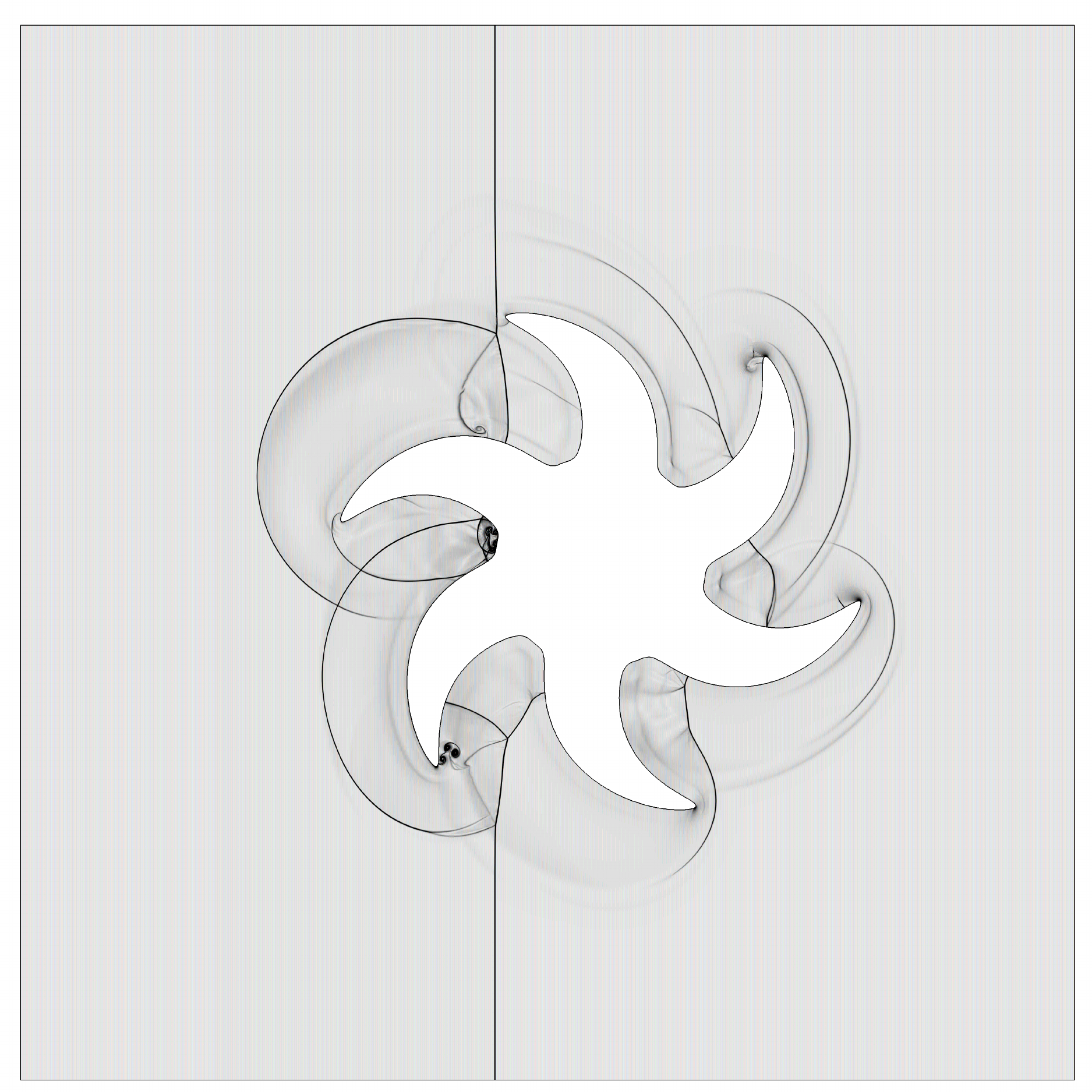}{\figWidth}};
  \draw ( 0   ,0.0) node[anchor=south west,xshift=-4pt,yshift=+0pt] {\trimfig{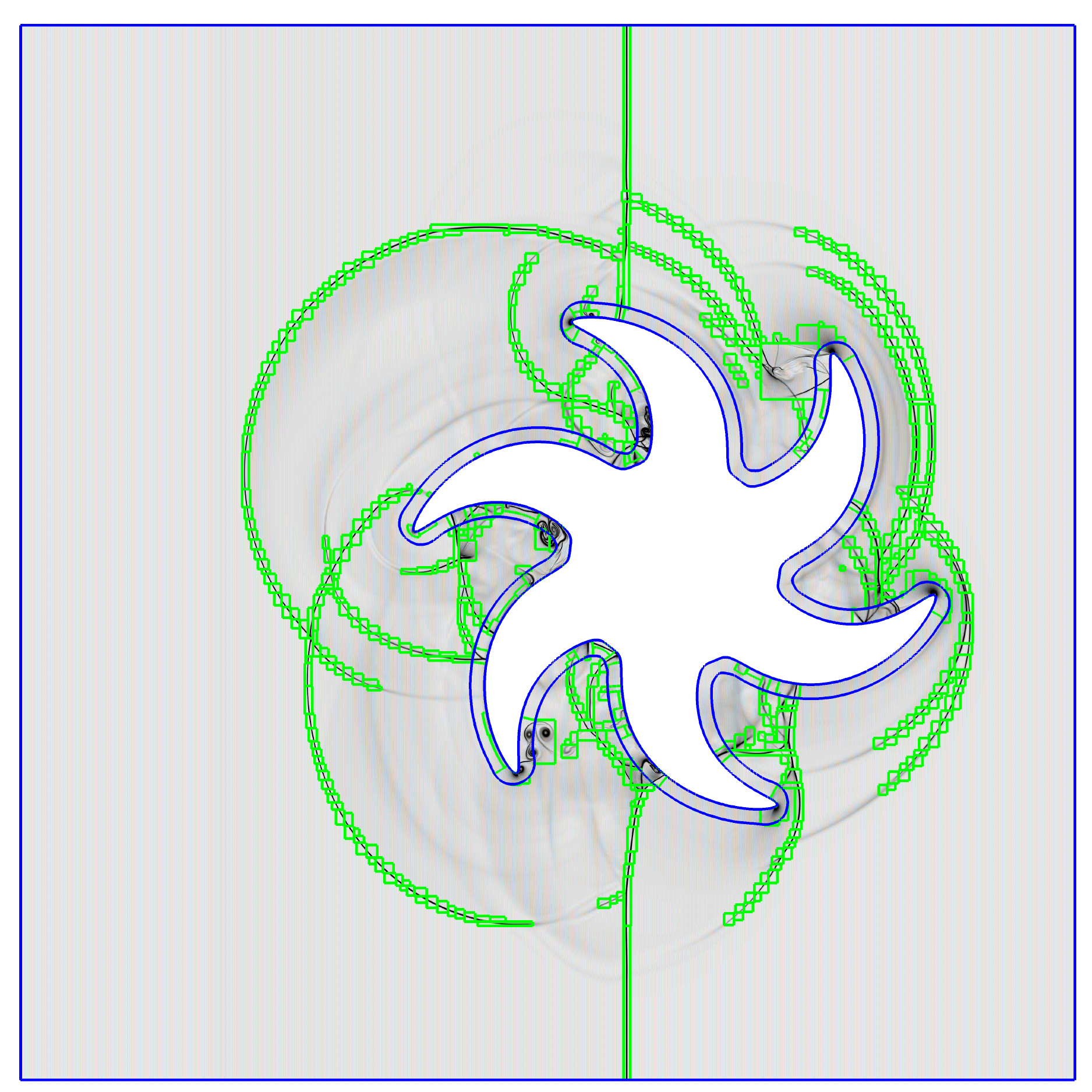}{\figWidth}};
  \draw ( 8.25,0.0) node[anchor=south west,xshift=-4pt,yshift=+0pt] {\trimfig{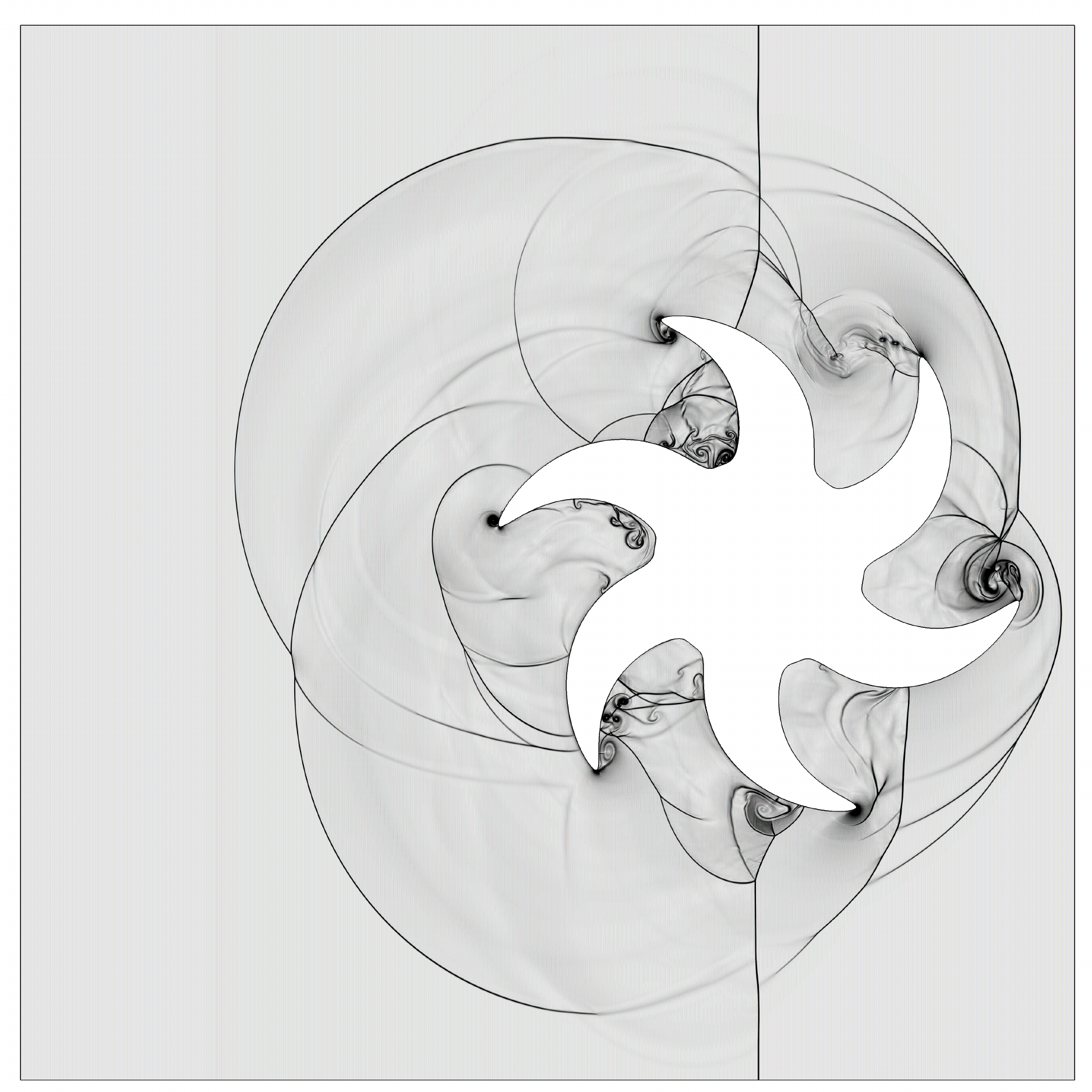}{\figWidth}};
  \draw (0.00,15.0) node[draw,fill=white,anchor=north west,xshift=+2pt,yshift=+1pt] {\scriptsize $t=.25$};
  \draw (8.25,15.0) node[draw,fill=white,anchor=north west,xshift=+2pt,yshift=+1pt] {\scriptsize $t=.50$};
  \draw (0.00, 7.5) node[draw,fill=white,anchor=north west,xshift=+2pt,yshift=-1pt] {\scriptsize $t=.75$};
  \draw (8.25, 7.5) node[draw,fill=white,anchor=north west,xshift=+2pt,yshift=-1pt] {\scriptsize $t=1.0$};
\end{tikzpicture}
\end{center}
  \caption{Shock impacting a {\em starfish} of zero mass. Schlieren images of the solution at times $t=0.25$, $0.5$, $0.75$ and $1.0$ computed
    with on grid $\Gcb^{(32\times4)}$. The boundaries of the AMR refinement grids are shown at $t=0.75$. }
  \label{fig:starFishEvolve}
\end{figure}
}

Fig.~\ref{fig:starFishEvolve} shows the evolution of the solution at four times. 
A complicated set of reflected and transmitted shocks form as the lead shock impacts the different arms of the body. These impacts cause
the body to rapidly accelerate at different times.
Numerous Mach stems, shock triple points and roll-ups can be identified.
Fig.~\ref{fig:shockImpactingStarFishRigidBodyInfo} shows the time history of the rigid body dynamics, comparing 
results from a coarse and fine grid. The coarse and fine grid results are in excellent agreement. 
As seen from Fig.~\ref{fig:shockImpactingStarFishRigidBodyInfo}, the initial impact of the shock
on the body causes it to accelerate to the right and slightly downward. The body begins to rotate in the counter-clockwise
direction although at later times it rotates in the clockwise direction. The time histories of rigid body velocity and
angular velocity undergo rapid changes changes at various times (e.g. when the lead shock impacts an arm).

{
\begin{figure}[hbt]
\newcommand{\figWidth}{7.5cm}
\newcommand{\trimfig}[2]{\trimFig{#1}{#2}{0.0}{0.}{.0}{.0}}
\begin{center}
\begin{tikzpicture}[scale=1]
  \useasboundingbox (0,.75) rectangle (8.,6.5);  
  \draw ( 0, 0) node[anchor=south west,xshift=-4pt,yshift=+0pt] {\trimfig{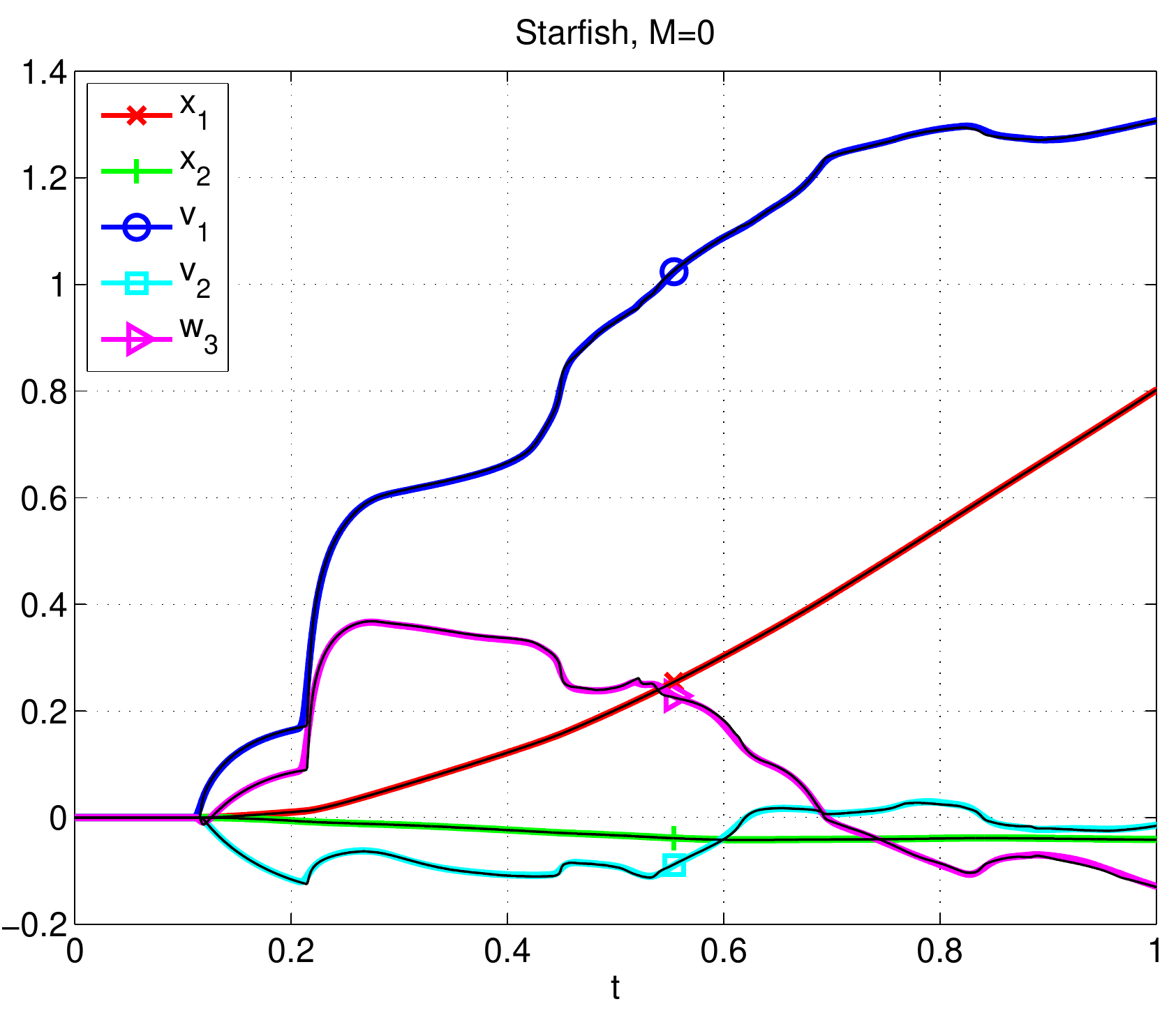}{\figWidth}};
\end{tikzpicture}
\end{center}
  \caption{Shock impacting a {\em starfish} of zero mass: time histories of the center of mass, $(x_1,x_2)$, the velocity of the center of
  mass, $(v_1,v_2)$ and the angular velocity $w_3$. The colored lines are results from the coarse
  grid $\Gcb^{(16)}$ while  the black lines are results using the finer grid $\Gcb^{(16\times4)}$.}   
  \label{fig:shockImpactingStarFishRigidBodyInfo}
\end{figure}
}

{
\newcommand{\figWidth}{5.5cm}
\newcommand{\trimfig}[2]{\trimPlotb{#1}{#2}{.2}{.1}{.1}{.1}}
\begin{figure}[hbt]
\begin{center}
\begin{tikzpicture}[scale=1]
  \useasboundingbox (0,.5) rectangle (16.5,6.5);  
  \draw ( 0.0, 0) node[anchor=south west,xshift=-4pt,yshift=+0pt] {\trimfig{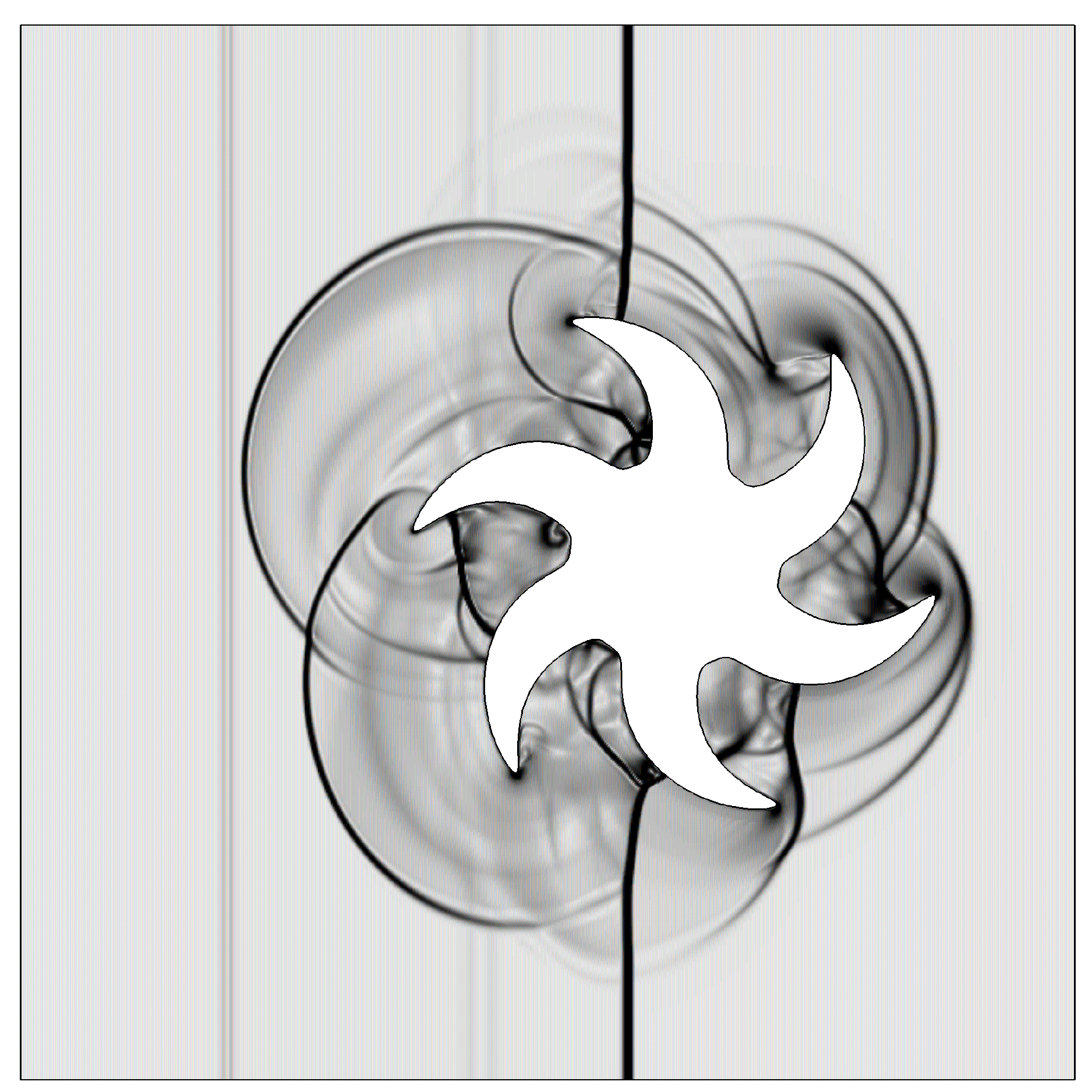}{\figWidth}};
  \draw ( 5.7, 0) node[anchor=south west,xshift=-4pt,yshift=+0pt] {\trimfig{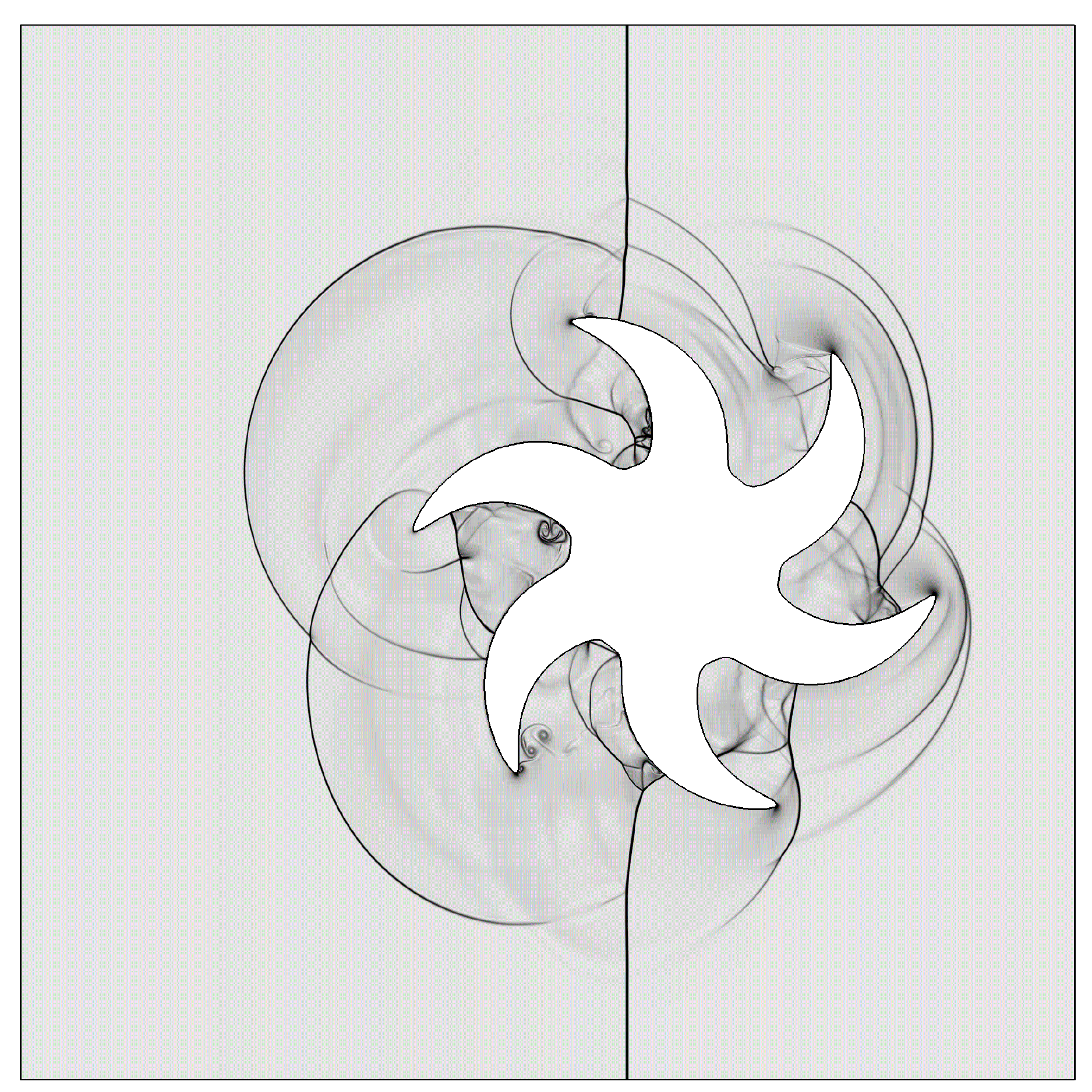}{\figWidth}};
  \draw (11.4, 0) node[anchor=south west,xshift=-4pt,yshift=+0pt] {\trimfig{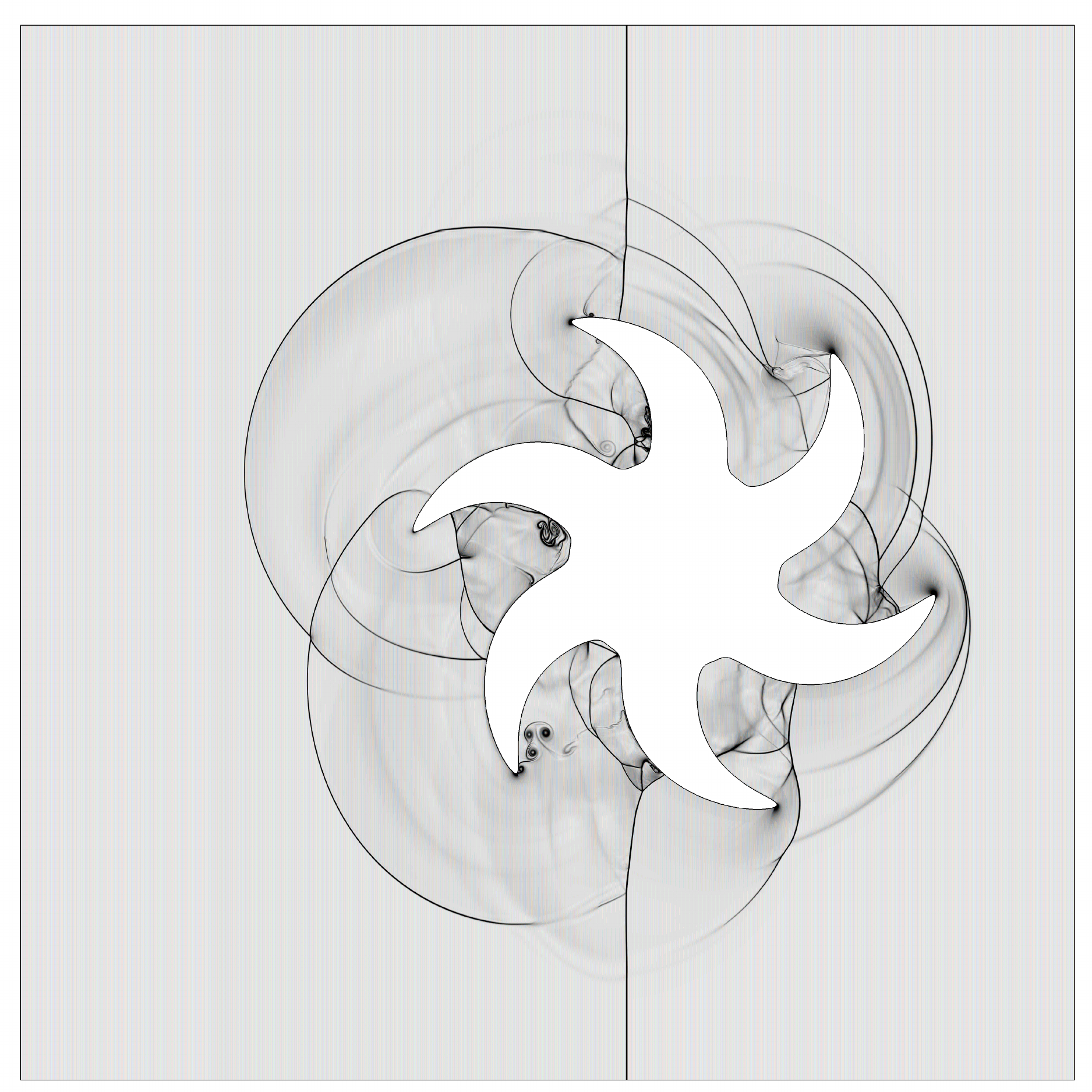}{\figWidth}};
  \draw (0.00,6.5) node[draw,fill=white,anchor=north west,xshift=+2pt,yshift=-3pt] {\scriptsize Coarse grid};
  \draw (5.70,6.5) node[draw,fill=white,anchor=north west,xshift=+2pt,yshift=-3pt] {\scriptsize Medium grid};
  \draw (11.4,6.5) node[draw,fill=white,anchor=north west,xshift=+2pt,yshift=-3pt] {\scriptsize Fine grid};
%
\end{tikzpicture}
\end{center}
  \caption{Shock impacting a {\em starfish} of zero mass. A comparison of schlieren images of the solution at $t=0.75$ 
      computed on the coarse grid $\Gcb^{(16)}$ (left),
      medium grid $\Gcb^{(16\times4)}$ (middle) and fine grid $\Gcb^{(32\times4)}$ (right).}
  \label{fig:starFishCompare}
\end{figure}
}

In Fig.~\ref{fig:starFishCompare} we compare the
schlieren images of the solution at $t=0.75$ from grids of different
resolutions. As for the ellipse, these result show good agreement in the basic structure of the
solution, with additional fine scale features appearing as the grid is
refined.

\section{Conclusions}\label{sec:conclusions}

We have presented a stable partitioned scheme for the coupling of light rigid
bodies with inviscid compressible fluids. This new {\em added-mass} scheme, derived from 
an analysis of a fluid/rigid-body Riemann problem, defines the 
force on the rigid body as a sum of the usual fluid surface forces (due to the
pressure) plus an impedance weighted difference of the local fluid velocity and
the velocity of the rigid body. The form of the added-mass terms are thus elucidated. 
The scheme uses a standard upwind scheme and explicit time-stepping for the fluid and a
diagonally implicit Runge-Kutta scheme for the small system of ordinary differential equations governing
the motion of the rigid body. 
The scheme was analyzed in one-dimension and shown to be well defined 
and stable, with a {\em large} time-step,
even when the mass of the rigid body, $\mrb$, goes to zero. In contrast the traditional
FSI coupling algorithm has a time-step restriction that goes to zero as $\mrb$
approaches zero. Both a first-order accurate upwind scheme and a second-order accurate
Law-Wendroff scheme were analyzed. Numerical computations in one-dimension confirmed the
results of the theory and showed that the scheme was well behaved and accurate even when $\mrb=0$. 

The added-mass scheme was then extended to 
multiple space dimensions. The result was an {\em added-mass} form of the 
Newton-Euler equations for rigid-body motion that included four added-mass tensors.
The added-mass tensors couple the translational and angular velocities
of the body and are defined in terms surface integrals involving the fluid
impedance.
Numerical results in two-dimensions were presented for both smooth and discontinuous problems.
Second-order convergence was demonstrated
using a smoothly receding piston problem with known exact solution, and a smoothly accelerated
ellipse. The robustness of the scheme was demonstrated for the difficult 
cases of a shock impacting an ellipse and {\em starfish shaped} body, both with zero mass and zero moment of inertia.
The solutions to these problems were computed on a sequence of
grids of increasing resolution (utilizing adaptive mesh refinement), 
with the results on the different grids comparing favorably. 
There are a number of avenues open for follow-on work including the
extension of the current scheme to three dimensions and
viscous flows. In addition, we are currently
investigating approaches for coupling incompressible flow
with light bodies (both rigid and deformable).

\appendix
\section{An analytic solution for the one-dimensional FSI model problem} \label{sec:testProblem}

Consider the one-dimensional FSI problem illustrated in Fig.~\ref{fig:modelFig}
consisting of a rigid body embedded between two (linearized) fluid domains. The governing
equations are defined in Section~\ref{sec:analysis}. 
To simplify the presentation, we take the width of the rigid body to be zero, $\width=0$.
The solution for $\width>0$ follows easily from the solution for $\width=0$.
Let the displacements in the left and right domains be defined by
$\UL(x,t)=\int_0^t\uL(x,\tau)\, d\tau$ and $\UR(x,t)=\int_0^t\uR(x,\tau)\,
d\tau$ respectively, and let the rigid body position be given by
$U_{b}(t)$. The second-order wave equations
\begin{align}
  \partial_{tt}\UL(x,t)-\cL^2\partial_{xx}\UL(x,t) & =0, \quad \hbox{for $x<0$}, \\
  \partial_{tt}\UR(x,t)-\cR^2\partial_{xx}\UR(x,t) & =0, \quad \hbox{for $x>0$},
\end{align}
describe the evolution of $\UL$ and $\UR$.  The evolution of the rigid body
position is given by the rigid body equations of motion with the applied stress from the fluid
determining the force on the body
\begin{equation}
  \mass\partial_{tt}{U}_{b}(t)=\rho_Rc_R^2\partial_x\UR(0,t)-\rhoL\cL^2\partial_x\UL(0,t).
\end{equation}
Assume given initial conditions 
\begin{align}
  \UL(x,0) &=U_0(x),  & \VL(x,0) & =V_0(x), \quad \hbox{for $x<0$} &\\
  \UR(x,0)& =U_0(x), & \VR(x,0) & =V_0(x), \quad \hbox{for $x>0$} &\\
  U_b(0) &= U_0(0), & \partial_t{U}_b(0) &=V_0(0).
\end{align}
The exact solution for $x < 0$ can be written in terms of the d'Alembert solution as
\begin{equation}
  \UL(x,t)=f_L(x-c_Lt)+g_L(x+c_Lt) , 
  \label{eq:leftSol}
\end{equation}
where
\begin{align}
   f_L(\xi) = & \frac{1}{2}\left(U_0(\xi)-\frac{1}{c_L}\int_{0}^{\xi}V_0(s)\, ds\right) , \medskip\\
   g_L(\xi) = & \left\{ \begin{array}{cc}
     \frac{1}{2}\left(U_0(\xi)+\frac{1}{c_L}\int_{0}^{\xi}V_0(s)\, ds\right), \qquad & \hbox{for $\xi<0$}, \\
     U_{b}\left(\frac{\xi}{c_L}\right)-f_L(-\xi), \qquad & \hbox{for $\xi\ge 0$.}
   \end{array}\right.
 \end{align}
Likewise for $x>0$, the solution can be written 
\begin{equation}
  \UR(x,t)=f_R(x-c_Rt)+g_R(x+c_Rt)
  \label{eq:rightSol}
\end{equation}
where
\begin{align}
   f_R(\xi) = & \left\{ \begin{array}{cc} 
     \frac{1}{2}\left(U_0(\xi)-\frac{1}{c_R}\int_{0}^{\xi}V_0(s)\, ds\right), \qquad & \hbox{ for $\xi>0$ }, \\
     U_{b}\left(\frac{-\xi}{c_R}\right)-g_R(-\xi), \qquad & \hbox{for $\xi\le 0$},
   \end{array}\right.  \medskip\\
   g_R(\xi) = & \frac{1}{2}\left(U_0(\xi)+\frac{1}{c_R}\int_{0}^{\xi}V_0(s)\, ds\right).
 \end{align}
For $x\le-\cL t$ or $x \ge \cR t$ the solution is given by the usual d'Alembert solution for the Cauchy problem,
\[
  U(x,t) = \frac{1}{2}\big(U_0(x-c t)+U_0(x+c t)\big) + \frac{1}{2 c}\int_{x-c t}^{x+c t} V_0(s)\, ds,
\]
where $c=\cL$ or $c=\cR$ for the left and right domains, respectively. 
For $-\cL t < x < \cR t$, the left and right solutions are coupled to the
rigid body. For this case, the unknown interface position $U_{b}$ is
found as the solution to the linear ODE
\begin{equation}
  \mass\partial_{tt}{U}_{b}(t)+\left(z_R+z_L\right)\partial_t{U}_{b}(t) = g(t)
  \label{eq:coupledODE}
\end{equation}
where $g(t) = \rhoR\cR^2\partial_x{U}_0(\cR t)-\rhoL\cL^2\partial_x{U}_0(-\cL t)+z_R V_0(\cR t)+z_L V_0(-\cL t)$. Solutions to the corresponding homogeneous ODE $ \mass\partial_{tt}\eta(t)+\left(z_R+z_L\right)\partial_t\eta(t) = 0$ are easily found as
\begin{align*}
  \eta_1(t) = & e^{-t (z_R+z_L)/\mass}, \quad\text{and~~}   \eta_2(t) = 1.
\end{align*}
The method of variation of parameters can be used to derive an exact solution to (\ref{eq:coupledODE})
by looking for a solution of the form
\begin{equation}
  U_{b}(t) = k_1(t)\eta_1(t)+k_2(t)\eta_2(t).
  \label{eq:ODESol}
\end{equation}
The unknown functions $k_1(t)$ and $k_2(t)$ are found to be
\begin{align}
  k_1(t)= & -\int \frac{\eta_1(t)g(t)}{W[\eta_1,\eta_2](t)} \, dt + \hbox{const}, \\
  k_2(t) = & \int \frac{\eta_2(t)g(t)}{W[\eta_1,\eta_2](t)} \, dt + \hbox{const},
\end{align}
where $W[\eta_1,\eta_2](t)$ is the Wronskian of the homogeneous solutions. The integration constants are determined by the initial conditions. For a more detailed discussion on solution methods for (\ref{eq:coupledODE}) refer to~\cite{boyce97} for example.

A specific solution of the form (\ref{eq:ODESol}) is determined by specifying initial conditions $U_0(x)$ and $V_0(x)$. We illustrate with an example where an initial Gaussian pulse (of velocity and stress) moves from left to right and interacts with the rigid body and fluid domains as time progresses. Let the initial conditions be given as
\begin{align}
  U_0(x) = & -\frac{1}{4}\frac{\sqrt{\pi}\erf\left(\beta(x-x_0)\right)}{\beta},  \label{eq:test_u_ic}\\
  V_0(x) = & \frac{c_L}{2}\exp\left(-\beta^2(x-x_0)^2\right). \label{eq:test_v_ic}
\end{align}
Here $\beta>0$ and $x_0<0$ are parameters used to define the center and width of the initial pulse. Also notice that we envision the pulse to originate entirely in the left domain which is the reason for the appearance of $c_L$ in the initial condition definition.
The velocity of the rigid body can be found as
\begin{align}
  \dot{U}_{b}(t) =  & 
    \frac{z_R(\cR-\cL)\sqrt{\pi}}{4\cR\beta\mass}
    \exp\left(\frac{(z_L+z_R)(z_L+z_R-4\beta^2\mass\cR(\cR t-x_0))}{4\cR^2\mass^2\beta^2}\right) \nonumber \\
    & \qquad\qquad\left[
      \erf\left(\frac{z_L+z_R-2\cR\beta^2\mass (\cR t-x_0)}{2\cR\mass\beta}\right)-
      \erf\left(\frac{z_L+z_R +2\cR\beta^2\mass x_0}{2\cR\mass\beta}\right)
    \right]- \nonumber \\
    &\frac{z_L \sqrt{\pi}}{2\beta\mass}
    \exp\left(\frac{(z_L+z_R)(z_L+z_R -4\cL\beta^2\mass (\cL t+x_0))}{4\cL^2\mass^2\beta^2}\right) \nonumber \\
    & \qquad\qquad\left[
      \erf\left(\frac{z_L+z_R-2\cL\beta^2\mass (\cL t+x_0)}{2\cL\mass\beta}\right)-
      \erf\left(\frac{z_L+z_R-2\cL\beta^2\mass x_0}{2\cL\mass\beta}\right)
    \right]+ \nonumber \\
  & \frac{\cL}{2}\exp\left(-\beta^2 x_0^2-\frac{(z_L+z_R)t}{\mass}\right).
  \label{eq:exactSol}
\end{align}
Analytic expressions for the position and acceleration of the body 
are determined from~\eqref{eq:exactSol}, by integration and differentiation, respectively.
Note that (\ref{eq:exactSol}) is not easily evaluated numerically with standard math libraries as $\mass \to
0$. For the small mass case, (\ref{eq:exactSol}) can be evaluated using asymptotic expansions of the 
error functions as their arguments approach plus or minus infinity. The desired level of accuracy can be obtained
by appropriately truncating the resulting series expansion. In practice, we find that for $\mass \lessapprox 0.1$
such a procedure should be used.

\section{Examples of added mass matrices for constant fluid impedance} \label{sec:addedMassMatrices}

In this section we illustrate the form of the added mass matrices defined by ~\eqref{eq:AddedMassMatrixI}-\eqref{eq:AddedMassMatrixII},
for some common body shapes when the fluid impedance $z_f$ is taken to be constant.
We denote the entries of $\Avv$ by $\avv_{ij}$, the entries of $\Avw$ by $\avw_{ij}$ and the entries
of $\Aww$ by $\aww_{ij}$.
Note that in actual FSI simulations the coefficients of the added mass matrices (which depend
on a variable impedance) are computed for general bodies using numerical quadrature and so there is no need to determine these coefficients analytically. The results in this appendix are therefore presented for two reasons. The first is to help readers understand the nature of the added mass matrices for some simple bodies. The second is because the added mass matrices for simple bodies are useful in their own right, for example when treating flows with infinitesimally small embedded particles~\cite{parmar08}.

\subsection{Added-mass matrices for an ellipse} \label{sec:addedMassMatricesEllipse}

Consider a two dimensional ellipse with semi-axes of length $a$ and $b$ and center of mass $\xv_0=\zerov$. 
A point on the ellipse is $\xv(\theta)=[a\cos(\theta),~ b\sin(\theta),~ 0]^T$. The tangent to this point is
\[
  \frac{\xv_\theta}{\|{\xv_\theta} \|}
          =[-a\sin(\theta),~ b\cos(\theta),~ 0]^T/\sqrt{a^2\sin^2(\theta)+b^2\cos^2(\theta)} .
\]
Thus 
\begin{align*}
 \nv&=[b\cos(\theta),~ a\sin(\theta),~ 0]^T/\sqrt{a^2\sin^2(\theta)+b^2\cos^2(\theta)}, \\
 \yv&=[a\cos(\theta),~ b\sin\theta,~0]^T,
\end{align*}
and 
\begin{align*}
   Y\nv = [0,~0,~ (a^2-b^2)\cos(\theta)\sin(\theta)]^T/\sqrt{a^2\sin^2(\theta)+b^2\cos^2(\theta)}.
\end{align*}
Thus (leaving out some zero rows and columns which do not apply in two-dimensions)
\begin{align}
  \nv \nv^T =  
    \frac{1}{a^2\sin^2(\theta)+b^2\cos^2(\theta)}
    \begin{bmatrix}
      b^2\cos^2(\theta)              & a b \cos(\theta)\sin(\theta) \\
      a b \cos(\theta)\sin(\theta)   & a^2\sin^2(\theta)  
  \end{bmatrix}, \\ 
  Y\nv (Y\nv)^T = \frac{1}{a^2\sin^2(\theta)+b^2\cos^2(\theta)}
     \begin{bmatrix}
      0 & 0 & 0 \\
      0 & 0 & 0 \\
      0 & 0 & (a^2-b^2)^2\cos^2(\theta)\sin^2(\theta)
  \end{bmatrix} . 
\end{align}
The increment in arclength is $ds=\sqrt{d\xv\cdot d\xv} = \sqrt{a^2\sin^2(\theta)+ b^2\cos^2(\theta)}~ d\theta$.
Thus
\begin{alignat}{3}
    \Avv &= 
                   \begin{bmatrix}
                     \avv_{11} & \avv_{12}    \\
                     \avv_{21} & \avv_{22}
                    \end{bmatrix}    
    =    \int_0^{2\pi} \frac{z_f}{\sqrt{a^2\sin^2(\theta)+b^2\cos^2(\theta)}}
                   \begin{bmatrix}
                     b^2\cos^2(\theta) & a b \cos(\theta)\sin(\theta)   \\
                     a b \cos(\theta)\sin(\theta)    & a^2\sin^2(\theta)
                    \end{bmatrix} \,d\theta ,
\end{alignat}
\begin{alignat}{3}
    \Aww &= 
               \begin{bmatrix}
                     0    & 0   & 0  \\
                     0    & 0   & 0 \\
                     0    & 0   & \aww_{33}
                    \end{bmatrix}
     =  \int_0^{2\pi}  \frac{z_f}{\sqrt{a^2\sin^2(\theta)+b^2\cos^2(\theta)}} \begin{bmatrix}
                     0    & 0   & 0  \\
                     0    & 0   & 0 \\
                     0    & 0   & (a^2-b^2)^2\cos^2(\theta)\sin^2(\theta)
                    \end{bmatrix}\,d\theta ,
\end{alignat}
and 
\begin{alignat}{3}
    \Avw &= (\Awv)^T = 
               \begin{bmatrix}
                     0    & 0   & \avw_{13}  \\
                     0    & 0   & \avw_{23} \\
                     0    & 0   & 0
               \end{bmatrix}
     =  \int_0^{2\pi}  \frac{z_f}{\sqrt{a^2\sin^2(\theta)+b^2\cos^2(\theta)}} \begin{bmatrix}
                     0    & 0   & b (a^2-b^2)\cos^2(\theta)\sin(\theta)  \\
                     0    & 0   & a (a^2-b^2)\cos(\theta)\sin^2(\theta) \\
                     0    & 0   & 0 
                    \end{bmatrix}\,d\theta .
\end{alignat}
Values for $\avv_{11}$, $\avv_{22}$, and $\aww_{33}$, 
(which can be written in terms of elliptic integrals) for some ratios of $b$ to $a$ are 
given in Figure~\ref{fig:ellipseIntegralCoefficients}. 
The values for $\avv_{12}$, $\avw_{13}$ and $\avw_{23}$
are zero for
uniform $z_f$ (but can be non-zero when $z_f$ varies). 
Note that for the case of a circle, $a=b$, $\avv_{11} = \avv_{22} = (z_f/a) \pi a^2$ where $\pi a^2$ is the area of the circle.
Compare this result to that for the sphere in Section~\ref{sec:addedMassMatricesEllipsoid}.
\begin{figure}[hbt]
\begin{center}
\begin{tabular}{|l|c|c|c|c|} \hline 
               & $b=a$      & $b=a/2$    &  $b=a/10$  &  $b=a/100$   \\ \hline
$\avv_{11}$    & $\pi z_f a$   & $1.26 z_f a$   &  $.108 z_f a$  &  $.0020 z_f a$   \\ \hline
$\avv_{22}$    & $\pi z_f a$   & $3.58 z_f a$   &  $3.96 z_f a$  &  $3.99  z_f a$   \\ \hline
$\aww_{33}$    & $0$        & $.581 z_f a^3$   &  $1.27 z_f a^3$  &  $1.33  z_f a^3$   \\ \hline
\end{tabular}
\caption{Components of the added-mass matrices for an ellipse for various values of $b/a$ with constant $z_f$. 
Values for $b/a\ne1$ are approximate.}
\label{fig:ellipseIntegralCoefficients}
\end{center}
\end{figure}

\subsection{Added-mass matrices for an ellipsoid} \label{sec:addedMassMatricesEllipsoid}

We consider an ellipsoid with semi-axes of length $a$, $b$ and $c$ and center of mass at $\xv_0=\zerov$.
A point on the surface of the ellipsoid is given by
\begin{align*}
   \xv(\theta,\phi) &= [ a \sin(\phi) \cos(\theta), b \sin(\phi) \sin(\theta), c\cos(\phi) ]^T, \quad \phi\in[0,\pi], \quad
           \theta\in[0,2\pi].
\end{align*}
From this formula it is straightforward to determine $\nv$ and $Y\nv$ in the formulae for the added mass matrices.
For a sphere of radius $a$, i.e. $a=b=c$, we get ($4\pi/3\approx 4.18879$) 
\begin{alignat}{3}
    \Avv &= 
          z_f a^2 \begin{bmatrix}
                   4\pi/3 & 0        & 0  \\
                     0    & 4\pi/3   & 0 \\
                     0    & 0        & 4\pi/3 
                    \end{bmatrix}, 
     \quad 
    \Avw &= 
          z_f a^3 \begin{bmatrix}
                     0    & 0   & 0  \\
                     0    & 0   & 0 \\
                     0    & 0   & 0 
                    \end{bmatrix}, 
     \quad 
    \Aww &= 
          z_f a^4 \begin{bmatrix}
                     0    & 0   & 0  \\
                     0    & 0   & 0 \\
                     0    & 0   & 0 
                    \end{bmatrix} . 
\end{alignat}
Recall that the volume of the sphere is $V=4\pi a^3/3$ so that $\avv_{ii}= (z_f/a) V$, $i=1,2,3$. 
The rotational added-mass entries $\aww_{ii}$, $i=1,2,3$ are zero since a rotating sphere exerts no
force on the adjacent (inviscid) fluid. 

For $b=a$, $c=2 a$, we can compute the added-mass matrix entries approximately by quadrature giving the 
values
\begin{alignat}{3}
    \Avv &= 
          z_f a^2 \begin{bmatrix}
                     9.254     & 0   & 0  \\
                     0    & 9.254    & 0 \\
                     0    & 0   & 2.971 
                    \end{bmatrix}, 
     \quad 
    \Avw &= 
          z_f a^3 \begin{bmatrix}
                     0    & 0   & 0  \\
                     0    & 0   & 0 \\
                     0    & 0   & 0 
                    \end{bmatrix}, 
     \quad 
    \Aww &= 
          z_f a^4 \begin{bmatrix}
                     4.712    & 0   & 0  \\
                     0    & 4.712    & 0 \\
                     0    & 0   & 0 
                    \end{bmatrix} . 
\end{alignat}
This ellipsoid is longest along the $z$-axis and has circular cross-sections for $z$ constant. The values
of $\avv_{11}$ and $\avv_{22}$ are larger than $\avv_{33}$ which indicates that the added mass is
larger for translational motions in the $x$- or $y$-directions compared to the $z$-direction. In other
words is takes more force to move the ellipsoid in the $x-$ or $y-$directions compared to the $z$-direction. This is consistent
with the shape of the ellipsoid which is longest along the $z$-axis and thus has a greater effective
cross-sectional area when viewed from the $x-$ or $y-$directions.

For $b=2 a$, $c=3 a$, the added-mass matrix entries are given approximately by
\begin{alignat}{3}
    \Avv &= 
          z_f a^2 \begin{bmatrix}
                     32.307    & 0   & 0  \\
                     0    & 11.023   & 0 \\
                     0    & 0   & 5.552 
                    \end{bmatrix}, 
     \quad 
    \Avw &= 
          z_f a^3 \begin{bmatrix}
                     0    & 0   & 0  \\
                     0    & 0   & 0 \\
                     0    & 0   & 0 
                    \end{bmatrix}, 
     \quad 
    \Aww &= 
          z_f a^4 \begin{bmatrix}
                     6.840    & 0   & 0  \\
                     0    & 53.511   & 0 \\
                     0    & 0   & 15.963
                    \end{bmatrix} . 
\end{alignat}
In this case, the translational added mass $\avv_{11}$ is largest, consistent with the effective cross-sectional area being largest
when the ellipsoid is viewed in the $x$-direction. In other words it takes more force to move the ellipsoid in the $x$-direction,
compared to the other directions. 

\subsection{Added-mass matrices for a rectangle}
  The added-mass matrices for bodies with piecewise constant surface normals are easily computed. Consider the rectangular body of length $l_x$, height $l_y$, and center of mass $\xv=0$ given by $\Rc = \{ (x,y) \,|\, -l_x/2 \le x \le l_x/2, ~ -l_y/2 \le y \le l_y/2 \}$. The added-mass matrices for this case are
\begin{alignat}{3}
    \Avv &= 
          z_f  \begin{bmatrix}
          2l_y & 0 & 0 \\
          0 & 2l_x & 0\\
          0 & 0 & 0
       \end{bmatrix}, 
     \quad 
    \Avw &= 
          z_f \begin{bmatrix}
                     0    & 0   & 0  \\
                     0    & 0   & 0 \\
                     0    & 0   & 0 
                    \end{bmatrix}, 
     \quad 
    \Aww &= 
          z_f \begin{bmatrix}
                     0    & 0   & 0  \\
                     0    & 0   & 0 \\
                     0    & 0   & \frac{1}{6}\left(l_x^3+l_y^3\right)
                    \end{bmatrix} . 
\end{alignat}

\subsection{Added-mass matrices for a rectangular prism}
  Finally, consider the rectangular prism with dimensions $l_x$, $l_y$, $l_z$, and center of mass $\xv=0$ given by ${\mathcal P} = \{ (x,y,z) \,|\, -l_x/2 \le x \le l_x/2, ~ -l_y/2 \le y \le l_y/2, ~ -l_z/2 \le z \le l_z/2, \}$. The added-mass matrices are
\begin{alignat}{3}
    \Avv &= 
          z_f  \begin{bmatrix}
          2l_yl_z & 0 & 0 \\
          0 & 2l_xl_z & 0\\
          0 & 0 & 2l_xl_y
       \end{bmatrix}, 
     \quad 
    \Avw &= 
          z_f \begin{bmatrix}
                     0    & 0   & 0  \\
                     0    & 0   & 0 \\
                     0    & 0   & 0 
                    \end{bmatrix}, 
     \quad 
    \Aww &= 
          z_f \begin{bmatrix}
                     \frac{l_x}{6}\left(l_y^3+l_z^3\right)    & 0   & 0  \\
                     0    & \frac{l_y}{6}\left(l_x^3+l_z^3\right)   & 0 \\
                     0    & 0   & \frac{l_z}{6}\left(l_x^3+l_y^3\right)
                    \end{bmatrix} . 
\end{alignat}

\bibliographystyle{elsart-num}
\bibliography{journal-ISI,jwb}

\end{document}